\documentclass{amsart}

\usepackage{amsmath,amssymb,mathrsfs,amsthm,ascmac,subfiles,mathtools}
\usepackage{color} %textcolor
\usepackage{comment} % 複数行のコメントアウト
\usepackage{multirow} % for \multirow in table.
\mathtoolsset{showonlyrefs=true}
\allowdisplaybreaks[4] % 長いalign等の改ページを許す
\numberwithin{equation}{section}
\theoremstyle{break}
\newtheorem{Def}{Definition}[section]
\newtheorem{Thm}{Theorem}[section]
\newtheorem{Prop}{Proposition}[section]
\newtheorem{Lem}{Lemma}[section]

\newtheorem{Rem}{Remark}[section]

\newcommand{\fb}[1]{{\bf #1}}

\title[$\mathscr{R}$-boundedness for Stokes resolvent problem]
{A solution formula and the $\mathscr{R}$-boundedness for the generalized Stokes resolvent problem 
in an infinite layer with Neumann boundary condition}
\author{Kenta Oishi \\ Nagoya University}
\address{ 
Kenta Oishi \endgraf 
Graduate School of Mathematics \endgraf 
Nagoya University \endgraf 
Furocho, Chikusaku \endgraf 
Nagoya 464-8602 \endgraf 
Japan} 
\email{m16011b@math.nagoya-u.ac.jp}
\date{}

\begin{document}
\maketitle

%%%%%%%%%%%%%%%%%%%%%%%%%%%%%%%%%%%%%%%%%%%%%%%%%%%%%%%%%
%%%%%%%%%%%%%%%%%%%%%%%%%%%%%%%%%%%%%%%%%%%%%%%%%%%%%%%%%
%	ABSTRACT
%%%%%%%%%%%%%%%%%%%%%%%%%%%%%%%%%%%%%%%%%%%%%%%%%%%%%%%%%
%%%%%%%%%%%%%%%%%%%%%%%%%%%%%%%%%%%%%%%%%%%%%%%%%%%%%%%%%
\begin{abstract}
We consider the generalized Stokes resolvent problem in an infinite layer with Neumann boundary conditions. 
This problem arises from a free boundary problem 
describing the motion of incompressible viscous one-phase fluid flow without surface tension
in an infinite layer bounded both from above and from below by free surfaces. 
We derive a new exact solution formula to the generalized Stokes resolvent problem 
and prove the $\mathscr{R}$-boundedness of the solution operator families 
with resolvent parameter $\lambda$ varying in a sector $\Sigma_{\varepsilon,\gamma_0}$ 
for any $\gamma_0>0$ and $0<\varepsilon<\pi/2$, 
where $\Sigma_{\varepsilon,\gamma_0}
=\{ \lambda\in\mathbb{C}\setminus\{0\} \mid |\arg\lambda|\leq\pi-\varepsilon, \ |\lambda|>\gamma_0 \}$. 
As applications, 
we obtain the maximal $L_p$-$L_q$ regularity for the nonstationary Stokes problem
and then establish the well-posedness locally in time of the nonlinear free boundary problem mentioned above 
in $L_p$-$L_q$ setting. 
We make full use of the solution formula to take $\gamma_0>0$ arbitrarily, 
while in general domains we only know the $\mathscr{R}$-boundedness for $\gamma_0\gg1$ from the result by Shibata. 
As compared with the case of Neumann-Dirichlet boundary condition studied by Saito, 
analysis is even harder on account of higher singularity of the symbols in the solution formula. 
\end{abstract}

%%%%%%%%%%%%%%%%%%%%%%%%%%%%%%%%%%%%%%%%%%%%%%%%%%%%%%%%%%
%%%%%%%%%%%%%%%%%%%%%%%%%%%%%%%%%%%%%%%%%%%%%%%%%%%%%%%%%%
%%														%%
%%					Sec.1 Introduction						%%
%%														%%
%%%%%%%%%%%%%%%%%%%%%%%%%%%%%%%%%%%%%%%%%%%%%%%%%%%%%%%%%%
%%%%%%%%%%%%%%%%%%%%%%%%%%%%%%%%%%%%%%%%%%%%%%%%%%%%%%%%%%

\section{Introduction} \label{sec:intro}

%dom.
%%%%%%%%%%%%%%%%%%%%%%%%%%%%%%%%%%%%%%%%%%%%
%Stokes resolvent prob (& Stokes eq). / explanation of prob. / known results 
%%%%%%%%%%%%%%%%%%%%%%%%%%%%%%%%%%%%%%%%%%%%
We consider the generalized Stokes resolvent problem and the nonstationary Stokes problem in an infinite layer $\Omega$ 
with Neumann boundary conditions: 
\begin{gather} \label{SR} %\tag{$\text{SR}_{\lambda,\delta}$}
\left\{
\begin{array}{ll}
\lambda \fb{u}- \text{Div}_{}\,{\fb{S}}(\fb{u},\theta) = \fb{f}, \quad \text{div}_{}\,{\fb{u}}= g & \text{ in } \Omega, \\
\fb{S}(\fb{u},\theta)\nu= \fb{h}& \text{ on } \partial\Omega,
\end{array} 
\right. 
\\ \label{S} 
\left\{
\begin{array}{ll}
\partial_t{\fb{U}}- \text{Div}_{}\,{\fb{S}}(\fb{U},\Theta) = \fb{F}, \quad \text{div}_{}\,{\fb{U}}=G & \text{in } \Omega\times(0,\infty), \\
\fb{S}(\fb{U},\Theta)\nu= \fb{H}& \text{on } \partial\Omega\times(0,\infty), \\
\fb{U}|_{t=0}=0 & \text{in } \Omega, 
\end{array}
\right. 
\end{gather}
where the domain $\Omega$ is given by  
\begin{align*}
\Omega = \{ x = (x',x_N)^\mathsf{T} \in \mathbb{R}^N \mid x' = (x_1,\cdots,x_{N-1})^\mathsf{T} \in \mathbb{R}^{N-1}, \ 0<x_N<\delta \} \quad (\delta>0) 
\end{align*}
and $N\geq 2$. 
%%%%%%%%%%%%%%%%%%%%%%% explanation of prob
Here, by $\fb{u}=(u_1(x),\cdots,u_N(x))^\mathsf{T}$ and $\theta=\theta(x)$, 
we denote respectively unknown $N$-component velocity vector and scalar pressure, 
while vector fields $\fb{f}=(f_1(x),\cdots,f_N(x))^\mathsf{T}$, $\fb{h}=(h_1(x),\cdots,h_N(x))^\mathsf{T}$ 
and scalar function $g=g(x)$ are prescribed. 
Concerning $\fb{U}=\fb{U}(x,t)$, $\Theta=\Theta(x,t)$, $\fb{F}=\fb{F}(x,t)$, $G=G(x,t)$ and $\fb{H}=\fb{H}(x,t)$, 
it would be obvious what they are. 
The stress tensor $\fb{S}(\fb{u},\theta)$ is given by $\mu \fb{D}(\fb{u}) - \theta{\fb{I}}$, 
where $\mu$ is a positive constant which denotes the viscosity coefficient, 
$\fb{I}$ is the $N\times N$ identity matrix, 
and $\fb{D}(\fb{u})$ is the doubled deformation tensor 
whose $(j,k)$ component is $\fb{D}_{jk}(\fb{u})=\partial_ku_j+\partial_ju_k$ with $\partial_j=\partial/\partial x_j$. 
We denote by $\nu= (\nu_1(x),\cdots,\nu_N(x))^\mathsf{T}$ the unit outer normal vector to $\partial\Omega$. 
Finally, we set $\text{div}_{}\,{\fb{u}}= \sum_{j=1}^N \partial_ju_j$ and, 
given matrix field $\fb{M}$ with $(j,k)$ component $M_{jk}$, $\text{Div}_{}\,{\fb{M}}$ is defined by the $N$-component vector
whose $j$-th component is $\sum_{k=1}^N \partial_k M_{jk}$. 
In this paper, 
we derive a new exact solution formula to the generalized Stokes resolvent problem 
and prove the $\mathscr{R}$-boundedness of the solution operator families. 
As applications, we obtain the maximal $L_p$-$L_q$ regularity for \eqref{S} 
and then establish the well-posedness locally in time 
for the free boundary problem \eqref{NS} below for the Navier-Stokes equations. 

%%%%%%%%%%%%%%%%%%%%%%% known results of SR, S
Problems \eqref{SR} and \eqref{S} in the layer have been studied in the case of other boundary conditions. 
For the problem in which the boundary condition on the lower boundary is replaced by the Dirichlet one  
\begin{align}
\fb{u}=0 \text{ on } \Gamma_0 = \{ x=(x',x_N) \in \mathbb{R}^N \mid x_N=0 \}, 
\end{align}
Abe \cite{Abe04MMAS} provided a solution formula of \eqref{SR} 
and proved the resolvent estimate for $\lambda \in \Sigma_{\varepsilon,\gamma_0}$ 
with any $\gamma_0>0$ and $0<\varepsilon<\pi/2$, where
\begin{align} \label{eq:resoset}
\Sigma_{\varepsilon,\gamma_0}= \{ \lambda \in \mathbb{C}\setminus\{0\} \mid |\arg\lambda| \leq \pi-\varepsilon, \ |\lambda| > \gamma_0 \}. 
\end{align}
Abels showed the resolvent estimate for $\lambda \in \Sigma_{\varepsilon,0} \cup \{0\}$ with any $0<\varepsilon<\pi/2$ in \cite{Abel06MN} 
and obtained the same result for asymptotically flat layers in \cite{Abel05MFM1}. 
Moreover, he showed that the Stokes operator admits a bounded $H_\infty$-calculus in \cite{Abel05MFM2} 
and proved the maximal regularity for $\fb{f}\in L_q(0,\infty;L_q(\Omega))$ with $3/2<q<\infty$ %from it 
in \cite{Abel05ADE}. 
In \cite{Sai15MMAS}, Saito provided a new solution formula to \eqref{SR} 
subject to Neumann-Dirichlet boundary condition mentioned above
and established the $\mathscr{R}$-boundedness of the solution operator families 
with resolvent parameter $\lambda\in{\Sigma_{\varepsilon,\gamma_0}}$ for any $\gamma_0>0$ and $0<\varepsilon<\pi/2$. 
He obtained the maximal $L_p$-$L_q$ regularity for $1<p,q<\infty$ as a corollary. 
In the case of Dirichlet boundary conditions on both upper and lower boundaries, 
Abe and Shibata \cite{AbeShiba03MSJ,AbeShiba03MFM}, Abels and Wiegner \cite{AbelWie05DIE}, 
Abe and Yamazaki \cite{AbeYama10MFM} and von Below and Bolkart \cite{BelBol17MN} 
derived solution formulas and obtained the resolvent estimates. 
In addition, in \cite{Abel02EE}, 
Abels proved the existence of bounded imaginary powers of the Stokes operator 
and, as a consequence, the maximal $L_p$-$L_q$ regularity. 
However, analysis of \eqref{SR} and \eqref{S} with Neumann boundary conditions on both sides of the boundary $\partial\Omega$ 
is less developed. 
We do not know any solution formula of the generalized Stokes resolvent problem \eqref{SR}.  

On the other hand, in general domains $\Omega$ subject to Dirichlet boundary condition on $\Gamma_b\subset{\partial}\Omega$ 
and Neumann boundary condition on $\Gamma=\partial\Omega\setminus\Gamma_b$, 
Shibata \cite{Shiba13MFM} showed the resolvent estimate under the assumption: 
the unique existence of solution $\theta\in \mathcal{W}^1_q(\Omega)$ to the weak Dirichlet-Neumann problem 
\begin{align} \label{WDNP}
(\nabla{\theta},\nabla\varphi)_\Omega = (\fb{f}_0,\nabla\varphi)_\Omega \text{ for any } \varphi \in \mathcal{W}^1_q(\Omega)
\end{align}
for any $\fb{f}_0\in{L_q(\Omega)}^N$. 
Here, $\mathcal{W}^1_q(\Omega)$ is any closed subspace of $\widehat{W}^1_{q,\Gamma}(\Omega)$ containing $W^1_{q,\Gamma}(\Omega)$, 
where $\widehat{W}^1_{q,\Gamma}(\Omega)=\{ \theta\in{L_{q,\text{loc}}(\overline{\Omega})} \mid \nabla{\theta}\in{L_q(\Omega)}^N, \ \theta|_\Gamma=0 \}$
and $W^1_{q,\Gamma}(\Omega)=\{ \theta\in{W^1_q(\Omega)} \mid \theta|_\Gamma=0 \}$. 
In \cite{Shiba14DIE}, this result was further developed by showing the $\mathscr{R}$-boundedness and the maximal $L_p$-$L_q$ regularity. 
As for the case of Neumann boundary conditions on both sides of the boundary $\partial\Omega$ of the layer $\Omega$, 
the unique solvability of \eqref{WDNP} with $\mathcal{W}^1_q(\Omega)=W^1_{q,0}(\Omega)$ 
was proved by Simader and Ziegler \cite{SimZie98QM}. 
As a consequence, the $\mathscr{R}$-boundedness of solution operator families with $\lambda\in{\Sigma_{\varepsilon,\gamma_0}}$ of \eqref{SR} 
is available provided $\gamma_0>0$ is large enough, 
however, it is better to develop the theory with $\gamma_0>0$ arbitrarily small for the layer  
without relying on the framework of \cite{Shiba13MFM,Shiba14DIE}. 
We note that, as a corollary of a main result of the present paper, 
the unique solvability of \eqref{WDNP} 
%with $\mathcal{W}^1_q(\Omega)=\widehat{W}^1_{q,0}(\Omega)$ 
with $\mathcal{W}^1_q(\Omega)=\widehat{W}^1_{q,0}(\Omega)$, which coincides with $W^1_{q,0}(\Omega)$ for the layer, 
is recovered thanks to observation by Shibata \cite[Remark 1.7]{Shiba13MFM}. 
%and that the space $\widehat{W}^1_{q,0}(\Omega)$ coincides with $W^1_{q,0}(\Omega)$ for the layer.  

%abstract of main result(R-boundedness in layer:No2HD) 
The purpose of this paper is to provide a new solution formula of \eqref{SR} 
and to show the $\mathscr{R}$-boundedness of the solution operator families 
with the resolvent parameter $\lambda\in \Sigma_{\varepsilon,\gamma_0}$ for arbitrary $\gamma_0>0$ and $0<\varepsilon<\pi/2$. 
From this, we obtain the resolvent estimates for the same $\lambda$. 
%Cor(maximal reg/LWP/reso esti/WDNP)
As an application, we prove 
the maximal $L_p$-$L_q$ regularity for \eqref{S} with $1<p,q<\infty$ 
by using the operator-valued Fourier multiplier theorem due to Weis \cite[Theorem 3.4]{Wei01MA}. %LWPは後回しにした
%emphasizing point(gamma0>0) 
In order to derive a new exact solution formula to \eqref{SR}, 
we apply the Fourier transform with respect to tangential variable $x'\in{\mathbb{R}^{N-1}}$. 
We can take any $\gamma_0>0$, see \eqref{eq:resoset}, by taking advantage of use of this formula, 
while $\gamma_0$ was taken large enough in \cite{Shiba14DIE} for general domains. 
It is desirable to obtain the result above for $\lambda\in{\Sigma_{\varepsilon,0}}$, 
which is the crucial step toward analysis of large time behavior of solutions to \eqref{S}. 
This issue will be discussed elsewhere. 
The condition $\gamma_0>0$ is needed for several steps, 
but the most essential part is the estimate of the determinant, $\det{\fb{L}}$, in the solution formula, see Lemma \ref{Lem:EstidetL}. 
The solution formula itself is valid for $\lambda=0$ as well, 
however, when $\lambda=0$, we see that $|\det{\fb{L}}|^{-1}$ is too singular at the origin $\xi'=0$ in the Fourier side, 
see Remark \ref{Rem:gamma0} for details. 
	%strategy(estimate of symbol) 
Our strategy follows \cite{ShibaShimi12MSJ} and \cite{Sai15MMAS}, 
and it is based on a series of technical lemmas which guarantees the $\mathscr{R}$-boundedness from pointwise estimates of the kernel
by considering the solution formula as a singular integral operator. 
%relation with known results(high singu/avoid P) 
As compared with Neumann-Dirichlet boundary condition, 
however, the kernel decays only with the order $|x'|^{-(N-1)}$ unlike the case of $\mathbb{R}^N$ 
and, in the Fourier side, our symbol in the solution formula possesses higher singularity at the origin $\xi'=0$, 
see Remark \ref{Rem:highsingular}. 
This is because the same boundary conditions on both sides of $\partial\Omega$ imply 
that $\det{\fb{L}}$ involves similar rows and, therefore, is degenerate for $\xi'\to0$. 
Moreover, the estimate of $|\det{\fb{L}}|^{-1}$ is not homogeneous 
in the sense that the rate for $\xi'\to0$ is different from the one for $|\xi'|\to\infty$. 
In order to overcome those difficulties, 
we carry out a cut-off procedure and then employ analysis developed by Saito \cite[Lemma 5.5]{Sai15MMAS}, 
see Lemma \ref{Lem:RbddVS}. 
To deal with rather singular symbols mentioned above, 
as in \cite{Sai15MMAS}, after fixing the normal variable, we regard the solution formula as a singular integral on $\mathbb{R}^{N-1}$  
and deduce an estimate uniformly with respect to the normal variable. 
Then we handle the integral in the normal direction with the aid of the boundedness of the domain in that direction. 

%%%%%%%%%%%%%%%%%%%%%%% original prob. / explanation of prob. / known results
Problems \eqref{SR} and \eqref{S} arise from a free boundary problem for the Navier-Stokes equations 
describing the motion of incompressible viscous fluid flow with free surfaces without taking account of surface tension: 
\begin{align} \label{NS} 
\left\{
\begin{array}{ll}
\partial_t{\fb{v}}+ (\fb{v}\cdot\nabla)\fb{v}- \text{Div}_{}\,{\fb{S}}(\fb{v},\pi) = 0, \quad \text{div}_{}\,{\fb{v}}=0 & \text{in } \Omega(t), \ 0<t<T, \\
\fb{S}(\fb{v},\pi)\nu_t = 0, \quad \fb{v}\cdot{\nu}_t=V_n & \text{on } \partial\Omega(t), \ 0<t<T, \\  %\Gamma(t), \ 0<t<T, \\
%\fb{v}= 0 & \text{on } \Gamma_0\times(0,T), \\
\fb{v}|_{t=0}=\fb{v}_0 & \text{in } \Omega. %_0. 
\end{array}
\right.
\end{align}
Here, 
$\fb{v}=(v_1(x,t),\cdots,v_N(x,t))^\mathsf{T}$ is the velocity vector field and $\pi=\pi(x,t)$ is the pressure 
in a time-dependent domain $\Omega(t)$, 
while $\fb{v}_0 = (v_{01}(x),\cdots,v_{0N}(x))^\mathsf{T}$ is a given initial velocity in the initial domain $\Omega$. 
By $V_n$ we denote the velocity of the evolution of $\partial\Omega(t)$ 
and $\nu_t$ stands for the unit outer normal to $\partial\Omega(t)$. 
The novelty of the problem \eqref{NS} is that both upper and lower boundaries are free ones to be determined. 
If we replace the boundary condition on the lower boundary by the non-slip one, 
the problem is considered in an asymptotic layer 
\begin{align}
\Omega(t) = \{ x = (x',x_N) \in \mathbb{R}^N \mid -b(x')<x_N<\eta(x',t) \} 
\end{align}
with a fixed bottom, where the only free boundary is the upper surface. 
In this setting, there are extensive studies and they will be mentioned in order. 
In the $L_2$-framework, 
the existence of solutions locally in time was established by Beale \cite{Bea81CPAM} without surface tension, 
whereas by Allain \cite{All85AFST,All87AMO} and by Tani \cite{Tani96ARMA} with surface tension. 
%%%%%%%%%%%%%%%%%%%%%%%%%%%%%%%%%%%% 
Here, in the latter case, the boundary condition on the free surface should be  
\begin{align}
\fb{S}(\fb{v},\pi)\nu_t = \sigma{\mathcal{H}}{\nu}_t, \quad \fb{v}\cdot{\nu}_t=V_n,  
\end{align}
with $\mathcal{H}$ being the doubled mean curvature of $\partial\Omega(t)$, 
where $\sigma>0$ is a constant representing the coefficient of surface tension. 
%%%%%%%%%%%%%%%%%%%%%%%%%%%%%%%%%%%% 
Also, in \cite{Tera85Hi} and \cite{Tera92Ky}, 
Teramoto showed the local well-posedness especially in an inclined layer without and with surface tension, respectively. 
The global well-posedness was proved by Beale \cite{Bea84ARMA} with surface tension 
and by Tani and Tanaka \cite{TaniTana95ARMA} with and without surface tension. 
Beale and Nishida \cite{BeaNishi85NMS} and Hataya and Kawashima \cite{HataKawa09NA} 
studied the large time behavior of the solution obtained in the study of Beale \cite{Bea84ARMA}. 
Hataya \cite{Hata09Ky} established the existence of a global solution with some decay properties 
under the periodic boundary condition in the horizontal direction.    
In the $L_q$-framework, Abels \cite{Abel05ADE} obtained the local well-posedness without surface tension. 
In the $L_p$-$L_q$ setting, Saito \cite{Sai18DE} proved the well-posedness globally in time without surface tension. 
Once we have the maximal $L_p$-$L_q$ regularity for \eqref{S} 
as a corollary of the $\mathscr{R}$-boundedness of the solution operator families of \eqref{SR}, 
a fix-point argument as in \cite{Shiba15DE,Sai18DE,ShibaShimi07DIE} leads to 
the local well-posedness of \eqref{NS} in $L_p$-in-time and $L_q$-in-space setting for $2<p<\infty$ and $N<q<\infty$, 
although this result is covered by the theory for general domains under the unique solvability of \eqref{WDNP} 
which was established by Shibata without and with taking account of surface tension 
in \cite{Shiba15DE} and in \cite{Shiba16MFDPF,Shiba16DCDS}, respectively.  
Nevertheless, for completeness of the present paper, 
we give the statement of a unique existence of a local solution to \eqref{NS} without proof, see Theorem \ref{Thm:LWP}.  

%(notations and def. for main result / main result)
%sections
This paper is organized as follows: 
in Section 2, we state our main results on the $\mathscr{R}$-boundedness for \eqref{SR}, 
maximal regularity for \eqref{S} and local solvability of \eqref{NS}. 
In Section 3, we reduce \eqref{SR} to the problem in which the data are prescribed only on the boundary. 
The solution formula for the latter problem derived in Section 4 is a novelty of the present paper. 
In Section 5, we introduce some technical lemmas to prove the $\mathscr{R}$-boundedness from estimates of symbols 
and, finally, Section 6 is devoted to completion of the proof.

%%%%%%%%%%%%%%%%%%%%%%%%%%%%%%%%%%%%%%%%%%%%%%%%%%%%%%%%%%
%%%%%%%%%%%%%%%%%%%%%%%%%%%%%%%%%%%%%%%%%%%%%%%%%%%%%%%%%%
%%														%%
%%					Sec.2 Main Results						%%
%%														%%
%%%%%%%%%%%%%%%%%%%%%%%%%%%%%%%%%%%%%%%%%%%%%%%%%%%%%%%%%%
%%%%%%%%%%%%%%%%%%%%%%%%%%%%%%%%%%%%%%%%%%%%%%%%%%%%%%%%%%

\section{Main Results} \label{sec:mainthm} %Main theorem

In this section, we introduce notation and several function spaces, 
which are used throughout this paper, and then provide our main results. 

%%%%%%%%%%%%%%%%%%%%%%%%%%%%%%%%%%%%%%%%%%%%%%%%%%%
% Notation
%%%%%%%%%%%%%%%%%%%%%%%%%%%%%%%%%%%%%%%%%%%%%%%%%%%
% sets, multi-index, derivative, bounded operators
We denote the sets of all natural numbers, real numbers, and complex numbers 
by $\mathbb{N}$, $\mathbb{R}$ and $\mathbb{C}$, respectively, and set $\mathbb{N}_0=\mathbb{N}\cup\{0\}$. 
For multi-index $\alpha=(\alpha_1,\cdots,\alpha_N) \in \mathbb{N}_0^N$, we write
$|\alpha| = \alpha_1+\cdots+\alpha_N$ 
and $\partial_x^\alpha = \partial_1^{\alpha_1} \cdots \partial_N^{\alpha_N}$, 
where $x=(x_1,\cdots,x_N)$ and $\partial_J=\partial/\partial x_J$ for $1\leq J\leq N$. 
Given scalar function $f$ and $N$-vector function $\fb{g}= (g_1,\cdots,g_N)$, we set
\begin{align}
\nabla f &= (\partial_1f,\cdots,\partial_N f)^\mathsf{T},& 
\nabla \fb{g}&= (\partial_jg_i)_{1\leq i,j\leq N} \\
\nabla^2f &= (\partial^\alpha f \mid |\alpha|=2),& 
\qquad \nabla^2\fb{g}&= (\partial^\alpha g_J \mid |\alpha|=2, \ J=1,\cdots,N).
\end{align}
Let $D=\mathbb{R}^n,\,\mathbb{R}^n_+$, $\Omega$. 
For scalar functions $f$, $g$ and $N$-vector functions $\fb{f}$, $\fb{g}$, we write
\begin{align}
(f,g)_D &= \int_D f(x)g(x) \,dx, & (\fb{f},\fb{g})_D &= \int_D \fb{f}(x)\cdot{\fb{g}}(x) \,dx, \\
\langle f,g\rangle_{\partial D} &= \int_{\partial D} f(x)g(x) \,d\sigma, & 
\langle{\fb{f}},\fb{g}\rangle_{\partial D} &= \int_{\partial D} \fb{f}(x)\cdot{\fb{g}}(x) \,d\sigma, 
\end{align}
where $\fb{a}\cdot\fb{b} = \sum_{i=1}^N a_ib_i$ 
for $\fb{a}=(a_1,\cdots,a_N)^\mathsf{T}$ and $\fb{b}=(b_1,\cdots,b_N)^\mathsf{T}$, 
and $d\sigma$ is the surface element of $\partial D$. 
% functional space, inner product, Fourier transformation
Let $1\leq q\leq \infty$ and $m \in \mathbb{N}_0$. 
The symbols $L_q(D)$ and $W^m_q(D)$ stand for the Lebesgue space and Sobolev space 
with their associated norms $\|\cdot\|_{L_q(\Omega)}$ and $\|\cdot\|_{W^m_q(\Omega)}$, respectively. 
Here, $W^0_q(D)=L_q(D)$. 
We denote by $C_0^\infty(D)$ the set of all $C^\infty(D)$ functions whose supports are compact and contained in $D$. 
We define
\begin{alignat*}{1}
\widehat{W}^1_q(D) &= \{ \theta\in L_{q,\text{loc}}(\overline{D}
) \mid \nabla \theta\in L_q(\Omega)^N \}, \\
W^1_{q,0}(D) &= \{\varphi\in{W^1_q(D)} \mid \varphi|_{\partial\Omega}=0\}, \\
\dot{W}^1_{q,0}(D) &= \{\varphi\in{\widehat{W}^1_q(D)} \mid \varphi|_{\partial\Omega}=0\}, \\
\widehat{W}^1_{q,0}(D) &= \text{ the closure of } W^1_{q,0}(D) \text{ in } \dot{W}^1_{q,0}(D)
 \text{ with respect to } \|\nabla\cdot\|_{L_q(\Omega)}, \\
\widehat{W}^{-1}_q(D) &= \text{ the dual space of } \widehat{W}^1_{q',0}(D), 
\end{alignat*}
where $q'$ is the dual exponent of $q$, that is, $1/q+1/q'=1$. 

\begin{Rem}
Similarly to \cite[Theorem A.3 (4)]{Shiba13MFM}, 
we have $\dot{W}^1_{q,0}(\Omega) = \widehat{W}^1_{q,0}(\Omega)$ for the layer $\Omega$ under consideration. 
Indeed, given $\varphi\in\dot{W}^1_{q,0}(\Omega)$, 
we consider an extension $E_0\varphi\in\dot{W}^1_{q,0}(\mathbb{R}^N_+)$ of $\varphi$ by setting 0 outside $\Omega$: 
\begin{align} \label{eq:DefE0}
E_0\varphi(x) = \left\{ 
\begin{array}{ll}
\varphi(x), &\text{when } x\in\Omega, \\
0,&\text{when } x\not\in\Omega, 
\end{array}
\right. 
\end{align} 
so that the question is reduced to the case of the half space. 
Then we obtain the desired property by \cite[Lemma A.1]{Shiba13MFM} and \cite[Lemma A.2]{Shiba13MFM}. 
\end{Rem} 

Given functions $f=f(x)$ and $g=g(\xi)$ on $\mathbb{R}^N$, 
the Fourier transform and its inverse transform are denoted by $\mathscr{F}_x$ and $\mathscr{F}_\xi^{-1}$, that is, 
\begin{align}
\mathscr{F}_x [f](\xi) &= \int_{\mathbb{R}^N} e^{-ix\cdot \xi} f(x) \, dx, 
& \mathscr{F}_\xi^{-1} [g](x) &= \frac{1}{(2\pi)^N} \int_{\mathbb{R}^N} e^{ix\cdot \xi} g(\xi) \, d\xi. 
\end{align}
Also, the partial Fourier transform with respect to $x'=(x_1,\cdots,x_{N-1})$ and its inverse transform are defined by
\begin{align}
\mathscr{F}_{x'} [f](\xi',x_N) &= \widehat{f}(\xi',x_N) = \int_{\mathbb{R}^{N-1}} e^{-ix'\cdot \xi'} f(x',x_N) \, dx', \\ 
\mathscr{F}_{\xi'}^{-1} [g](x',x_N) &= \mathscr{F}_{\xi'}^{-1}[g(\xi',x_N)](x') = \frac{1}{(2\pi)^{N-1}} \int_{\mathbb{R}^{N-1}} e^{ix'\cdot \xi'} g(\xi',x_N) \, d\xi'.
\end{align}

% constants
Given Banach spaces $X$ and $Y$, 
we denote by $\mathcal{L}(X,Y)$ the Banach space of all bounded linear operators from $X$ to $Y$, 
and we write $\mathcal{L}(X) = \mathcal{L}(X,X)$ to shorten notation. 
For $d \in \mathbb{N}$ and Banach space $X$ with norm $\|\cdot\|_X$, 
$d$-product of $X$ is denoted by $X^d$, 
nevertheless we continue to write $\|\cdot\|_X$ instead of $\|\cdot\|_{X^d}$ for abbreviation. 
We often write $\gamma=\text{Re\,}\lambda$ and $\tau=\text{Im\,}\lambda$ for the resolvent parameter $\lambda$ in the sector 
\begin{align} 
\Sigma_{\varepsilon,\gamma_0}= \{ \lambda \in \mathbb{C}\setminus\{0\} \mid |\arg\lambda| \leq \pi-\varepsilon, \ |\lambda| > \gamma_0 \} \quad 
(0<\varepsilon<\pi/2, \ \gamma_0\geq0). 
\end{align}
Finally, the letter $C$ denotes generic constants and $C_{a,b,\cdots}$ stands for constants depending on the quantities $a,b,\cdots$. 
Both constants $C$ and $C_{a,b,\cdots}$ may change from line to line. 

%%%%%%%%%%%%%%%%%%%%%%%%%%%%%%%%%%%%%%%%%%%%%%%%%%%%%%%%%%%%%%%%%%%%%%%%
% R-bdd
%%%%%%%%%%%%%%%%%%%%%%%%%%%%%%%%%%%%%%%%%%%%%%%%%%%%%%%%%%%%%%%%%%%%%%%%
The notion of the $\mathscr{R}$-boundedness of operator families is defined as follows. 
\begin{Def} \label{Def:Rbdd}
Let $X$ and $Y$ be Banach spaces. 
The family $\mathcal{T} \subset \mathcal{L}(X,Y)$ is called $\mathscr{R}$-bounded 
if there exist $1 \leq p < \infty$ and $C>0$ such that 
for any $m\in{\mathbb{N}}$, $\{T_j\}_{j=1}^m \subset \mathcal{T}$, $\{x_j\}_{j=1}^m \subset X$ 
and sequence $\{r_j\}_{j=1}^m$ of independent, symmetric and $\{\pm1\}$-valued random variables on $(0,1)$, 
there holds the estimate 
\begin{align} \label{eq:DefRbdd}
\bigg\{\int_0^1 \Big\| \sum_{j=1}^m r_j(u)T_jx_j \Big\|_Y^p \, du\bigg\}^{1/p} 
\leq C \bigg\{\int_0^1 \Big\| \sum_{j=1}^m r_j(u)x_j \Big\|_X^p du\bigg\}^{1/p}. 
\end{align}
The infimum of such $C$ is called $\mathscr{R}$-bound 
and denoted by $\mathscr{R}_{\mathcal{L}(X,Y)}(\mathcal{T})$, or $\mathscr{R}_{\mathcal{L}(X)}(\mathcal{T})$ if $X=Y$. 
\end{Def}

%%%%%%%%%%%%%%%%%%%%%%%%%%%%%%% Remark
\begin{Rem} \label{Rem:Rbdd}
\begin{itemize}
\item[a)]
It is well known that \eqref{eq:DefRbdd} holds for any $p \in [1,\infty)$ if it holds for some $p \in [1,\infty)$ 
by Kahane's inequality (cf. \cite[Theorem 3.11]{HytNeeVerWei16EMG}). %OK? %check the bibliography
\item[b)]
We get the uniform boundedness of $\mathcal{T}$ from the $\mathscr{R}$-boundedness of $\mathcal{T}$ 
by letting $m=1$ in \eqref{eq:DefRbdd}. 
\end{itemize}
\end{Rem}

%%%%%%%%%%%%%%%%%%%%%%%%%%%%%%%%%%%%%%%%%%%%%%%%%%%%%%%%%%%%%%%%%%%%%%%%
% Main Results
%%%%%%%%%%%%%%%%%%%%%%%%%%%%%%%%%%%%%%%%%%%%%%%%%%%%%%%%%%%%%%%%%%%%%%%%
We are in a position to state our main result 
on the $\mathscr{R}$-boundedness of the solution operator families of the resolvent problem \eqref{SR}. 
Set  
\begin{align}
X_q(\Omega) = \{ &(F_1,F_2,F_3,F_4,F_5,F_6) \mid \\ 
&F_1, F_4, F_5 \in L_q(\Omega)^N, F_2 \in \widehat{W}^{-1}_q(\Omega), F_3 \in L_q(\Omega), F_6 \in L_q(\Omega)^{N^2} \}. 
\end{align}

\begin{Thm} \label{Thm:Rbdd}
For all $\lambda \in \mathbb{C}\setminus(-\infty,0]$, 
there exist operators $\mathcal{U}(\lambda)=(\mathcal{U}_1(\lambda),\cdots,\mathcal{U}_N(\lambda))$ and $\mathcal{P}(\lambda)$ 
(precisely, they are given by \eqref{eq:defUP}) 
satisfying $\mathcal{U}(\lambda)\in{\mathcal{L}}(X_q(\Omega), W^2_q(\Omega)^N)$ and 
$\mathcal{P}(\lambda)\in{\mathcal{L}}(X_q(\Omega), \widehat{W}^1_q(\Omega))$ for $1<q<\infty$ 
such that the following assertions hold: 
\begin{itemize}
\item[a)] 
For any $1<q<\infty$, $\lambda \in \mathbb{C}\setminus(-\infty,0]$ and the data
\begin{align}
(\fb{f},g,\fb{h}) \in L_q(\Omega)^N \times (\widehat{W}^{-1}_q(\Omega) \cap W^1_q(\Omega)) \times W^1_q(\Omega)^N, 
\end{align} 
the pair $(\fb{u},\theta) \in W^2_q(\Omega)^N \times \widehat{W}^1_q(\Omega)$ given by  
\begin{align}
(\fb{u},\theta)&=(\mathcal{U}(\lambda), \mathcal{P}(\lambda))(\fb{f}, \lambda g,\lambda^{1/2}g, \nabla g, \lambda^{1/2}{\fb{h}}, \nabla \fb{h}) 
\end{align}
is a solution of \eqref{SR}. 
Additionally, the solution of \eqref{SR} is unique, 
that is, if $(\fb{u},\theta)\in{W^2_q(\Omega)}^N\times{\widehat{W}^1_q(\Omega)}$ satisfies \eqref{SR} with $(\fb{f},g,\fb{h})=(0,0,0)$, 
then $\fb{u}=0$ and $\theta=0$ a.e.  
\item[b)] 
For any $1<q<\infty$, $0<\varepsilon<\pi/2$, $\gamma_0>0$, $\ell=0,1$ and $1\leq m,n,J\leq N$, there hold 
\begin{align} \label{eq:Rbddsolop}
\mathscr{R}_{\mathcal{L}(X_q(\Omega),L_q(\Omega))}(\{ (\tau{\partial}_\tau)^\ell \lambda \mathcal{U}_J(\lambda) \mid \lambda\in{\Sigma_{\varepsilon,\gamma_0}} \})&\leq C_{N,q,\varepsilon,\gamma_0,\mu,\delta}, \\ 
\mathscr{R}_{\mathcal{L}(X_q(\Omega),L_q(\Omega))}(\{ (\tau{\partial}_\tau)^\ell \gamma \mathcal{U}_J(\lambda) \mid \lambda\in{\Sigma_{\varepsilon,\gamma_0}} \})&\leq C_{N,q,\varepsilon,\gamma_0,\mu,\delta}, \\ 
\mathscr{R}_{\mathcal{L}(X_q(\Omega),L_q(\Omega))}(\{ (\tau{\partial}_\tau)^\ell \lambda^{1/2} \partial_m \mathcal{U}_J(\lambda) \mid \lambda\in{\Sigma_{\varepsilon,\gamma_0}} \})&\leq C_{N,q,\varepsilon,\gamma_0,\mu,\delta}, \\ 
\mathscr{R}_{\mathcal{L}(X_q(\Omega),L_q(\Omega))}(\{ (\tau{\partial}_\tau)^\ell \partial_m{\partial}_n \mathcal{U}_J(\lambda) \mid \lambda\in{\Sigma_{\varepsilon,\gamma_0}} \})&\leq C_{N,q,\varepsilon,\gamma_0,\mu,\delta}, \\
\mathscr{R}_{\mathcal{L}(X_q(\Omega),L_q(\Omega))}(\{ (\tau{\partial}_\tau)^\ell \partial_m \mathcal{P}(\lambda) \mid \lambda\in{\Sigma_{\varepsilon,\gamma_0}} \})&\leq C_{N,q,\varepsilon,\gamma_0,\mu,\delta}, 
\end{align} 
where $\Sigma_{\varepsilon,\gamma_0}$ is given by \eqref{eq:resoset} and $\lambda=\gamma+i\tau$. 
\end{itemize}
\end{Thm}

By this theorem and Remark \ref{Rem:Rbdd} b), we immediately obtain the resolvent estimates: 
\begin{Prop} \label{Prop:RE}
Let $1<q<\infty$, $0<\varepsilon<\pi/2$ and $\gamma_0>0$. 
There exists a constant $C=C_{N,q,\varepsilon,\gamma_0,\mu,\delta}>0$ 
such that for any $\lambda\in{\Sigma_{\varepsilon,\gamma_0}}$, $\fb{f}\in{L_q(\Omega)}^N$, $g\in{\widehat{W}^{-1}_q(\Omega)}\cap{W^1_q(\Omega)}$ and $\fb{h}\in{W^1_q(\Omega)}^N$, 
the solution $(\fb{u},p)$ given in Theorem \ref{Thm:Rbdd} satisfies 
\begin{align}
&\|(\lambda{\fb{u}},\lambda^{1/2}\nabla{\fb{u}},\nabla^2\fb{u},\nabla p)\|_{L_q(\Omega)} \\
&\leq C(\|\fb{f}\|_{L_q(\Omega)}+\|\lambda g\|_{\widehat{W}^{-1}_q(\Omega)}+\|(\lambda^{1/2}g,\nabla g)\|_{L_q(\Omega)}
+\|(\lambda^{1/2}{\fb{h}},\nabla{\fb{h}})\|_{L_q(\Omega)}). 
\end{align} 
\end{Prop}

%%%%%%%%%%%%%%%%%%%%%%%%%%%%%%%%%%%%%%%%%%%%%%%%%%%%%%%%%
% Unique Solvability of Weak Dirichlet-Neumann Problem
%%%%%%%%%%%%%%%%%%%%%%%%%%%%%%%%%%%%%%%%%%%%%%%%%%%%%%%%%
We obtain the unique solvablity of the weak Dirichlet problem \eqref{WDNP} 
with $\mathcal{W}^1_q(\Omega)=\widehat{W}^1_{q,0}(\Omega)$ 
by the preceding proposition and \cite[Remark 1.7]{Shiba13MFM}. 
Moreover, we deduce the estimate of the solution $\theta$ itself as well by the Poincar\'{e} inequality 
in a similar manner to \eqref{eq:0dpres} below.  
\begin{Prop} \label{Prop:USWDNP}
For all $\fb{f}_0\in{L_q(\Omega)}^N$, there exists a unique solution 
\begin{align}
\Pi_0{\fb{f}_0}:=\theta\in W^1_{q,0}(\Omega) 
\end{align}
of \eqref{WDNP} with $\mathcal{W}^1_q(\Omega)=\widehat{W}^1_{q,0}(\Omega)$ along with 
\begin{align}
\|\theta\|_{W^1_q(\Omega)} \leq C_{N,q,\delta} \|\fb{f}_0\|_{L_q(\Omega)}. 
\end{align}
\end{Prop}

%%%%%%%%%%%%%%%%%%%%%%%%%%%%%%%%%%%%%%%%%%%%%%%%%%%%%%%%%
% Remark 1.3 
%%%%%%%%%%%%%%%%%%%%%%%%%%%%%%%%%%%%%%%%%%%%%%%%%%%%%%%%%
\begin{Rem}
We can reconstruct a solution operator $\widetilde{\mathcal{P}}(\lambda)$ for the pressure $\theta$ to \eqref{SR} 
satisfying the $\mathscr{R}$-boundedness for 
$\widetilde{\mathcal{P}}(\lambda)$ as well as $\nabla\widetilde{\mathcal{P}}(\lambda)$ 
thanks to Proposition \ref{Prop:USWDNP}, 
where $\lambda$ varies in $\Sigma_{\varepsilon,\gamma_0}$ for any $0<\varepsilon<\pi/2$ and $\gamma_0>0$.  
In fact, $\theta$ satisfies the Poisson equation subject to the Dirichlet boundary condition, 
which can be reduced to the weak Dirichlet problem \eqref{WDNP} by subtracting the boundary data, 
as in \eqref{pres} and \eqref{pres2} below.  
%, $\theta-d$ obeys the homogeneous case
%, and so it solves the weak problem \eqref{WDNP} with data $\fb{f}_0=\fb{f}-\nabla d$ and $\mathcal{W}^1_q(\Omega)=\widehat{W}^1_q(\Omega)$. 
% whose weak problem is ...
Then we get 
\begin{align} \notag
\theta=\tilde\nu\cdot[\mu\fb{D}(\fb{u})\tilde\nu-\fb{h}]
+\Pi_0\big( -\lambda\fb{u} +\textup{Div}\,\big(\mu\fb{D}(\fb{u})\big) +\fb{f}
	-\nabla(\tilde\nu\cdot[\mu\fb{D}(\fb{u})\tilde\nu-\fb{h}]) \big)
\end{align}
and, accordingly, define $\widetilde{\mathcal{P}}(\lambda)$ by 
\begin{align*}
\widetilde{\mathcal{P}}(\lambda)\widetilde{F}
&= \tilde\nu\cdot [\mu\fb{D}(\mathcal{U}(\lambda)F)\tilde\nu-F_7] \\
&+\Pi_0\Big( 
-\lambda\mathcal{U}(\lambda)F +\textup{Div}\,\big(\mu\fb{D}(\mathcal{U}(\lambda)F)\big) +F_1 
-\nabla\big(\tilde\nu\cdot [\mu\fb{D}(\mathcal{U}(\lambda)F)\tilde\nu-F_7]\big) \Big)
%\widetilde{\mathcal{P}}(\lambda)\tilde{F}
%&= \tilde\nu\cdot [\mu\fb{D}(\mathcal{U}(\lambda)F)\tilde\nu-F_7] \\
%&\begin{aligned}
%+\Pi_0\big( 
%&-\lambda\mathcal{U}(\lambda)F +\textup{Div}\,\big(\mu\fb{D}(\mathcal{U}(\lambda)F)\big) +F_1 \\
%&-(\nabla\tilde\nu)\cdot M_{\lambda^{-1/2}}
%[\mu\lambda^{1/2}\fb{D}(\mathcal{U}(\lambda)F_\lambda)\tilde\nu-\lambda^{1/2}\fb{h}] \\
%&-\tilde\nu\cdot[\mu\nabla\fb{D}(\mathcal{U}(\lambda)F_\lambda)\tilde\nu 
%+M_{\lambda^{-1/2}}\lambda^{1/2}\fb{D}(\mathcal{U}(\lambda)F_\lambda)\nabla\tilde\nu-\nabla\fb{h}] \big)
%\end{aligned}
\end{align*}
for $\widetilde{F}:=(F,F_7)\in\widetilde{X}_q(\Omega):=X_q(\Omega)\times W^1_q(\Omega)^N$ so that
\begin{align}
\theta = \widetilde{\mathcal{P}}(\lambda)(\fb{f},\lambda g,\lambda^{1/2}g,\nabla g,\lambda^{1/2}\fb{h},\nabla \fb{h}, \fb{h}), 
\end{align}
where $\mathcal{U}(\lambda)$ and $\Pi_0$ are the solution operators 
defined in Theorem \ref{Thm:Rbdd} and Proposition \ref{Prop:USWDNP}, respectively, 
while the suitable extension $\tilde\nu$ of the unit outer normal $\nu$ is given by \eqref{eq:defMltn} below. 
It is possible to verify 
\begin{align}
\mathscr{R}_{\mathcal{L}(\widetilde{X}_q(\Omega),L_q(\Omega)^{N+1})}
(\{(\tau\partial_\tau)^\ell (\widetilde{\mathcal{P}}(\lambda), \nabla\widetilde{\mathcal{P}}(\lambda)) 
\mid \lambda\in\Sigma_{\varepsilon,\gamma_0}\}) \leq C_{N,q,\varepsilon,\gamma_0,\mu,\delta}
\end{align}
for any $0<\varepsilon<\pi/2$, $\gamma_0>0$ and $\ell=0,1$ 
on account of Theorem \ref{Thm:Rbdd}, Proposition \ref{Prop:USWDNP}, Lemma \ref{Lem:Rbddprop} below
together with the $\mathscr{R}$-boundedness for the multiplication operator $M_{\lambda^{-1/2}}$, 
see \eqref{eq:defMltn}--\eqref{eq:RbdMl}. 
\end{Rem}

%%%%%%%%%%%%%%%%%%%%%%%%%%%%%%%%%%%%%%%%%%%%%%%%%%%%%%%%%
%%%%%%%%%%%%%%%%%%%%%%%%%%%%%%%%%%%%%%%%%%%%%%%%%%%%%%%%%
% Corollary
%%%%%%%%%%%%%%%%%%%%%%%%%%%%%%%%%%%%%%%%%%%%%%%%%%%%%%%%%
%%%%%%%%%%%%%%%%%%%%%%%%%%%%%%%%%%%%%%%%%%%%%%%%%%%%%%%%%

%%%%%%%%%%%%%%%%%%%%%%%%%%%%%%%%%%%%%%%%%%%%%%%%%%%%%%%%%
% Notation for MR
%%%%%%%%%%%%%%%%%%%%%%%%%%%%%%%%%%%%%%%%%%%%%%%%%%%%%%%%%
As an important corollary to Theorem \ref{Thm:Rbdd}, 
we establish the maximal $L_p$-$L_q$ regularity for the nonstationary Stokes problem \eqref{S}  
with the aid of the operator-valued Fourier multiplier theorem due to Weis \cite{Wei01MA}. 
To describe the statement precisely, we introduce some function spaces. 
Let $I$ be an interval in $\mathbb{R}$, $X$ a Banach space, $1<p<\infty$, $m \in \mathbb{N}_0$ and $\gamma_0>0$. 
We denote the $X$-valued Bochner and Sobolev spaces by $L_p(I;X)$ and $W^m_p(I;X)$, respectively, and let
\begin{align*} 
W^m_{p,\gamma_0}(I;X) &= \{ f\in W^m_{p,\text{loc}}(\overline{I};X) \mid e^{-\gamma_0 t} f(t) \in W^m_p(I;X) \}, \\
W^m_{p,0,\gamma_0}(\mathbb{R};X) &= \{ f\in W^m_{p,\gamma_0}(\mathbb{R};X) \mid f(t)=0 \text{ for } t<0 \}, \\
L_{p,\gamma_0}(I;X) &= W^0_{p,\gamma_0}(I;X), \quad L_{p,0,\gamma_0}(\mathbb{R};X) = W^0_{p,0,\gamma_0}(\mathbb{R};X), \\
H^{1/2}_{p,0,\gamma_0}(\mathbb{R};X) &= \{ f\in L_{p,0,\gamma_0}(\mathbb{R};X) \mid 
e^{-\gamma t}\Lambda^{1/2}_\gamma f(t) \in L_p(\mathbb{R};X) \text{ for all }{\gamma}\geq\gamma_0 \}. 
\end{align*}
Here, we have set 
\begin{align}
[\Lambda_\gamma^s f](t) = \mathscr{L}_\gamma^{-1}[\lambda^s \mathscr{L}(\lambda)](t)
\end{align}
for $s>0$ and $\gamma>0$ where $\lambda = \gamma+i\tau \in \mathbb{C}$; 
for functions $f,g$ with $f(t)=0$ ($t<0$), 
we define the Laplace transform $\mathscr{L}[f]$ of $f$ and its inverse transform $\mathscr{L}_\gamma^{-1}[g]$ of $g$ by 
\begin{align} \label{eq:DefLaptrans}
\mathscr{L}(\lambda) &= \int_{-\infty}^\infty e^{-\lambda t} f(t) \,dt, & 
\mathscr{L}_\gamma^{-1}[g](t) &= \frac{1}{2\pi} e^{\gamma t} \int_{-\infty}^\infty e^{i\tau t} g(\gamma+i\tau) \,d\tau. 
\end{align}
Note that we have $W_{p,0,\gamma_0}^1(\mathbb{R};\widehat{W}^{-1}_q(\Omega)) \cap L_{p,0,\gamma_0}(\mathbb{R};W^1_q(\Omega)) \subset H_{p,0,\gamma_0}^{1/2}(\mathbb{R};L_q(\Omega))$; 
this is proved in \cite[Appendix A]{Shiba15DE} 
by applying Weis' operator-valued Fourier multiplier theorem after extending functions to be defined on $\mathbb{R}^N$.

%%%%%%%%%%%%%%%%%%%%%%%%%%%%%%%%%%%%%%%%%%%%%%%%%%%%%%%%%
% MR
%%%%%%%%%%%%%%%%%%%%%%%%%%%%%%%%%%%%%%%%%%%%%%%%%%%%%%%%%
The maximal $L_p$-$L_q$ regularity theorem for \eqref{S} is stated in the following theorem. 
This can be proved by the same way as in \cite[Theorem 2.1]{Sai15MMAS} and we may omit the proof. 
\begin{Thm} \label{Thm:MR}
Let $1<p,q<\infty$ and $\gamma_0>0$. Then, for every data
\begin{align*}
\fb{F}\in L_{p,0,\gamma_0}(\mathbb{R};L_q(\Omega)^N), \quad 
G \in W_{p,0,\gamma_0}^1(\mathbb{R};\widehat{W}^{-1}_q(\Omega)) \cap L_{p,0,\gamma_0}(\mathbb{R};W^1_q(\Omega)), \\ 
\fb{H}\in H_{p,0,\gamma_0}^{1/2}(\mathbb{R};L_q(\Omega)^N) \cap L_{p,0,\gamma_0}(\mathbb{R};W^1_q(\Omega)^N), 
\end{align*}
problem \eqref{S} admits a solution $(\fb{U},\Theta)$ of class 
\begin{align} \label{eq:MRTsol}
\begin{aligned} 
\fb{U}&\in W^1_{p,\gamma_0}(0,\infty;L_q(\Omega)^N) \cap L_{p,\gamma_0}(0,\infty;W^2_q(\Omega)^N), \\ 
\Theta&\in L_{p,\gamma_0}(0,\infty;\widehat{W}^1_q(\Omega)) 
\end{aligned}
\end{align}
satisfying the estimate 
\begin{align*} %\label{eq:MR}
&\| e^{-\gamma t}(\partial_t{\fb{U}},\gamma{\fb{U}},\Lambda_\gamma^{1/2}\nabla{\fb{U}},\nabla^2\fb{U},\nabla{\Theta}) \|_{L_p(0,\infty;L_q(\Omega))} \\
&\leq C_{\gamma_0}\{ \| e^{-\gamma t}(\fb{F},\Lambda_\gamma^{1/2}G,\nabla G,\Lambda_\gamma^{1/2}{\fb{H}},\nabla{\fb{H}}) \|_{L_p(\mathbb{R};L_q(\Omega))} 
+ \|e^{-\gamma t}{\partial}_tG\|_{L_p(\mathbb{R};\widehat{W}^{-1}_q(\Omega))} \}  
\end{align*}
for any $\gamma\geq\gamma_0$ with some constant $C_{\gamma_0}$ independent of $\gamma$. 
Moreover, the solution of \eqref{S} is unique, 
that is, if $(\fb{U},\Theta)$ of class \eqref{eq:MRTsol} is a solution to \eqref{S} with $(\fb{F},G,\fb{H})=(0,0,0)$, 
then, $(\fb{U},\Theta)(x,t)=(0,0)$ a.e. $(x,t)\in\Omega\times(0,\infty)$.   
\end{Thm}

Using a fixed-point argument based on this theorem, 
we can prove the local well-posedness of the free boundary problem \eqref{NS}. 
We first derive a quasilinear problem in the fixed layer from \eqref{NS} and next introduce further notation. 
%%%%%%%%%%%%%%%%%%%%%%%%%%%%%%%%%%%%%%%%%%%%%%%%%%%%%%%%%%%%%%%%%%%%%%%%
% reduction
%%%%%%%%%%%%%%%%%%%%%%%%%%%%%%%%%%%%%%%%%%%%%%%%%%%%%%%%%%%%%%%%%%%%%%%%
Since $\Omega(t)$ is unknown, 
we rewrite the equation \eqref{NS} in the Lagrange coordinates $y \in \Omega$ instead of the Euler ones $x \in \Omega(t)$ 
by using the Lagrange transform: 
\begin{gather} \notag
x=y + \int_0^t \fb{u}(y,s) \, ds \equiv \fb{X}_{\fb{u}}(y,t) \\
\fb{u}= (u_1(y,t),\cdots,u_N(y,t))=\fb{v}(\fb{X}_{\fb{u}}(y,t),t), \quad \theta(y,t)=\pi(\fb{X}_{\fb{u}}(y,t),t). 
\end{gather}
Then the pair $(\fb{u},\theta)$ obeys the following problem (cf. \cite[Appendix A]{ShibaShimi07DIE}): 
\begin{align} \label{LNS} %\tag{NS}
\left\{
\begin{array}{ll}
\partial_t{\fb{u}}- \text{Div}\,{\fb{S}}(\fb{u},\theta) = \fb{f}(\fb{u}), \quad 
\text{div}_{}\,{\fb{u}}=g(\fb{u})=\text{div}\,\fb{g}(\fb{u}) & \text{in } \Omega\times(0,T), \\
\fb{S}(\fb{u},\theta)\nu= \fb{h}(\fb{u}) & \text{on } \partial\Omega\times(0,T), \\
\fb{u}|_{t=0}=\fb{v}_0 & \text{in } \Omega. 
\end{array}
\right.
\end{align}
Here, nonlinear terms 
$\fb{f}(\fb{u})=(f_1(\fb{u}),\cdots,f_N(\fb{u}))$, $g(\fb{u})$, $\fb{g}(\fb{u})=(g_1(\fb{u}),\cdots,g_N(\fb{u}))$ 
and $\fb{h}(\fb{u})=(h_1(\fb{u}),\cdots,h_N(\fb{u}))$ are given by
\begin{align} \notag
f_i(\fb{u}) &= \sum_j \fb{V}^1_{ij}\left( \int_0^t \nabla{\fb{u}}\,ds \right)\partial_tu_j 
+ \sum_{j,k,\ell}{\fb{V}}^2_{ijk\ell}\left( \int_0^t \nabla{\fb{u}}\,ds \right)\partial_\ell{\partial}_ku_j \\
&+ \sum_{j,k,\ell,m,n}{\fb{V}}^3_{ijk\ell mn}\left( \int_0^t \nabla{\fb{u}}\,ds \right) \int_0^t \partial_\ell{\partial}_ku_j\,ds \, \partial_nu_m, \\
g(\fb{u}) &= \sum_{j,k}{\fb{V}}^4_{jk}\left( \int_0^t \nabla{\fb{u}}\,ds \right)\partial_ku_j, \quad 
g_i(\fb{u}) = \sum_j{\fb{V}}^5_{ij}\left( \int_0^t \nabla{\fb{u}}\,ds \right)u_j, \\ 
h_i(\fb{u}) &= \sum_{j,k}{\fb{V}}^6_{ijk}\left( \int_0^t \nabla{\fb{u}}\,ds \right)\partial_ku_j 
\end{align}
with some polynomials $\fb{V}^1_{ij}$, $\fb{V}^2_{ijk\ell}$, $\fb{V}^3_{ijk\ell mn}$, $\fb{V}^4_{jk}$, $\fb{V}^5_{ij}$ and $\fb{V}^6_{ijk}$ such that 
\begin{align}
\fb{V}^1_{ij}(\fb{O}), \ \fb{V}^2_{ijk\ell}(\fb{O}), \ \fb{V}^4_{jk}(\fb{O}), \ \fb{V}^5_{ij}(\fb{O}), \ \fb{V}^6_{ijk}(\fb{O})=0.  
\end{align} 

%%%%%%%%%%%%%%%%%%%%%%%%%%%%%%%%%%%%%%%%%%%%%%%%%%%%%%%%%
% S with v0
%%%%%%%%%%%%%%%%%%%%%%%%%%%%%%%%%%%%%%%%%%%%%%%%%%%%%%%%%
As the linearized problem associated with \eqref{LNS}, 
we have to study the Stokes initial value problem
\begin{align} \label{Sv0} 
\left\{
\begin{array}{ll}
\partial_t{\fb{u}}- \text{Div}_{}\,{\fb{S}}(\fb{u},\theta) = \fb{F}, \quad \text{div}_{}\,{\fb{u}}=G & \text{in } \Omega\times(0,\infty), \\
\fb{S}(\fb{u},\theta)\nu= \fb{H}& \text{on } \partial\Omega\times(0,\infty), \\
\fb{u}|_{t=0}=\fb{v}_0 & \text{in } \Omega. 
\end{array}
\right.
\end{align}
By virtue of Theorem \ref{Thm:MR} for \eqref{S} with zero initial condition, 
it suffices to solve the case where $(\fb{F},G,\fb{H})=(0,0,0)$. 
To show the generation of an analytic semigroup, we follow the ideas in \cite[Section 4]{GruSol91MS}, 
see also \cite[p. 159, 160]{ShibaShimi08RAM}. 
By applying $\text{div}$ to the first equation 
and by taking the $N$-th component of the boundary condition in \eqref{Sv0} with $(\fb{F},G,\fb{H})=(0,0,0)$, 
we have 
\begin{align} \label{pres}
\left\{\begin{array}{ll}
\Delta{\theta}=0 & \text{in } \Omega, \\ 
\theta=2\mu{\partial}_Nu_N-\text{div}_{}\,{\fb{u}}& \text{on } \partial\Omega 
\end{array}\right.
\end{align}
since $\text{div}_{}\,{\fb{u}}=0$. 
Set $\tilde{\theta}=\theta-(2\mu{\partial}_Nu_N-\text{div}_{}\,{\fb{u}})$. Then we get 
\begin{align} \label{pres2}
\left\{\begin{array}{ll}
\Delta\tilde{\theta}=-\text{div}_{}\,\nabla(2\mu{\partial}_Nu_N-\text{div}_{}\,{\fb{u}}) & \text{in } \Omega, \\ 
\tilde{\theta}=0 & \text{on } \partial\Omega, 
\end{array}\right.
\end{align}
whose weak formulation is given by \eqref{WDNP} 
with $\fb{f}=-\nabla(2\mu{\partial}_Nu_N-\text{div}_{}\,{\fb{u}})$. 
%%%%%%%%%%%%%%%%%%%%%%%%%%%%%%%%%%%%%%%%%%%%%%%%%%%%%%%%%
% sol op P and Stokes op
%%%%%%%%%%%%%%%%%%%%%%%%%%%%%%%%%%%%%%%%%%%%%%%%%%%%%%%%%
By use of the solution operator $\Pi_0$ obtained in Proposition \ref{Prop:USWDNP},  
the solution $\theta$ to \eqref{pres} is given by 
\begin{align}
&\theta=\Pi{\fb{u}}:=(2\mu{\partial}_Nu_N-\text{div}_{}\,{\fb{u}})-\Pi_0(\nabla(2\mu{\partial}_Nu_N-\text{div}_{}\,{\fb{u}})), \\ 
&\Pi:W^2_q(\Omega)\to W^1_q(\Omega). 
\end{align}
In this way, the problem \eqref{Sv0} with $(\fb{F},G,\fb{H})=(0,0,0)$ is reduced to 
\begin{align} \label{RSv0} 
\left\{
\begin{array}{ll}
\partial_t{\fb{u}}- \text{Div}_{}\,{\fb{S}}(\fb{u},\Pi(\fb{u})) = 0, \quad \text{div}_{}\,{\fb{u}}=0 & \text{in } \Omega\times(0,\infty), \\
\fb{S}(\fb{u},\Pi(\fb{u}))\nu= 0 & \text{on } \partial\Omega\times(0,\infty), \\
\fb{u}|_{t=0}=\fb{v}_0 & \text{in } \Omega. 
\end{array}
\right.
\end{align}
Note that the second equation $\text{div}_{}\,{\fb{u}}=0$ can be recovered if $\fb{v}_0$ is taken from  
\begin{align}
J_q(\Omega) = \{ \fb{u}\in L_q(\Omega)^N \mid \text{div}_{}\,{\fb{u}}=0\}, 
\end{align}
by the uniqueness of solutions to the initial value problem 
for the heat equation subject to the Dirichlet boundary condition, 
which $\text{div}_{}\,{\fb{u}}$ obeys, see \cite[p. 243]{GruSol91MS}. 
We then define the Stokes operator $A_q$ by
\begin{align*} %\label{eq:DefAq}
D(A_q)=\{ \fb{u}\in J_q(\Omega)\cap{W^2_q(\Omega)}^N \mid \fb{S}(\fb{u},\Pi(\fb{u}))\nu=0 \text{ on } \partial\Omega \}, \ 
A_q{\fb{u}}= -\text{Div}_{}\,{\fb{S}}(\fb{u},\Pi(\fb{u})). 
\end{align*}

%%%%%%%%%%%%%%%%%%%%%%%%%%%%%%%%%%%%%%%%%%%%%%%%%%%%%%%%%
% Generation of Semigroup
%%%%%%%%%%%%%%%%%%%%%%%%%%%%%%%%%%%%%%%%%%%%%%%%%%%%%%%%%
Then the system \eqref{RSv0} is formulated as   
\begin{align}
\partial_t{\fb{u}}+A_q{\fb{u}}=0, \quad \fb{u}|_{t=0}=\fb{v}_0. 
\end{align}
By Proposition \ref{Prop:RE} and by the argument in \cite[Lemma 4.4]{Sai18DE} and \cite[Lemma 3.7]{ShibaShimi08RAM}, 
we get the following proposition. 
\begin{Prop} \label{Prop:GeneSG}
The operator $-A_q$ generates an analytic semigroup $\{e^{-tA_q}\}_{t\geq0}$ of class $C^0$ on $J_q(\Omega)$ for $1<q<\infty$. 
\end{Prop}
Hence, $(e^{-tA_q}{\fb{v}}_0,\Pi(e^{-tA_q}{\fb{v}}_0))$ is a solution of \eqref{Sv0} with $(\fb{F},G,\fb{H})=(0,0,0)$. 
Let $(\fb{U},\Theta)$ be the solution obtained in Theorem \ref{Thm:MR}. 
Then $\fb{u}=e^{-tA_q}{\fb{v}}_0+\fb{U}$ and $\theta=\Pi(e^{-tA_q}{\fb{v}}_0)+\Theta$ solve the problem \eqref{Sv0}. 

%%%%%%%%%%%%%%%%%%%%%%%%%%%%%%%%%%%%%%%%%%%%%%%%%%%%%%%%%
% LWP of LNS
%%%%%%%%%%%%%%%%%%%%%%%%%%%%%%%%%%%%%%%%%%%%%%%%%%%%%%%%%
Let $1<p,q<\infty$. We define the Besov space $B^{2(1-1/p)}_{q,p}(\Omega)$ by 
$B^{2(1-1/p)}_{q,p}(\Omega) = (L_q(\Omega),W^2_q(\Omega))_{1-1/p,p}$ 
by use of real interpolation functor $(\cdot,\cdot)_{1-1/p,p}$. 
Moreover, let $\mathcal{D}_{q,p}(\Omega) = (J_q(\Omega),D(A_q))_{1-1/p,p}$. 
The following theorem provides the local well-posedness for the nonlinear problem \eqref{LNS}. 
The proof may be omitted since similar arguments can be found in \cite{Shiba15DE,Sai18DE,ShibaShimi07DIE}. 
\begin{Thm} \label{Thm:LWP}
Let $2<p<\infty$ and $N<q<\infty$. 
For all $R>0$, there exists $T=T(R)>0$ such that 
for any initial data $\fb{v}_0 \in \mathcal{D}_{q,p}(\Omega)\subset B^{2(1-1/p)}_{q,p}(\Omega)^N$ 
with $\| \fb{v}_0 \|_{B^{2(1-1/p)}_{q,p}(\Omega)}\leq R$, 
the problem \eqref{LNS} admits a unique solution 
\begin{align}
\fb{u}\in W^1_p(0,T;L_q(\Omega)^N) \cap L_p(0,T;W^2_q(\Omega)^N) 
\end{align}
with some pressure term $\theta\in L_p(0,T;\widehat{W}^1_q(\Omega))$ 
satisfying the following estimate: 
\begin{align}
\| \partial_t{\fb{u}}\|_{L_p(0,T;L_q(\Omega))} + \| \fb{u}\|_{L_p(0,T;W^2_q(\Omega))} \leq M_0R 
\end{align}
with some constant $M_0$ independent of $T$ and $R$. 
\end{Thm}

%%%%%%%%%%%%%%%%%%%%%%%%%%%%%%%%%%%%%%%%%%%%%%%%%%%%%%%%%%
%%%%%%%%%%%%%%%%%%%%%%%%%%%%%%%%%%%%%%%%%%%%%%%%%%%%%%%%%%
%%														%%
%%	Sec.3 Reduction to the problem only with boundary data		%%
%%														%%
%%%%%%%%%%%%%%%%%%%%%%%%%%%%%%%%%%%%%%%%%%%%%%%%%%%%%%%%%%
%%%%%%%%%%%%%%%%%%%%%%%%%%%%%%%%%%%%%%%%%%%%%%%%%%%%%%%%%%

\section{Reduction to the problem only with boundary data} \label{reduce}
In this section, we reduce the problem \eqref{SR} to the case where $\fb{f}=0$ and $g=0$, 
and the main theorem to the corresponding theorem.  

We begin with the following properties of $\mathscr{R}$-bounded families.  

%%%%%%%%%%%%%%%%%%%%%%%%%%%%%%%%%%%%%%%%%%%%%%%%%%%%%%%%%%%%%%%
% Lemma: Rbddprop, from now on
%%%%%%%%%%%%%%%%%%%%%%%%%%%%%%%%%%%%%%%%%%%%%%%%%%%%%%%%%%%%%%%
\begin{Lem}[{\cite[Proposition 3.4]{DenHiePru03MAMS}}] 
\label{Lem:Rbddprop}
Let $X$, $Y$ and $Z$ be Banach spaces. 
\begin{itemize}
\item[i)] Given $T\in{\mathcal{L}}(X,Y)$, the singleton $\{T\}$ is $\mathscr{R}$-bounded in $\mathcal{L}(X,Y)$. 
\item[ii)] Let $\mathcal{S}$ and $\mathcal{T}$ be $\mathscr{R}$-bounded families on $\mathcal{L}(X,Y)$. 
Then $\mathcal{S}+\mathcal{T}=\{ S+T \mid S \in \mathcal{S}, \ T \in \mathcal{T} \}$ is also $\mathscr{R}$-bounded on $\mathcal{L}(X,Y)$ and
\begin{align}
\mathscr{R}_{\mathcal{L}(X,Y)}(\mathcal{S}+\mathcal{T}) \leq \mathscr{R}_{\mathcal{L}(X,Y)}(\mathcal{S}) + \mathscr{R}_{\mathcal{L}(X,Y)}(\mathcal{T}). 
\end{align} 
\item[iii)] Let $\mathcal{S}$ and $\mathcal{T}$ be $\mathscr{R}$-bounded families on $\mathcal{L}(X,Y)$ and $\mathcal{L}(Y,Z)$, respectively. 
Then $\mathcal{T}{\mathcal{S}}=\{ TS \mid T \in \mathcal{T}, \ S \in \mathcal{S} \}$ is $\mathscr{R}$-bounded on $\mathcal{L}(X,Z)$ and
\begin{align}
\mathscr{R}_{\mathcal{L}(X,Z)}(\mathcal{T}{\mathcal{S}}) \leq \mathscr{R}_{\mathcal{L}(Y,Z)}(\mathcal{T})\mathscr{R}_{\mathcal{L}(X,Y)}(\mathcal{S}). 
\end{align} 
\end{itemize} 
\end{Lem}

%%%%%%%%%%%%%%%%%%%%%%%%%%%%%%%%%%%%%%%%%%%%%%%%%%%%%%%%%%%%%%%%%%%%%
%%%%%%%%%%%%%%%%%%%%%%%%%%%%%%%%%%%%%%%%%%%%%%%%%%%%%%%%%%%%%%%%%%%%%
% reduction to the case g = 0
%%%%%%%%%%%%%%%%%%%%%%%%%%%%%%%%%%%%%%%%%%%%%%%%%%%%%%%%%%%%%%%%%%%%%
%%%%%%%%%%%%%%%%%%%%%%%%%%%%%%%%%%%%%%%%%%%%%%%%%%%%%%%%%%%%%%%%%%%%%
First, we reduce \eqref{SR} to the case $g=0$. 
Let $\varphi \in C^\infty(\mathbb{R})$ be a cut-off function satisfying 
\begin{align}
0\leq\varphi(x_N)\leq1, \quad \varphi(x_N) = \left\{
\begin{array}{ll}
0, & |x_N| \leq 1/3,\\
1, & |x_N| \geq 2/3,
\end{array}
\right.
\end{align}
and set 
\begin{align} \label{eq:varphi0d}
\varphi_\delta(x_N) = \varphi(x_N/\delta), \quad \varphi_0(x_N) = 1-\varphi_\delta(x_N). 
\end{align}
Then, given function $f$ defined on $\Omega$, we define 
\begin{align} \label{eq:oddext}
\begin{aligned}
f_0^o(x)&= \left\{ 
\begin{aligned}
&\varphi_0(x_N)f(x',x_N), && x_N>0, \\ 
&-\varphi_0(-x_N)f(x',-x_N), && x_N<0,
\end{aligned}
\right. \\ 
f_\delta^o(x)&= \left\{ 
\begin{aligned}
&\varphi_\delta(x_N)f(x',x_N), && x_N<\delta, \\
&-\varphi_\delta(2\delta-x_N)f(x',2\delta-x_N), && x_N>\delta.
\end{aligned}
\right. 
\end{aligned}
\end{align}
The following lemma gives a solution to the divergence equation. 

%%%%%%%%%%%%%%%%%%%%%%%%%%%%%%%%%%%%%%%%%%%%%%%%%%%%%%%%%%%%%%%%%%%%%
% Lem: div eq.
%%%%%%%%%%%%%%%%%%%%%%%%%%%%%%%%%%%%%%%%%%%%%%%%%%%%%%%%%%%%%%%%%%%%%
\begin{Lem}[{\cite[Theorem 1.2]{Sai15MMAS}}] 
\label{Lem:diveq}
We define the operator $\mathcal{V}_0g=(\mathcal{V}_{01}g,\cdots,\mathcal{V}_{0N}g)^\mathsf{T}$ by 
\begin{align} \label{eq:K0def}
\mathcal{V}_{0J}g(x) = -\mathscr{F}_\xi^{-1} \left[ \frac{i\xi_J}{|\xi|^2} \mathscr{F}_x[g^\ast](\xi) \right](x) \quad (J=1,\cdots,N)
\end{align}
where $g^\ast(x)=g_0^o(x)+g_\delta^o(x)\in{\widehat{W}^{-1}_q(\mathbb{R}^N)}\cap{W^1_q(\mathbb{R}^N)}$. 
Then, for $1<q<\infty$ and $g \in \widehat{W}^{-1}_q(\Omega)\cap W^1_q(\Omega)$,  
the operator $\mathcal{V}_0 \in \mathcal{L}(\widehat{W}^{-1}_q(\Omega)\cap W^1_q(\Omega), W^2_q(\Omega)^N)$ 
satisfies the following properties: 
\begin{itemize}
\item[a)] 
$\fb{v}=\mathcal{V}_0g$ solves the divergence equation $\text{div}_{}\,{\fb{v}}=g$. 
\item[b)] 
We have the estimates 
\begin{align}
\| \mathcal{V}_0g \|_{L_q(\Omega)} \leq C\| g \|_{\widehat{W}^{-1}_q(\Omega)}, \quad \| \nabla \mathcal{V}_0g \|_{L_q(\Omega)} \leq C \| g \|_{L_q(\Omega)}, \\ 
\| \nabla^2 \mathcal{V}_0g \|_{L_q(\Omega)} \leq C \| \nabla g \|_{L_q(\Omega)}. 
\end{align}
\end{itemize}
\end{Lem}

%%%%%%%%%%%%%%%%%%%%%%%%%%%%%%%%%%%%%%%%%%%%%%%%%%%%%%%%%%%%%%%%%%%%%
% reduction to g=0
%%%%%%%%%%%%%%%%%%%%%%%%%%%%%%%%%%%%%%%%%%%%%%%%%%%%%%%%%%%%%%%%%%%%%
Set $\fb{u}=\mathcal{V}_0g+\fb{v}$ in \eqref{SR}. We have 
\begin{align} \label{eq:SRdfh}
\left\{
\begin{array}{ll}
\lambda \fb{v}- \text{Div}_{}\,{\fb{S}}(\fb{v},\theta) = \tilde{\fb{f}}, \quad \text{div}_{}\,{\fb{v}}= 0 & \text{in } \Omega, \\
\fb{S}(\fb{v},\theta)\nu= \fb{h}- \mu{\fb{D}}(\mathcal{V}_0g)\nu& \text{on } \partial\Omega,
\end{array}
\right.
\end{align}
where 
\begin{align} \label{eq:Defft}
\tilde{\fb{f}}= \fb{f}- \lambda \mathcal{V}_0g + \text{Div}_{}\,(\mu{\fb{D}}(\mathcal{V}_0g)). 
\end{align}

%%%%%%%%%%%%%%%%%%%%%%%%%%%%%%%%%%%%%%%%%%%%%%%%%%%%%%%%%%%%%%%%%%%%%
%%%%%%%%%%%%%%%%%%%%%%%%%%%%%%%%%%%%%%%%%%%%%%%%%%%%%%%%%%%%%%%%%%%%%
% reduction to the case f = 0
%%%%%%%%%%%%%%%%%%%%%%%%%%%%%%%%%%%%%%%%%%%%%%%%%%%%%%%%%%%%%%%%%%%%%
%%%%%%%%%%%%%%%%%%%%%%%%%%%%%%%%%%%%%%%%%%%%%%%%%%%%%%%%%%%%%%%%%%%%%
Next, we reduce the system \eqref{eq:SRdfh} to the case where the data are only on the boundary. 
For this purpose, given function $\fb{f}$ defined on $\Omega$, we consider the problem in the whole space 
\begin{align} \label{eq:SRw}
\lambda{\fb{v}}-\text{Div}_{}\,{\fb{S}}(\fb{v},\pi)=E{\fb{f}}, \quad \text{div}_{}\,{\fb{v}}=0 \quad \text{in } \mathbb{R}^N. 
\end{align}
Here, $E{\fb{f}}=(\bar{f}_1,\cdots,\bar{f}_N)$ 
with $\bar{f}_j=f_{j0}^e+f_{j\delta}^o$ ($j=1,\cdots,N-1$) and $\bar{f}_N=f_{N0}^o+f_{N\delta}^e$, 
where we have set, in addition to \eqref{eq:oddext}, 
\begin{align} 
f_0^e(x)&= \left\{ 
\begin{aligned}
&\varphi_0(x_N)f(x',x_N) && x_N>0, \\ 
&\varphi_0(-x_N)f(x',-x_N) && x_N<0,
\end{aligned}
\right. \\ 
f_\delta^e(x)&= \left\{ 
\begin{aligned}
&\varphi_\delta(x_N)f(x',x_N) && x_N<\delta, \\
&\varphi_\delta(2\delta-x_N)f(x',2\delta-x_N) && x_N>\delta.
\end{aligned}
\right. 
\end{align}
The following lemma on the $\mathscr{R}$-boundedness of the solution operator families for \eqref{eq:SRw} 
is proved by Saito \cite{Sai15MMAS}. 
%%%%%%%%%%%%%%%%%%%%%%%%%%%%%%%%%%%%%%%%%%%%%%%%%%%%%%%%%%%%%%%%%%%%%
% Lem: SRw
%%%%%%%%%%%%%%%%%%%%%%%%%%%%%%%%%%%%%%%%%%%%%%%%%%%%%%%%%%%%%%%%%%%%%
\begin{Lem}[{\cite[Lemma 2.7]{Sai15MMAS}}] 
\label{Lem:Rbddha}
For all $\lambda\in{\mathbb{C}\setminus(-\infty,0]}$, there exist operators $S_0(\lambda)=(S_{01}(\lambda),\cdots,S_{0N}(\lambda))$ and $T_0$ 
satisfying $S_0(\lambda)\in{\mathcal{L}}(L_q(\Omega)^N,W^2_q(\Omega)^N)$ and $T_0\in{\mathcal{L}}(L_q(\Omega)^N,W^1_q(\Omega))$ for $1<q<\infty$ 
such that the following assertions hold: 
\begin{itemize}
\item[a)] For any $\lambda \in \mathbb{C}\setminus(-\infty,0]$ and $\fb{f}\in L_q(\Omega)^N$, the pair 
\begin{align}
(\fb{v},\pi)=(S_0(\lambda){\fb{f}}, T_0{\fb{f}})\in{W^2_q(\Omega)}^N\times{W^1_q(\Omega)} 
\end{align}
is the restriction to $\Omega$ of a unique solution to \eqref{eq:SRw}. 
\item[b)] 
For any $1<q<\infty$, $0<\varepsilon<\pi/2$, $\ell=0,1$ and $1\leq m,n,J \leq N$, 
there hold
\begin{align}
\mathscr{R}_{\mathcal{L}(L_q(\Omega))}(\{ (\tau{\partial}_\tau)^\ell \lambda S_{0J}(\lambda) \mid \lambda\in{\Sigma_{\varepsilon,0}} \})&\leq C_{N,q,\varepsilon,\mu,\delta}, \\ 
\mathscr{R}_{\mathcal{L}(L_q(\Omega))}(\{ (\tau{\partial}_\tau)^\ell \gamma S_{0J}(\lambda) \mid \lambda\in{\Sigma_{\varepsilon,0}} \})&\leq C_{N,q,\varepsilon,\mu,\delta}, \\ 
\mathscr{R}_{\mathcal{L}(L_q(\Omega))}(\{ (\tau{\partial}_\tau)^\ell \lambda^{1/2} \partial_m S_{0J}(\lambda) \mid \lambda\in{\Sigma_{\varepsilon,0}} \})&\leq C_{N,q,\varepsilon,\mu,\delta}, \\ 
\mathscr{R}_{\mathcal{L}(L_q(\Omega))}(\{ (\tau{\partial}_\tau)^\ell \partial_m{\partial}_n S_{0J}(\lambda) \mid \lambda\in{\Sigma_{\varepsilon,0}} \})&\leq C_{N,q,\varepsilon,\mu,\delta}, \\ \label{eq:RbdT0}
\|T_0\|_{\mathcal{L}(L_q(\Omega)^N,L_q(\Omega))}+\|\nabla{T_0}\|_{\mathcal{L}(L_q(\Omega)^N)}&\leq C_{N,q,\mu,\delta}, 
\end{align}
where $\Sigma_{\varepsilon,0}$ is given by \eqref{eq:resoset} with $\gamma_0=0$, and $\lambda=\gamma+i\tau$. 
\end{itemize}
\end{Lem}

Let $\fb{v}=S_0(\lambda){\tilde{\fb{f}}}+\fb{w}$ and $\theta=T_0{\tilde{\fb{f}}}+\pi$ in \eqref{eq:SRdfh}. We have
\begin{align} \label{eq:SRdh}
\left\{
\begin{array}{ll}
\lambda \fb{w}- \text{Div}_{}\,{\fb{S}}(\fb{w},\pi) = 0, \quad \text{div}_{}\,{\fb{w}}= 0 & \text{in } \Omega, \\
\fb{S}(\fb{w},\pi)\nu=\tilde{\fb{h}}\ & \text{on } \partial\Omega, 
\end{array}
\right.
\end{align}
where 
\begin{align} \label{eq:Defht}
\tilde{\fb{h}}=\fb{h}- \mu{\fb{D}}(\mathcal{V}_0g)\nu-\fb{S}(S_0(\lambda){\tilde{\fb{f}}},T_0{\tilde{\fb{f}}})\nu. 
\end{align}

%%%%%%%%%%%%%%%%%%%%%%%%%%%%%%%%%%%%%%%%%%%%%%%%%%%%%%%%%%%%%%%%%%%%%
%%%%%%%%%%%%%%%%%%%%%%%%%%%%%%%%%%%%%%%%%%%%%%%%%%%%%%%%%%%%%%%%%%%%%
% It suffices to show
%%%%%%%%%%%%%%%%%%%%%%%%%%%%%%%%%%%%%%%%%%%%%%%%%%%%%%%%%%%%%%%%%%%%%
%%%%%%%%%%%%%%%%%%%%%%%%%%%%%%%%%%%%%%%%%%%%%%%%%%%%%%%%%%%%%%%%%%%%%
Then, as justified later in this section, 
it suffices to prove the $\mathscr{R}$-boundedness of the solution operator families of \eqref{SR} with $\fb{f}=0$ and $g=0$: 
\begin{align} \label{eq:SRdh} %\tag{$\text{SR}_{\lambda,\delta}$} % Stokes Resolvent problem with Data H 
\left\{
\begin{array}{ll}
\lambda \fb{u} - \text{Div}_{}\,{\fb{S}}(\fb{u},\theta) = 0, \quad \text{div}_{}\,\fb{u} = 0 & \text{in } \Omega, \\
\fb{S}(\fb{u},\theta)\nu= \fb{h} & \text{on } \partial\Omega. 
\end{array}
\right.
\end{align}

%%%%%%%%%%%%%%%%%%%%%%%%%%%%%%%%%%%%%%%%%%%%%%%%%%%%%%%%%%%%%%%%%%%%%
% Thm: R-b'ddness data h case
%%%%%%%%%%%%%%%%%%%%%%%%%%%%%%%%%%%%%%%%%%%%%%%%%%%%%%%%%%%%%%%%%%%%%
\begin{Thm} \label{Thm:Rbdddh}
For all $\lambda \in \mathbb{C}\setminus(-\infty,0]$, 
there exist operators $\mathcal{S}(\lambda)=(\mathcal{S}_1(\lambda),\cdots,\mathcal{S}_N(\lambda))$ and $\mathcal{T}(\lambda)$ 
(precisely, they are given by \eqref{eq:defSNP} and \eqref{eq:defSj}) satisfying 
$\mathcal{S}(\lambda)\in{\mathcal{L}}(L_q(\Omega)^{N+N^2}, W^2_q(\Omega)^N)$ and 
$\mathcal{T}(\lambda)\in{\mathcal{L}}(L_q(\Omega)^{N+N^2}, \widehat{W}^1_q(\Omega))$ for $1<q<\infty$ such that 
the following assertions hold: 
\begin{itemize}
\item[a)] 
For any $1<q<\infty$, $\lambda \in \mathbb{C}\setminus(-\infty,0]$ and $\fb{h}=(\fb{h}',h_N)=(h_1,\cdots,h_N) \in W^1_q(\Omega)^N$, 
the pair $(\fb{u},\theta) \in W^2_q(\Omega)^N \times \widehat{W}^1_q(\Omega)$ given by 
\begin{align} \label{eq:RbdddhSolform}
\fb{u}=\mathcal{S}(\lambda)(\lambda^{1/2}\fb{h}', h_N, \nabla \fb{h}), \quad 
\theta=\mathcal{T}(\lambda)(\lambda^{1/2}\fb{h}', h_N, \nabla \fb{h})  
\end{align}
is a solution of \eqref{eq:SRdh}. 
\item[b)] 
%The following operator families in $\mathcal{L}(L_q(\Omega)^{N+N^2},L_q(\Omega))$ are $\mathscr{R}$-bounded 
For any $1<q<\infty$, $0<\varepsilon<\pi/2$, $\gamma_0>0$, 
$\ell=0,1$ and $1\leq m,n,J\leq N$, there hold
\begin{align} \label{eq:Rbdddh} 
\begin{aligned}
\mathscr{R}_{\mathcal{L}(L_q(\Omega))}(\{ (\tau{\partial}_\tau)^\ell \lambda \mathcal{S}_J(\lambda) \mid \lambda\in{\Sigma_{\varepsilon,\gamma_0}} \})&\leq C_{N,q,\varepsilon,\gamma_0,\mu,\delta}, \\  
\mathscr{R}_{\mathcal{L}(L_q(\Omega))}(\{ (\tau{\partial}_\tau)^\ell \gamma \mathcal{S}_J(\lambda) \mid \lambda\in{\Sigma_{\varepsilon,\gamma_0}} \})&\leq C_{N,q,\varepsilon,\gamma_0,\mu,\delta}, \\ 
\mathscr{R}_{\mathcal{L}(L_q(\Omega))}(\{ (\tau{\partial}_\tau)^\ell \lambda^{1/2} \partial_m \mathcal{S}_J(\lambda) \mid \lambda\in{\Sigma_{\varepsilon,\gamma_0}} \})&\leq C_{N,q,\varepsilon,\gamma_0,\mu,\delta}, \\%& 
\mathscr{R}_{\mathcal{L}(L_q(\Omega))}(\{ (\tau{\partial}_\tau)^\ell \partial_m{\partial}_n \mathcal{S}_J(\lambda) \mid \lambda\in{\Sigma_{\varepsilon,\gamma_0}} \})&\leq C_{N,q,\varepsilon,\gamma_0,\mu,\delta}, \\
\mathscr{R}_{\mathcal{L}(L_q(\Omega))}(\{ (\tau{\partial}_\tau)^\ell \partial_m \mathcal{T}(\lambda) \mid \lambda\in{\Sigma_{\varepsilon,\gamma_0}} \})&\leq C_{N,q,\varepsilon,\gamma_0,\mu,\delta}, %& &
\end{aligned}
\end{align}
where $\Sigma_{\varepsilon,\gamma_0}$ is given by \eqref{eq:resoset} and $\lambda=\gamma+i\tau$. 
\end{itemize}
\end{Thm}

\begin{Rem} \label{Rem:Rbdddh}
It is reasonable (and possible by this theorem on account of $\gamma_0>0$) to show the $\mathscr{R}$-boundedness for 
\begin{align}
\fb{u}=\mathcal{S}(\lambda)(\lambda^{1/2}\fb{h}, \nabla \fb{h}), \quad 
\theta=\mathcal{T}(\lambda)(\lambda^{1/2}\fb{h}, \nabla \fb{h})  
\end{align}
instead of \eqref{eq:RbdddhSolform}. 
In view of \eqref{eq:RbdT0} and \eqref{eq:Defht}, however, 
this is not useful to show Theorem \ref{Thm:Rbdd} 
since one cannot control $\lambda^{1/2}{T_0}{\tilde{\fb{f}}}{\nu}$ that appears in $\lambda^{1/2} \tilde{\fb{h}}$. 
This is why we provide Theorem \ref{Thm:Rbdddh} with the solution operators \eqref{eq:RbdddhSolform}. 
Note that the pressure $T_0\tilde{\fb{f}}$ is present only in the $N$-th component of $\tilde{\fb{h}}$. 
\end{Rem}

In the remainder of this section, 
we prove that our main theorem (Theorem \ref{Thm:Rbdd}) follows from Theorem \ref{Thm:Rbdddh}, 
which will be established in Section \ref{Sec:Pf}. 
To show the uniqueness of solutions by duality, we use the following lemma, see for instance \cite[Lemma 7.1]{Sai15MMAS}. 

%%%%%%%%%%%%%%%%%%%%%%%%%%%%%%%%%%%% Lem
\begin{Lem} 
\label{Lem:GaussDivFD}
Let $1<q<\infty$ and let $q'$ be its dual exponent. 
For given $\fb{u}\in W^2_q(\Omega)^N$, $\fb{v}\in W_{q'}^2(\Omega)^N$, $\theta\in W^1_q(\Omega)$ and $\pi\in W_{q'}^1(\Omega)$, 
we have the following formula: 
\begin{align}
(\fb{u},\textup{Div}\,\fb{S}(\fb{v},\pi))_\Omega = (\textup{Div}\,\fb{S}(\fb{u},\theta),\fb{v})_\Omega 
+ \langle \fb{u},\fb{S}(\fb{v},\pi)\nu\rangle_{\partial\Omega} - \langle \fb{S}(\fb{u},\theta)\nu,\fb{v}\rangle_{\partial\Omega} \\
+ (\textup{div}\,\fb{u},\pi)_\Omega - (\theta,\textup{div}\,\fb{v})_\Omega. 
\end{align}
\end{Lem}

%%%%%%%%%%%%%%%%%%%%%%%%%%%%%%%%%%%%%%%%%%%%%%%%%%%%%%%%%%%%%%%%%%%%%
% recover the Rbddness
%%%%%%%%%%%%%%%%%%%%%%%%%%%%%%%%%%%%%%%%%%%%%%%%%%%%%%%%%%%%%%%%%%%%%
\begin{proof}[Proof of Theorem \ref{Thm:Rbdd} from Theorem \ref{Thm:Rbdddh}]
By summarizing the aforementioned arguments, the solution of \eqref{SR} is given by 
\begin{align} \label{eq:upred} %formula for u,p in reduction
\quad (\fb{u},\,\theta) = 
\big( \mathcal{V}_0g+S_0(\lambda){\tilde{\fb{f}}}+\mathcal{S}(\lambda)(\lambda^{1/2}{\tilde{\fb{h}}}',\tilde{h}_N,\nabla{\tilde{\fb{h}}}), \, 
T_0{\tilde{\fb{f}}}+\mathcal{T}(\lambda)(\lambda^{1/2}{\tilde{\fb{h}}}',\tilde{h}_N,\nabla{\tilde{\fb{h}}}) \big). 
\quad 
\end{align}
We define the operator $\fb{\widetilde{\mathcal{V}}}_0=(\widetilde{\mathcal{V}}_0^{jk\ell})_{1\leq j,k,\ell\leq N}$ by 
\begin{align}
\widetilde{\mathcal{V}}_0^{jk\ell } \fb{f}&= 2\partial_\ell {\mathcal{V}}_{0j} f_k, \quad \text{when } 1\leq k\leq N-1, \\ 
\widetilde{\mathcal{V}}_0^{jN\ell } \fb{f}&= 2\partial_\ell {\mathcal{V}}_{0N} f_j, \quad \text{when } 1\leq j\leq N-1, \\ 
\widetilde{\mathcal{V}}_0^{NN\ell } \fb{f}&= 2f_\ell -2\sum_{j'=1}^{N-1}{\partial}_\ell {\mathcal{V}}_{0j'}f_{j'} 
\end{align}
for $\fb{f}=(f_1,\cdots,f_N)\in L_q(\Omega)^N$ and $1\leq j,k,\ell \leq N$. 
Then, for $g\in{W^1_q(\Omega)}$, 
\begin{align}
\fb{\widetilde{\mathcal{V}}}_0(\nabla g)=(\widetilde{\mathcal{V}}_0^{jk\ell }(\nabla g))_{1\leq j,k,\ell \leq N}
=(\partial_\ell D_{jk}(\mathcal{V}_0g))_{1\leq j,k,\ell \leq N}=\nabla{\fb{D}}(\mathcal{V}_0g) 
\end{align}
since 
\begin{align} %\label{eq:rwDij}
&\begin{aligned} %\notag
&\partial_\ell {D_{jk}}(\mathcal{V}_0g) \\ %\notag
&= \partial_\ell \left(-\partial_k{\mathscr{F}}_\xi^{-1} \left[ \frac{i\xi_{j}}{|\xi|^2} \mathscr{F}_x[g^\ast](\xi) \right](x)
-\partial_{j}{\mathscr{F}}_\xi^{-1} \left[ \frac{i\xi_k}{|\xi|^2} \mathscr{F}_x[g^\ast](\xi) \right](x)\right) \\ %\notag 
&= -2\partial_\ell {\mathscr{F}}_\xi^{-1} \left[ \frac{i\xi_{j}}{|\xi|^2} \mathscr{F}_x[(\partial_kg)^\ast](\xi) \right](x) \\ 
&= 2\partial_\ell {\mathcal{V}}_{0j}(\partial_kg) 
= \widetilde{\mathcal{V}}_0^{jk\ell }(\nabla g) \quad \text{when $1\leq k \leq N-1$,}  
\end{aligned} 
\end{align}
similarly, 
\begin{align}
\partial_\ell {D_{jN}}(\mathcal{V}_0 g)=\widetilde{\mathcal{V}}_0^{jN\ell }(\nabla g) \quad \text{when $1\leq j\leq N-1$} 
\end{align}
and 
\begin{align} \label{eq:rwDNN}
&\begin{aligned}
&\partial_\ell {D_{NN}}(\mathcal{V}_0g) = \partial_\ell(2\partial_N\mathcal{V}_{0N}g)  \\
&= 2\partial_\ell \bigg(\text{div}_{}\,{\mathcal{V}}_0g-\sum_{j'=1}^{N-1}{\partial}_{j'}{\mathcal{V}}_{0j'}g\bigg) 
= 2\partial_\ell g-2\sum_{j'=1}^{N-1}{\partial}_\ell {\mathcal{V}}_{0j'}(\partial_{j'}g)
=\widetilde{\mathcal{V}}_0^{NN\ell }(\nabla g). 
\end{aligned}
\end{align} 
By Lemma \ref{Lem:diveq}, we have 
\begin{align} \label{eq:proptK0}
\fb{\widetilde{\mathcal{V}}}_0\in{\mathcal{L}}(L_q(\Omega)^N,L_q(\Omega)^{N^3}). 
\end{align} 
Additionally, we define $M_{\lambda^{-1/2}}$ and an extension $\tilde{\nu}=(0,\cdots,0,\tilde{\nu}_N)$ of $\nu$ by 
\begin{align} \label{eq:defMltn}
\begin{aligned}
&M_{\lambda^{-1/2}}f=\lambda^{-1/2}f \quad (f\in{L_q(\Omega)}), \\ 
&\tilde{\nu}(x)=\varphi_\delta(x_N)\nu(x',\delta)+\varphi_0(x_N)\nu(x',0). 
\end{aligned}
\end{align}
By the Kahane's contraction principle, see \cite[Proposition 2.5]{KunWei04Sp}, if $\gamma_0>0$, we have 
\begin{align} \label{eq:RbdMl}
\mathscr{R}_{\mathcal{L}(L_q(\Omega))}(\{M_{\lambda^{-1/2}}\mid \lambda\in{\Sigma_{\varepsilon,\gamma_0}}\})\leq2\gamma_0^{-1/2}. 
\end{align}
In view of \eqref{eq:Defft} and \eqref{eq:Defht}, we define 
\begin{align} 
&[R_0(\fb{f}, \lambda g, \nabla g)]_i \\ 
&= [\tilde{\fb{f}}]_i \\ 
&= [\fb{f}]_i - \mathcal{V}_{0i}\lambda g 
+ \sum_{j=1}^N \mu{\fb{\widetilde{\mathcal{V}}}}_0^{ijj}(\nabla g), \\ %\label{eq:rw2} %R_1'
&R_1'(\lambda)(\fb{f}, \lambda g, \lambda^{1/2}g, \nabla g, \lambda^{1/2}{\fb{h}}, \nabla{\fb{h}}) \\ 
&= \lambda^{1/2}{\tilde{\fb{h}}}' \\ 
&=\lambda^{1/2}{\fb{h}}' - \mu[\fb{D}(\mathcal{V}_0\lambda^{1/2}g)\tilde{\nu}]'
-\mu[\lambda^{1/2}{\fb{D}}(S_0(\lambda){R}_0(\fb{f}, \lambda g, \nabla g))\tilde{\nu}]', \\ %\label{eq:rw3} %R_1N
&R_{1N}(\lambda)(\fb{f}, \lambda g, \lambda^{1/2}g, \nabla g, \lambda^{1/2}{\fb{h}}, \nabla{\fb{h}}) \\ 
&= \tilde h_N \\ 
&=M_{\lambda^{-1/2}}\lambda^{1/2} h_N - \mu M_{\lambda^{-1/2}}[\fb{D}(\mathcal{V}_0\lambda^{1/2}g)\tilde{\nu}]_N \\
&-\mu M_{\lambda^{-1/2}}[\lambda^{1/2}{\fb{D}}(S_0(\lambda){R}_0(\fb{f}, \lambda g, \nabla g))\tilde{\nu}]_N
+T_0{R}_0(\fb{f}, \lambda g, \nabla g)\tilde{\nu}_N, \\ %\label{eq:rw4} %R_2
\label{eq:rw}
&\begin{aligned}
&R_2(\lambda)(\fb{f}, \lambda g, \lambda^{1/2}g, \nabla g, \lambda^{1/2}{\fb{h}}, \nabla{\fb{h}}) \\ 
&=\nabla\tilde{\fb{h}}\\ 
&=\nabla{\fb{h}}- \nabla\big(\mu{\fb{D}}(\mathcal{V}_0g)\tilde{\nu}\big) 
-\nabla\big(\fb{S}(S_0(\lambda){R}_0(\fb{f}, \lambda g, \nabla g),T_0{R}_0(\fb{f}, \lambda g, \nabla g))\tilde{\nu}\big) \\ 
&=\nabla{\fb{h}}- \mu{\fb{\widetilde{\mathcal{V}}}}_0(\nabla g)\tilde{\nu}
- \mu M_{\lambda^{-1/2}}{\fb{D}}(\mathcal{V}_0\lambda^{1/2}g)\nabla\tilde{\nu}\\ 
&-\nabla{\fb{S}}(S_0(\lambda){R}_0(\fb{f}, \lambda g, \nabla g),T_0{R}_0(\fb{f}, \lambda g, \nabla g))\tilde{\nu}\\ 
&-M_{\lambda^{-1/2}}\mu\lambda^{1/2}{\fb{D}}(S_0(\lambda){R}_0(\fb{f}, \lambda g, \nabla g))\nabla\tilde{\nu}
+T_0{R}_0(\fb{f}, \lambda g, \nabla g)\nabla\tilde{\nu}, 
\end{aligned}
\end{align}
where $[\fb{g}]_J$ is the $J$-th component of $N$-vector function $\fb{g}$ for $1\leq J\leq N$ 
and $[\fb{g}]'=([\fb{g}]_1,\cdots,[\fb{g}]_{N-1})$. 
By Lemma \ref{Lem:Rbddprop}, Lemma \ref{Lem:diveq}, Lemma \ref{Lem:Rbddha}, \eqref{eq:proptK0} and \eqref{eq:RbdMl}, 
\begin{align} \label{eq:RbddRjl}
\begin{aligned}
R_0\in{\mathcal{L}}(L_q(\Omega)^N \times \widehat{W}^{-1}_q(\Omega) \times L_q(\Omega)^N,L_q(\Omega)^N), \\ 
\mathscr{R}_{\mathcal{L}(X_q(\Omega),L_q(\Omega))}(\{ R_1'(\lambda) \mid \lambda\in{\Sigma_{\varepsilon,0}} \}){\leq C_{N,q,\varepsilon,\mu,\delta}}, \\ 
\mathscr{R}_{\mathcal{L}(X_q(\Omega),L_q(\Omega))}(\{ R_{1N}(\lambda) \mid \lambda\in{\Sigma_{\varepsilon,0}} \}){\leq C_{N,q,\varepsilon,\mu,\delta}}, \\ 
\mathscr{R}_{\mathcal{L}(X_q(\Omega),L_q(\Omega))}(\{ R_2(\lambda) \mid \lambda\in{\Sigma_{\varepsilon,0}} \}){\leq C_{N,q,\varepsilon,\mu,\delta}}.  
\end{aligned}
\end{align}
In view of \eqref{eq:upred}, we define $\mathcal{U}(\lambda)$ and $\mathcal{P}(\lambda)$ by 
\begin{align} \label{eq:defUP}
\begin{aligned}
&\mathcal{U}(\lambda)(\fb{f}, \lambda g,\lambda^{1/2}g, \nabla g, \lambda^{1/2}{\fb{h}}, \nabla \fb{h}) \\ 
&= \mathcal{V}_0g+S_0(\lambda){R}_0(\lambda)(\fb{f}, \lambda g, \nabla g) %\\ 
+\mathcal{S}(\lambda)\big(R_1'(\lambda), R_{1N}(\lambda), R_2(\lambda)\big)
(\fb{f}, \lambda g, \lambda^{1/2}g, \nabla g, \lambda^{1/2}{\fb{h}}, \nabla{\fb{h}}), \\ 
&\mathcal{P}(\lambda)(\fb{f}, \lambda g,\lambda^{1/2}g, \nabla g, \lambda^{1/2}{\fb{h}}, \nabla \fb{h}) \\ 
&= T_0{R}_0(\lambda)(\fb{f}, \lambda g, \nabla g) %\\ 
+\mathcal{T}(\lambda)\big(R_1'(\lambda), R_{1N}(\lambda), R_2(\lambda)\big)
(\fb{f}, \lambda g, \lambda^{1/2}g, \nabla g, \lambda^{1/2}{\fb{h}}, \nabla{\fb{h}}). 
\end{aligned}
\end{align}
Then, from Lemma \ref{Lem:Rbddprop}, Lemma \ref{Lem:diveq}, Lemma \ref{Lem:Rbddha}, 
Theorem \ref{Thm:Rbdddh} and \eqref{eq:RbddRjl}, 
we obtain the existence and the $\mathscr{R}$-boundedness of the solution operator families. 

Thus, the proof will be completed by showing the uniqueness of solutions to \eqref{SR}. 
Assume that $(\fb{u},\theta)\in{W^2_q(\Omega)}^N\times{\widehat{W}^1_q(\Omega)}$ satisfies the following equation:  
\begin{align} \label{SR0} %\tag{$\text{SR}_{\lambda,\delta}$}
\left\{
\begin{array}{ll}
\lambda \fb{u}- \text{Div}_{}\,{\fb{S}}(\fb{u},\theta) = 0, \quad \text{div}_{}\,{\fb{u}}= 0 & \text{in } \Omega, \\
\fb{S}(\fb{u},\theta)\nu= 0 & \text{on } \partial\Omega. 
\end{array}
\right.
\end{align}
We find $\theta\in{L_q(\Omega)}$ by using the $N$-th component of the boundary condition as follows:
\begin{align} \label{eq:0dpres}
\|\theta\|_{L_q(\Omega)} 
&= \bigg\| 2\mu{\partial}_Nu_N(x',0)+\int_0^{x_N}{\partial}_N{\theta}(x',y_N)\,dy_N \bigg\|_{L_q(\Omega)} \\
&\leq 2\mu\delta^{1/q}\|\partial_Nu_N(x',0)\|_{L_q(\mathbb{R}^{N-1})} + \delta\|\partial_N{\theta}\|_{L_q(\Omega)} \\
&\leq C(\|u\|_{W^2_q(\Omega)} + \|\partial_N{\theta}\|_{L_q(\Omega)})<\infty. 
\end{align}
Thus, given $\varphi\in C_0^\infty(\Omega)$, 
we take a solution $(\fb{v},\pi)\in W^2_{q'}(\Omega)^N\times W^1_{q'}(\Omega)$ of the problem 
\begin{align} \label{SRdual} %\tag{$\text{SR}_{\lambda,\delta}$}
\left\{
\begin{array}{ll}
\overline{\lambda} \fb{v}- \text{Div}_{}\,{\fb{S}}(\fb{v},\pi) = \varphi, \quad \text{div}_{}\,{\fb{v}}= 0 & \text{in } \Omega, \\
\fb{S}(\fb{v},\pi)\nu= 0 & \text{on } \partial\Omega, 
\end{array}
\right.
\end{align}
and employ Lemma \ref{Lem:GaussDivFD} to see that  
\begin{align}
(\fb{u},\varphi)_\Omega = (\fb{u},\overline{\lambda} \fb{v}- \text{Div}_{}\,{\fb{S}}(\fb{v},\pi))_\Omega 
= (\lambda \fb{u}- \text{Div}_{}\,{\fb{S}}(\fb{u},\theta),\fb{v})_\Omega = 0, 
\end{align}
which implies $\fb{u}=0$. 
Hence, $\nabla{\theta}=0$ by the first equation of \eqref{SR0}, and so $\theta=0$ by the boundary condition, which proves the uniqueness. 
\end{proof}

%%%%%%%%%%%%%%%%%%%%%%%%%%%%%%%%%%%%%%%%%%%%%%%%%%%%%%%%%%%%%%
%%%%%%%%%%%%%%%%%%%%%%%%%%%%%%%%%%%%%%%%%%%%%%%%%%%%%%%%%%%%%%
%%															%%
%%	Sec.4 Solution formulas for the problem only with boundary data	%%
%%															%%
%%%%%%%%%%%%%%%%%%%%%%%%%%%%%%%%%%%%%%%%%%%%%%%%%%%%%%%%%%%%%%
%%%%%%%%%%%%%%%%%%%%%%%%%%%%%%%%%%%%%%%%%%%%%%%%%%%%%%%%%%%%%%

\section{Solution formulas for the problem only with boundary data} \label{sec:formula}

%%%%%%%%%%%%%%%%%%%%%%%%%%%%%%%%%%%%%%%%%%%%%%%%%%%%%%%%%%%%
% main result for this case
%%%%%%%%%%%%%%%%%%%%%%%%%%%%%%%%%%%%%%%%%%%%%%%%%%%%%%%%%%%%
In this section, we give the solution formulas of \eqref{eq:SRdh} in order to prove Theorem \ref{Thm:Rbdddh}. 
%%%%%%%%%%%%%%%%%%%%%%%%%%%%%%%%%%%%%%%%%%%%%%%%%%%%%%%%
% Find solution formula
%%%%%%%%%%%%%%%%%%%%%%%%%%%%%%%%%%%%%%%%%%%%%%%%%%%%%%%%
%%%%%%%%%%%%%%%%%%%%%%%%%%%%%%%%%%%%%%%% partial Fourier trans
By applying the partial Fourier transform with respect to $x'$ to \eqref{eq:SRdh}, we get
\begin{align} \label{eq:FSRdh}
\left\{
\begin{array}{ll}
\mu(B^2-\partial_N^2) \widehat{u}_j(\xi',x_N) + i\xi_j{\widehat{\theta}(\xi',x_N)} = 0 & 0<x_N<\delta, \\
\mu(B^2-\partial_N^2) \widehat{u}_N(\xi',x_N) + \partial_N{\widehat{\theta}(\xi',x_N)} = 0 & 0<x_N<\delta, \\ 
i\xi'\cdot{\widehat{u}{'}(\xi',x_N)} + \partial_N{\widehat{u}_N(\xi',x_N)} = 0 & 0<x_N<\delta, \\
\mu(\partial_N{\widehat{u}_j(\xi',x_N)} + i\xi_j{\widehat{u}_N(\xi',x_N)})\nu_N(x_N) = \widehat{h}_j(\xi',x_N) & x_N \in \{0,\delta\}, \\
(2\mu{\partial}_N{\widehat{u}_N(\xi',x_N)} - \widehat{\theta}(\xi',x_N))\nu_N(x_N) = \widehat{h}_N(\xi',x_N) & x_N \in \{0,\delta\}
\end{array}
\right.
\end{align} 
for $1\leq j\leq N-1$, where 
$\nu_N$ means the $N$-th component of the unit outer normal to $\partial\Omega$, that is, $\nu_N(\delta) = 1$ and $\nu_N(0)=-1$.
Here, we have set
\begin{align} \label{eq:DefAB}
A=|\xi'|, \quad B=\sqrt{\mu^{-1}\lambda+A^2} 
\end{align}
with $\text{Re\,} B>0$. 
%%%%%%%%%%%%%%%%%%%%%%%%%%%%%%%%%%%%%%%%%%%%%%% combine tangent
To simplify the first and fourth equations of \eqref{eq:FSRdh}, we multiply them by $-i\xi_j$ and add up the resultant formulas. 
Then, if we set 
\begin{align}
\widehat{u}_d(\xi',x_N)=i\xi'\cdot{\widehat{u}{'}(\xi',x_N)}, \quad  \widehat{h}_d(\xi',x_N)=i\xi'\cdot{\widehat{h}{'}(\xi',x_N)}, 
\end{align}
we obtain the following ordinary differential equations 
with only three unknowns $\widehat{u}_d(\xi',\cdot)$, $\widehat{u}_N(\xi',\cdot)$ and $\widehat{\theta}(\xi',\cdot)$: 
\begin{align} \label{eq:FSRdht} % after reduction for Tangent direction 
\left\{
\begin{array}{ll}
\mu(B^2-\partial_N^2) \widehat{u}_d(\xi',x_N) - A^2\widehat{\theta}(\xi',x_N) = 0 & 0<x_N<\delta, \\
\mu(B^2-\partial_N^2) \widehat{u}_N(\xi',x_N) + \partial_N{\widehat{\theta}(\xi',x_N)} = 0 & 0<x_N<\delta, \\ 
\widehat{u}_d(\xi',x_N) + \partial_N{\widehat{u}_N(\xi',x_N)} = 0 & 0<x_N<\delta, \\
\mu(\partial_N{\widehat{u}_d(\xi',x_N)} - A^2\widehat{u}_N(\xi',x_N))\nu_N(x_N) = \widehat{h}_d(\xi',x_N) & x_N \in \{0,\delta\}, \\
(2\mu{\partial}_N{\widehat{u}_N(\xi',x_N)} - \widehat{\theta}(\xi',x_N))\nu_N(x_N) = \widehat{h}_N(\xi',x_N) & x_N \in \{0,\delta\}.
\end{array}
\right.
\end{align} 

%%%%%%%%%%%%%%%%%%%%%%%%%%%%%%%%%%%%%%%%%%%%%%% find fundamental sol
In order to solve this system, 
we add the second equation multiplied by $\partial_N$ to the first equation in \eqref{eq:FSRdht} 
and then we have by the third equation 
\begin{align} \label{eq:FSRdhp}
(\partial_N^2-A^2)\widehat{\theta}(\xi',x_N)=0. 
\end{align}
Multiplying $(A^2-\partial_N^2)$ to the first and the second equations implies that
\begin{align} \label{eq:FSRdhu}
%(A^2-\partial_N^2)(B^2-\partial_N^2)\widehat{u}_d(\xi',x_N)=0, \quad (A^2-\partial_N^2)(B^2-\partial_N^2)\widehat{u}_N(\xi',x_N)=0. 
\begin{aligned}
(A^2-\partial_N^2)(B^2-\partial_N^2)\widehat{u}_d(\xi',x_N)&=0, \\
(A^2-\partial_N^2)(B^2-\partial_N^2)\widehat{u}_N(\xi',x_N)&=0. 
\end{aligned}
\end{align}
Thus, the solution to \eqref{eq:FSRdht} can be given by
\begin{align} 
\begin{aligned} \begin{aligned} \label{eq:ODEsolform0} 
\widehat{u}_d(\xi',x_N) &= \sum_{\ell=1,2} (\alpha_{\ell d}^0 e^{-Ad_\ell(x_N)} + \beta_{\ell d}^0 e^{-Bd_\ell(x_N)}), \\ 
\widehat{u}_N(\xi',x_N) &= \sum_{\ell=1,2} (\alpha_{\ell N}^0 e^{-Ad_\ell(x_N)} + \beta_{\ell N}^0 e^{-Bd_\ell(x_N)}), \\
\widehat{\theta}(\xi',x_N) &= \sum_{\ell=1,2} \gamma_\ell e^{-Ad_\ell(x_N)} 
\end{aligned} \end{aligned} 
\end{align}
with some coefficients 
$\alpha_{\ell d}^0$, $\beta_{\ell d}^0$, $\alpha_{\ell N}^0$, $\beta_{\ell N}^0$ and $\gamma_\ell$ 
depending on $\lambda$ and $\xi'$,
where 
\begin{align} \label{eq:Defdl}
d_\ell(x_N) = \left\{ 
\begin{array}{ll}
\delta-x_N, & \ell=1, \\ 
x_N, & \ell=2.
\end{array}
\right. 
\end{align}
%%%%%%%%%%%%%%%%%%%%%%%%%%%%%%%%%%%%%%%%%%%%%%% Insert it.
We obtain the following relations of these coefficients 
by inserting \eqref{eq:ODEsolform0} to the first, second and third equations of \eqref{eq:FSRdht}: 
\begin{align}
\begin{aligned} \label{eq:ODErelation}
\lambda \alpha_{\ell d}^0 - A^2 \gamma_\ell = 0, \quad 
\lambda \alpha_{\ell N}^0 + (-1)^{\ell-1} A \gamma_\ell = 0, \\
\alpha_{\ell d}^0 + (-1)^{\ell-1} A \alpha_{\ell N}^0 = 0, \quad
\beta_{\ell d}^0 + (-1)^{\ell-1} B \beta_{\ell N}^0 = 0 
\end{aligned}
\end{align}
for $\ell=1,2$. 
Then we have the following system for $\alpha_{\ell N}^0$ and $\beta_{\ell N}^0$ 
if we insert \eqref{eq:ODEsolform0} to the fourth and fifth equations of \eqref{eq:FSRdht} 
and rewrite them by using \eqref{eq:ODErelation}: 
\begin{align} 
\begin{aligned} \label{eq:AlgEq0}
-2A^2 (\alpha_{1N}^0 + a\alpha_{2N}^0) 
-D_0(\beta_{1N}^0 + b\beta_{2N}^0) &= \mu^{-1}{\widehat{h}_d(\xi',\delta)}, \\ 
-2A^2 (a\alpha_{1N}^0 + \alpha_{2N}^0) 
-D_0(b\beta_{1N}^0 + \beta_{2N}^0) &= -\mu^{-1}{\widehat{h}_d(\xi',0)}, \\ 
-D_0(-\alpha_{1N}^0 + a\alpha_{2N}^0) 
-2AB (-\beta_{1N}^0 + b\beta_{2N}^0) &= \mu^{-1}A{\widehat{h}_N(\xi',\delta)}, \\ 
-D_0(-a\alpha_{1N}^0 + \alpha_{2N}^0) 
-2AB (-b\beta_{1N}^0 + \beta_{2N}^0) &= -\mu^{-1}A{\widehat{h}_N(\xi',0)} 
\end{aligned}
\end{align} 
where 
\begin{align} \label{eq:DefabD0}
a = e^{-A\delta}, \quad b=e^{-B\delta}, \quad D_0=B^2+A^2. 
\end{align}

%%%%%%%%%%%%%%%%%%%%%%%%%%%%%%%%%%%%%%%%%%%%% change the solution form
Now, we rewrite \eqref{eq:ODEsolform0} as  
\begin{align} 
\begin{aligned} \label{eq:ODEsolform}
\widehat{u}_d(\xi',x_N) &= \sum_{\ell=1,2} (\mu_{\ell d} \mathcal{M}(d_\ell(x_N)) + \beta_{\ell d} e^{-Bd_\ell(x_N)}), \\
\widehat{u}_N(\xi',x_N) &= \sum_{\ell=1,2} (\mu_{\ell N} \mathcal{M}(d_\ell(x_N)) + \beta_{\ell N} e^{-Bd_\ell(x_N)}), \\
\widehat{\theta}(\xi',x_N) &= \sum_{\ell=1,2} \gamma_\ell e^{-Ad_\ell(x_N)} 
\end{aligned} 
\end{align}
with some coefficients $\mu_{\ell d}$, $\beta_{\ell d}$, $\mu_{\ell N}$ and $\beta_{\ell N}$ 
depending on $\lambda$ and $\xi'$ so that $(B-A)^{-1}$ does not appear, where
\begin{align} \label{eq:DefccM}
\mathcal{M}(x_N) = \frac{e^{-Bx_N}-e^{-Ax_N}}{B-A}. 
\end{align}
Then, by comparing \eqref{eq:ODEsolform0} and \eqref{eq:ODEsolform}, we get 
\begin{align} \label{eq:RelateODEsf}
\alpha_{\ell d}^0=-\frac{\mu_{\ell d}}{B-A}, \quad 
\beta_{\ell d}^0=\frac{\mu_{\ell d}}{B-A}+\beta_{\ell d}, \quad 
\alpha_{\ell N}^0=-\frac{\mu_{\ell N}}{B-A}, \quad 
\beta_{\ell N}^0=\frac{\mu_{\ell N}}{B-A}+\beta_{\ell N}. 
\end{align}
And thus, \eqref{eq:AlgEq0} together with $b=a+(B-A)\mathcal{M}(\delta)$ yields
\begin{align} \label{eq:AlgEqM} % Algebraic Equation in Matrix form. 
\fb{L}\left[
\begin{array}{r}
\mu_{1N} \\ \beta_{1N} \\ \mu_{2N} \\ \beta_{2N}
\end{array}
\right] = \left[
\begin{array}{c}
\mu^{-1}{\widehat{h}_d(\xi',\delta)} \\ -\mu^{-1}{\widehat{h}_d(\xi',0)} \\ \mu^{-1}A{\widehat{h}_N(\xi',\delta)} \\ -\mu^{-1}A{\widehat{h}_N(\xi',0)}
\end{array}
\right] 
\end{align}
where $(i,j)$ component $L_{ij}$ of $\fb{L}$ is given by  
\begin{align} \begin{aligned} \label{eq:compL} 
L_{11} &= -(B+A), & L_{12} &= -D_0, \\ 
L_{13} &= -(B+A)a-D_0{\mathcal{M}(\delta)}, & L_{14} &= -D_0{a}-D_0(B-A)\mathcal{M}(\delta), \\
L_{21} &= -(B+A)a-D_0{\mathcal{M}(\delta)}, & L_{22} &= -D_0{a}-D_0(B-A)\mathcal{M}(\delta), \\ 
L_{23} &= -(B+A), & L_{24} &= -D_0, \\
L_{31} &= -(B-A), & L_{32} &= 2AB, \\ 
L_{33} &= (B-A)a-2AB{\mathcal{M}(\delta)}, & L_{34} &= -2AB{a}-2AB(B-A)\mathcal{M}(\delta), \\
L_{41} &= -(B-A)a+2AB{\mathcal{M}(\delta)}, & L_{42} &= 2AB{a}+2AB(B-A)\mathcal{M}(\delta), \\ 
L_{43} &= B-A, & L_{44} &= -2AB, 
\end{aligned} \end{align}
where $A$ and $B$ are defined in \eqref{eq:DefAB}, 
$a$ and $D_0$ are defined by \eqref{eq:DefabD0}, and $\mathcal{M}(\delta)$ is given by \eqref{eq:DefccM} with $x_N=\delta$. 
Also, the coefficient $\gamma_\ell$ can be obtained as
\begin{align} \label{eq:gammalfmu} %gamma_l from mu
\gamma_\ell=(-1)^{\ell-1}\frac{\mu(B+A)}{A}\mu_{\ell N} 
\end{align}
from \eqref{eq:ODErelation} and \eqref{eq:RelateODEsf}. 

%%%%%%%%%%%%%%%%%%%%%%%%%%%%%%%%%%%%%%%%%%%%%%%%%%%%%%%%%%%%%%%%%%%%%
% solution formula
%%%%%%%%%%%%%%%%%%%%%%%%%%%%%%%%%%%%%%%%%%%%%%%%%%%%%%%%%%%%%%%%%%%%%
%%%%%%%%%%%%%%%%%%%%%%%%%%%%%% u_N, p

Thus, if we solve \eqref{eq:AlgEqM}, we obtain the solution formula of $u_N$ and $\theta$ 
by \eqref{eq:ODEsolform} and \eqref{eq:gammalfmu} as follows: 
\begin{align}  \label{eq:solform}
\begin{aligned} 
u_N(x) &= \sum_{\substack{k=1}}^4\sum_{\ell=1}^2 \mathscr{F}_{\xi'}^{-1} \left[ \left\{
\frac{L_{k,2\ell-1}}{\det{\fb{L}}}{\mathcal{M}}(d_\ell(x_N)) + \frac{L_{k,2\ell}}{\det{\fb{L}}}e^{-Bd_\ell(x_N)}
\right\} r_k \right](x'), \\ 
\theta(x) &= \sum_{\substack{k=1}}^4\sum_{\ell=1}^2 \mathscr{F}_{\xi'}^{-1} \left[ \left\{ 
(-1)^{\ell-1} \frac{\mu(B+A)}{A}\frac{L_{k,2\ell-1}}{\det{\fb{L}}}e^{-Ad_\ell(x_N)}
\right\} r_k \right](x'), 
\end{aligned}
\end{align} 
where $L_{i,j}$ denotes the $(i,j)$ cofactor of $\fb{L}$ and $r_k$ stands for the right-hand side: 
\begin{align*}
r_1=\mu^{-1}{\widehat{h}_d(\xi',\delta)}, \quad r_2= -\mu^{-1}{\widehat{h}_d(\xi',0)}, \quad 
r_3= \mu^{-1}A{\widehat{h}_N(\xi',\delta)}, \quad r_4=-\mu^{-1}A{\widehat{h}_N(\xi',0)}.
\end{align*} 
Here, the determinant of $\fb{L}$ is given by
\begin{align} \label{eq:detL}
\qquad \det \fb{L}=\frac{1}{(B-A)^2} \prod_{+,-} \{ (B^2+A^2)^2(1\pm a)(1\mp b) - 4A^3B(1\mp a)(1\pm b) \} \qquad 
\end{align}
and the cofactor $L_{i,j}$ is
\begin{align} \label{eq:cofacL} 
L_{1,1} &= -2AB{D_3}(1-a^2) - 16A^4B^2a{\mathcal{M}(\delta)}{}- 2AB(B^2-A^2)D_2{\mathcal{M}(\delta)}^2, \\ \notag 
L_{1,2} &= -(B-A)D_3(1-a^2) - 8A^3B(B-A)a{\mathcal{M}(\delta)}{}+ 2AB(B+A)D_2{\mathcal{M}(\delta)}^2, \\ \notag 
L_{1,3} &= 2AB{D_3}{a}(1-a^2) + 4AB(D_0^2-(B-A)D_3{a^2})\mathcal{M}(\delta){}- 2AB(B-A)^2D_3{a}\mathcal{M}(\delta)^2, \\ \notag 
L_{1,4} &= (B-A)D_3{a}(1-a^2) - (D_0{D_1}+(B^2-4AB+A^2)D_3{a^2})\mathcal{M}(\delta){}+ 2AB(B-A)D_3 a\mathcal{M}(\delta)^2, \\ \notag 
L_{2,1} &= 2AB{D_3}{a}(1-a^2) + 4AB(D_0^2-(B-A)D_3{a^2})\mathcal{M}(\delta){}- 2AB(B-A)^2D_3{a}\mathcal{M}(\delta)^2, \\ \notag 
L_{2,2} &= (B-A)D_3{a}(1-a^2) - (D_0{D_1}+(B^2-4AB+A^2)D_3{a^2})\mathcal{M}(\delta){}+ 2AB(B-A)D_3 a\mathcal{M}(\delta)^2, \\ \notag 
L_{2,3} &= -2AB{D_3}(1-a^2) - 16A^4B^2a{\mathcal{M}(\delta)}{}- 2AB(B^2-A^2)D_2{\mathcal{M}(\delta)}^2, \\ \notag 
L_{2,4} &= -(B-A)D_3(1-a^2) - 8A^3B(B-A)a{\mathcal{M}(\delta)}{}+ 2AB(B+A)D_2{\mathcal{M}(\delta)}^2, \\ \notag 
L_{3,1} &= -D_0{D_3}(1-a^2) + 2D_0^3a{\mathcal{M}(\delta)}{}+ (B^2-A^2)D_0{D_2}{\mathcal{M}(\delta)}^2, \\ \notag 
L_{3,2} &= (B+A)D_3(1-a^2) - 2(B+A)D_0^2a{\mathcal{M}(\delta)}{}- (B+A)D_0{D_2}{\mathcal{M}(\delta)}^2, \\ \notag 
L_{3,3} &= -D_0{D_3}{a}(1-a^2) + 2D_0(4A^3B+(B-A)D_3{a^2})\mathcal{M}(\delta){}+ (B-A)^2D_0D_3 a\mathcal{M}(\delta)^2, \\ \notag 
L_{3,4} &= (B+A)D_3{a}(1-a^2) - 2(A^2D_1+B^2D_3{a^2})\mathcal{M}(\delta){}- (B-A)D_0{D_3} a\mathcal{M}(\delta)^2, \\ \notag 
L_{4,1} &= D_0{D_3}{a}(1-a^2) - 2D_0(4A^3B+(B-A)D_3{a^2})\mathcal{M}(\delta){}- (B-A)^2D_0D_3 a\mathcal{M}(\delta)^2, \\ \notag 
L_{4,2} &= -(B+A)D_3{a}(1-a^2) + 2(A^2D_1+B^2D_3{a^2})\mathcal{M}(\delta){}+ (B-A)D_0{D_3} a\mathcal{M}(\delta)^2, \\ \notag 
L_{4,3} &= D_0{D_3}(1-a^2) - 2D_0^3a{\mathcal{M}(\delta)}{}- (B^2-A^2)D_0{D_2}{\mathcal{M}(\delta)}^2, \\ \notag 
L_{4,4} &= -(B+A)D_3(1-a^2) + 2(B+A)D_0^2a{\mathcal{M}(\delta)}{}+ (B+A)D_0{D_2}{\mathcal{M}(\delta)}^2, 
\end{align}
where $A$ and $B$ are defined by \eqref{eq:DefAB}, 
$a$, $b$ and $D_0$ are defined by \eqref{eq:DefabD0}, $\mathcal{M}(\delta)$ is given by \eqref{eq:DefccM} with $x_N=\delta$ and  
\begin{align} \label{eq:defDis}
\begin{aligned}
D_1&=-B^3+AB^2+3A^2B+A^3, \quad D_2=B^3-AB^2+3A^2B+A^3, \\ 
D_3&=B^3+AB^2+3A^2B-A^3. 
\end{aligned}
\end{align}

\begin{Rem} \label{Rem:highsingular}
The singularity appearing in the symbol of the solution formula \eqref{eq:solform} 
is higher than that for Neumann-Dirichlet boundary condition. 
In fact, the third and fourth rows of $\fb{L}$ coincide if $\xi'=0$ and, thereby, $\det{\fb{L}}\to0$ as $\xi'\to0$. 
This is because 
the upper and lower boundary conditions, which cause the third and fourth rows, respectively, are the same. 
In contrast to that, any two rows of $\fb{L}$ do not coincide when $\xi'=0$ and $\det{\fb{L}}\not\to 0$ as $\xi'\to0$ 
in the case of Neumann-Dirichlet boundary condition. 
\end{Rem}

%%%%%%%%%%%%%%%%%%%%%%%%%%%%%% u_j
Once we have $u_N$ and $\theta$, we can obtain $u_j$ with $j=1,\cdots,N-1$ 
from the $j$-th component of the first equation and the boundary condition in \eqref{eq:SRdh}, 
that is, as the solution of the following equation: 
\begin{align} \label{eq:PDEuj} % PDE for u_j
\left\{
\begin{array}{ll}
\lambda u_j - \mu \Delta u_j = -\partial_j \theta& \text{in } \Omega, \\
\partial_N{u}_j = \mu^{-1}{\nu}_N{h}_j-\partial_j{u}_N & \text{on } \partial\Omega. 
\end{array}
\right.
\end{align}

We first prove the $\mathscr{R}$-boundedness for $\mathcal{S}_N(\lambda)$ and $\mathcal{T}(\lambda)$
by analyzing solution formulas \eqref{eq:solform} with the aid of lemmas in Section 5. 
By combining those with \eqref{eq:PDEuj},  
we then show the $\mathscr{R}$-boundedness for $\mathcal{S}_j(\lambda)$ with $j=1,\cdots,N-1$, 
which is performed in Section 6.

%%%%%%%%%%%%%%%%%%%%%%%%%%%%%%%%%%%%%%%%%%%%%%%%%%%%%%%%%%
%%%%%%%%%%%%%%%%%%%%%%%%%%%%%%%%%%%%%%%%%%%%%%%%%%%%%%%%%%
%%														%%
%%	Sec.5												%%
%%														%%
%%%%%%%%%%%%%%%%%%%%%%%%%%%%%%%%%%%%%%%%%%%%%%%%%%%%%%%%%%
%%%%%%%%%%%%%%%%%%%%%%%%%%%%%%%%%%%%%%%%%%%%%%%%%%%%%%%%%%
\section{Technical lemmas} \label{sec:lem}

%%%%%%%%%%%%%%%%%%%%%%%%%%%%%%%%%%%%%%%%%%%%%%%%%%%%%%%%%%%%%%%
%%%%%%%%%%%%%%%%%%%%%%%%%%%%%%%%%%%%%%%%%%%%%%%%%%%%%%%%%%%%%%%
% Multiplier Class
%%%%%%%%%%%%%%%%%%%%%%%%%%%%%%%%%%%%%%%%%%%%%%%%%%%%%%%%%%%%%%%
%%%%%%%%%%%%%%%%%%%%%%%%%%%%%%%%%%%%%%%%%%%%%%%%%%%%%%%%%%%%%%%

%%%%%%%%%%%%%%%%%%%%%%%%%%%%%%%%%%%%%%%%%%%%%%%%%%%%%%%%%%%%%%%
% Intro and multiplier of order s with type \ell
%%%%%%%%%%%%%%%%%%%%%%%%%%%%%%%%%%%%%%%%%%%%%%%%%%%%%%%%%%%%%%%
In this section, we introduce some lemmas to establish Theorem \ref{Thm:Rbdddh}. %the $\mathscr{R}$-boundedness. 
To this end, first of all, we introduce some classes of symbols. 
Let $m(\lambda,\xi')$ be a function defined on $\Sigma_{\varepsilon,\gamma_0}\times(\mathbb{R}^{N-1}\setminus\{0\})$ 
with $0<\varepsilon<\pi/2$ and $\gamma_0\geq0$. 
For all $s \in \mathbb{R}$ and $k=1, 2$,  
$m(\lambda, \xi')$ is said to be a multiplier of order $s$ with type $k$ 
if it satisfies the following two conditions: 
\begin{itemize}
\item[i)] 
$m(\lambda,\xi')$ is infinitely many times differentiable with respect to $\xi'$ and $\tau$ where $\lambda = \gamma+i\tau$. 
\item[ii)] 
For any $\ell=0,1$, multi-index $\alpha' \in \mathbb{N}_0^{N-1}$ and $(\lambda,\xi') \in \Sigma_{\varepsilon,\gamma_0}\times(\mathbb{R}^{N-1}\setminus\{0\})$, 
\begin{align}
\left|(\tau{\partial}_\tau)^\ell \partial_{\xi'}^{\alpha'} m(\lambda,\xi')\right| \leq C_{\alpha'}\left(|\lambda|^{1/2}+|\xi'|\right)^{s-|\alpha'|} 
\end{align}
when $k=1$, and 
\begin{align}
\left|(\tau{\partial}_\tau)^\ell \partial_{\xi'}^{\alpha'} m(\lambda,\xi')\right| \leq C_{\alpha'}\left(|\lambda|^{1/2}+|\xi'|\right)^s|\xi'|^{-|\alpha'|} 
\end{align}
when $k=2$. 
\end{itemize}
We set 
\begin{align} \label{eq:DefbbM}
\mathbb{M}_{s,k,\varepsilon,\gamma_0} = \{ m(\lambda,\xi') \mid m(\lambda,\xi') \text{ is a multiplier of order $s$ with type $k$} \}. 
\end{align}
By the definition and the Leibniz rule, we have the following properties (see \cite{ShibaShimi12MSJ}). 

%%%%%%%%%%%%%%%%%%%%%%%%%%%%%%%%%%%%%%%%%%%%%%%%%%%%%%%%%%%%%%%
% Lemma: MultiplierClass
%%%%%%%%%%%%%%%%%%%%%%%%%%%%%%%%%%%%%%%%%%%%%%%%%%%%%%%%%%%%%%%
\begin{Lem} \label{Lem:MultiplierClass}
Let $s, s_1, s_2\in{\mathbb{R}}$, $k_1,k_2=1,2$, $0<\varepsilon<\pi/2$ and $\gamma_0\geq0$. 
\begin{itemize}
\item[a)] We have $\mathbb{M}_{s,1,\varepsilon,\gamma_0} \subset \mathbb{M}_{s,2,\varepsilon,\gamma_0}$. 
\item[b)] Given $m_j \in \mathbb{M}_{s_j,k_j,\varepsilon,\gamma_0}$ ($j=1,2$), 
we have $m_1m_2 \in \mathbb{M}_{s_1+s_2,\max\{k_1,k_2\},\varepsilon,\gamma_0}$. 
\end{itemize}
\end{Lem}

%%%%%%%%%%%%%%%%%%%%%%%%%%%%%%%%%%%%%%%%%%%%%%%%%%%%%%%%%%%%%%%
% Lemma: EstiSymbol
%%%%%%%%%%%%%%%%%%%%%%%%%%%%%%%%%%%%%%%%%%%%%%%%%%%%%%%%%%%%%%%
Let $A$, $B$ and $\mathcal{M}(x_N)$ be given by \eqref{eq:DefAB} and \eqref{eq:DefccM}. We set 
\begin{align} \label{eq:Defki}
k_i(x_N) = \left\{
\begin{array}{ll}
e^{-Bx_N}, & i=1, \\ 
e^{-Ax_N}, & i=2, \\ 
B{\mathcal{M}}(x_N), & i=3, 
\end{array}
\right.
\quad (0<x_N<\delta). 
\end{align}

\begin{Lem} \label{Lem:EstiSymbol}
Let $0<\varepsilon<\pi/2$ and $s \in \mathbb{R}$. 
For $x_N, y_N \in (0,\delta)$, $i,i_1,i_2=1,2,3$, $\ell=0,1$, 
$(\lambda,\xi') \in \Sigma_{\varepsilon,0} \times (\mathbb{R}^{N-1}\setminus\{0\})$ 
and multi-index $\alpha'\in{\mathbb{N}}_0^{N-1}$, 
we have the following estimates: 
\begin{gather} 
\begin{gathered} \label{esti:equivB}
c_{\varepsilon,\mu}(|\lambda|^{1/2}+A)\leq \text{Re\,} B \leq |B| \leq C_\mu(|\lambda|^{1/2}+A), \\ 
c_{\varepsilon,\mu}(|\lambda|^{1/2}+A)^3\leq |D_3| \leq C_\mu(|\lambda|^{1/2}+A)^3, 
\end{gathered} \\ 
\begin{gathered} \label{esti:multiclassAB}
|(\tau{\partial}_\tau)^\ell{\partial_{\xi'}^{\alpha'}}A^s| \leq C_{\alpha',s}A^{s-|\alpha'|}, \quad 
B^s \in \mathbb{M}_{s,1,\varepsilon,0}, \quad D_3^s \in \mathbb{M}_{3s,2,\varepsilon,0}, 
\end{gathered} \\ 
\begin{aligned} \label{esti:multiclasski}
&|(\tau{\partial}_\tau)^\ell{\partial_{\xi'}^{\alpha'}}{k_1(x_N)}| 
\leq C_{\alpha',\varepsilon,\mu}(|\lambda|^{1/2}+A)^{-|\alpha'|}e^{-c_{\varepsilon,\mu}(|\lambda|^{1/2}+A)x_N}, \\ 
&|(\tau{\partial}_\tau)^\ell{\partial_{\xi'}^{\alpha'}}k_i(x_N)| 
\leq C_{\alpha',\varepsilon,\mu}A^{-|\alpha'|}e^{-c_{\varepsilon,\mu}Ax_N}, 
\end{aligned} \\
\begin{aligned} \label{esti:multiclasskis}
&|(\tau{\partial}_\tau)^\ell{\partial_{\xi'}^{\alpha'}}{k_1(x_N)k_1(y_N)}| 
\leq C_{\alpha',\varepsilon,\mu}(|\lambda|^{1/2}+A)^{-|\alpha'|}e^{-c_{\varepsilon,\mu}(|\lambda|^{1/2}+A)(x_N+y_N)}, \\ 
&|(\tau{\partial}_\tau)^\ell{\partial_{\xi'}^{\alpha'}}k_{i_1}(x_N)k_{i_2}(y_N)| 
\leq C_{\alpha',\varepsilon,\mu}A^{-|\alpha'|}e^{-c_{\varepsilon,\mu}A(x_N+y_N)}. 
\end{aligned}
\end{gather}
Here, $D_3=B^3+AB^2+3A^2B-A^3$, which is defined in \eqref{eq:defDis}. 
\end{Lem}
\begin{proof}
The estimates \eqref{esti:equivB}, \eqref{esti:multiclassAB} and \eqref{esti:multiclasski} are proved 
in \cite[Lemma 4.4]{ShibaShimi03DIE}, \cite[Lemma 5.2]{ShibaShimi12MSJ}, \cite[Lemma 5.3]{ShibaShimi12MSJ}, respectively. 
The inequality \eqref{esti:multiclasskis} is obtained 
by  \eqref{esti:multiclasski} and the Leibniz rule as in \cite[Lemma 5.2]{Sai15MMAS}. 
\end{proof}

%%%%%%%%%%%%%%%%%%%%%%%%%%%%%%%%%%%%%%%%%%%%%%%%%%%%%%%%%%%%%%%
%%%%%%%%%%%%%%%%%%%%%%%%%%%%%%%%%%%%%%%%%%%%%%%%%%%%%%%%%%%%%%%
% Lemma for Rbdd
%%%%%%%%%%%%%%%%%%%%%%%%%%%%%%%%%%%%%%%%%%%%%%%%%%%%%%%%%%%%%%%
%%%%%%%%%%%%%%%%%%%%%%%%%%%%%%%%%%%%%%%%%%%%%%%%%%%%%%%%%%%%%%%
In view of Lemma \ref{Lem:Rbddprop} ii), 
we prove the $\mathscr{R}$-boundedness for each term of the solution formula \eqref{eq:solform}.
%%%%%%%%%%%%%%%%%%%%%%%%%%%%%%%%%%%%%%%%%%%%%%%%%%%%%%%%%%%%%%%
% set cut-off function and decay term
%%%%%%%%%%%%%%%%%%%%%%%%%%%%%%%%%%%%%%%%%%%%%%%%%%%%%%%%%%%%%%%
Let $\varphi_0$ and $\varphi_\delta$ be given by \eqref{eq:varphi0d}, and define 
\begin{align} \label{eq:DefCO}
\Phi_i(y_N) = \left\{
\begin{array}{ll}
\varphi_\delta(y_N), & i=1, \\
\varphi_0(y_N), & i=2, \\
\varphi_0'(y_N)=-\varphi_\delta'(y_N), & i=3. 
\end{array}
\right. 
\end{align}
In the same fashion as in \cite{Sai15MMAS}, we have the following lemma. 
The symbols of the operators and the range of $\gamma_0$ are slightly generalized as compared with those in \cite{Sai15MMAS}. 

%%%%%%%%%%%%%%%%%%%%%%%%%%%%%%%%%%%%%%%%%%%%%%%%%%%%%%%%%%%%%%%
% Lemma for R-bdd (usual case)
%%%%%%%%%%%%%%%%%%%%%%%%%%%%%%%%%%%%%%%%%%%%%%%%%%%%%%%%%%%%%%
\begin{Lem}[{\cite[Lemma 5.3]{Sai15MMAS}}] \label{Lem:RbddV} % R-b'dd after Volevich 
Let $0<\varepsilon<\pi/2$, $\gamma_0\geq0$ and $m_k \in \mathbb{M}_{0,k,\varepsilon,\gamma_0}$($k=1, 2$). 
We define the operators  
\begin{align} \notag 
[K_1(\lambda)h](x)&=
\int_0^\delta \mathscr{F}_{\xi'}^{-1}\left[ \Phi_i(y_N) m_1(\lambda,\xi') B k_1(d_{\ell_1}(x_N))k_1(d_{\ell_2}(y_N)) \widehat{h}(\xi',y_N) \right](x') \, dy_N, \\ 
\label{eq:DefK2}
[K_2(\lambda)h](x)&=
\int_0^\delta \mathscr{F}_{\xi'}^{-1}\left[ \Phi_i(y_N) m_2(\lambda,\xi') A k_{i_1}(d_{\ell_1}(x_N)) k_{i_2}(d_{\ell_2}(y_N)) \widehat{h}(\xi',y_N) \right](x') \, dy_N 
\end{align}
for $\lambda\in\Sigma_{\varepsilon,\gamma_0}$, $i,i_1,i_2=1,2,3$ and $\ell_1,\ell_2=1,2$, 
where $d_\ell(x_N)$ is given by \eqref{eq:Defdl}. 
Then, for any $1<q<\infty$ and $\ell=0,1$, there hold 
\begin{align}
\mathscr{R}_{\mathcal{L}(L_q(\Omega))}(\{ (\tau{\partial}_\tau)^\ell K_1(\lambda) \mid \lambda\in{\Sigma_{\varepsilon,\gamma_0}} \})&\leq C_{N,q,\varepsilon,\gamma_0,m_1}, \\ 
\mathscr{R}_{\mathcal{L}(L_q(\Omega))}(\{ (\tau{\partial}_\tau)^\ell K_2(\lambda) \mid \lambda\in{\Sigma_{\varepsilon,\gamma_0}} \})&\leq C_{N,q,\varepsilon,\gamma_0,m_2}. 
\end{align} 
Here, $\Sigma_{\varepsilon,\gamma_0}$ is given by \eqref{eq:resoset} and $\lambda=\gamma+i\tau$. 
\end{Lem}

%%%%%%%%%%%%%%%%%%%%%%%%%%%%%%%%%%%%%%%%%%%%%%%%%%%%%%%%%%%%%%%
% Lemma: Rbdd after Volevich for higher Singularity
%%%%%%%%%%%%%%%%%%%%%%%%%%%%%%%%%%%%%%%%%%%%%%%%%%%%%%%%%%%%%%%
In addition, by the argument in \cite[Lemma 5.5]{Sai15MMAS}, 
we have the lemma to treat operator families with higher singularity at the origin in the Fourier side.
As in Lemma \ref{Lem:RbddV}, the symbol of the operator and the range of $\gamma_0$ can be slightly generalized. 
\begin{Lem}[{\cite[Lemma 5.5]{Sai15MMAS}}] \label{Lem:RbddVS} 
Let $0<\varepsilon<\pi/2$ and $\gamma_0\geq 0$.  
Recall that $\mathbb{M}_{s,k,\varepsilon,\gamma_0}$, $\Phi_i(y_N)$, $k_i(x_N)$, $d_\ell(x_N)$ and $\Sigma_{\varepsilon,\gamma_0}$ 
are given by \eqref{eq:DefbbM}, \eqref{eq:DefCO}, \eqref{eq:Defki}, \eqref{eq:Defdl} and \eqref{eq:resoset}, respectively, 
and write $\lambda=\gamma+i\tau$ for $\lambda\in{\Sigma_{\varepsilon,\gamma_0}}$. 
Assume that $m_2 \in \mathbb{M}_{0,2,\varepsilon,\gamma_0}$. 
We define the operator  
\begin{align} \label{eq:DefK3}
[K_3(\lambda) h](x) &= 
\int_0^\delta \mathscr{F}_{\xi'}^{-1}\left[ \Phi_i(y_N) m_2(\lambda,\xi') k_{i_1}(d_{\ell_1}(x_N)) k_{i_2}(d_{\ell_2}(y_N)) \widehat{h}(\xi',y_N) \right](x') \, dy_N 
\end{align}
for $\lambda\in\Sigma_{\varepsilon,\gamma_0}$, $i,i_1,i_2=1,2,3$ and $\ell_1,\ell_2=1,2$. 
Then, for any $1<q<\infty$ and $\ell =0,1$, we have 
\begin{align}
\mathscr{R}_{\mathcal{L}(L_q(\Omega))}(\{ (\tau{\partial}_\tau)^\ell K_3(\lambda) \mid \lambda\in{\Sigma_{\varepsilon,\gamma_0}} \}){\leq C_{N,q,\varepsilon,\gamma_0,\delta,m_2}}. 
\end{align}
\end{Lem}

%%%%%%%%%%%%%%%%%%%%%%%%%%%%%%%%%%%%%%%%%%%%%%%%%%%%%%%%%
% Remark How to treat the higher singularity
%%%%%%%%%%%%%%%%%%%%%%%%%%%%%%%%%%%%%%%%%%%%%%%%%%%%%%%%%
\begin{Rem} \label{Rem:TreatHS}
As mentioned in Remark \ref{Rem:highsingular}, 
our symbol possesses higher singularity at $\xi'=0$ than the case of Neuman-Dirichlet boundary condition. 
Following \cite{ShibaShimi12MSJ}, which studies the case of the half space, as well as \cite[Lemma 5.3]{Sai15MMAS}, 
we recall the idea of the proof of Lemma \ref{Lem:RbddV} (lower singularity case).  
We rewrite $K_j(\lambda) \ (j=1,2)$ as 
\begin{align} \label{eq:SIinRem}
K_j(\lambda)h = \int_\Omega k_\lambda^j(x'-y',x_N,y_N)h(y',y_N) \,dy 
\end{align}
where 
\begin{align*}
k_\lambda^j(z',x_N,y_N)=\left\{
\begin{aligned}
&\mathscr{F}_{\xi'}^{-1}\left[ \Phi_i(y_N) m_1(\lambda,\xi') B k_1(d_{\ell_1}(x_N))k_1(d_{\ell_2}(y_N)) \right](z') && (j=1), \\
&\mathscr{F}_{\xi'}^{-1}\left[ \Phi_i(y_N) m_2(\lambda,\xi') A k_{i_1}(d_{\ell_1}(x_N)) k_{i_2}(d_{\ell_2}(y_N))\right](z') && (j=2). 
\end{aligned}
\right.
\end{align*}
Then we have 
\begin{align} \label{eq:EstkerLj}
|k_\lambda^j(z',x_N,y_N)|\leq C|(z',x_N,y_N)|^{-N}, 
\end{align}
which implies desired estimate, see the proof of \cite[Lemma 5.3]{Sai15MMAS} for details. 
In contrast to that, for the higher singularity case, 
we only know that the kernel $k_\lambda^3(z',x_N,y_N)$ of the operator $K_3(\lambda)$, 
which is defined similarly, decays with slower rate 
\begin{align}
|k_\lambda^3(z',x_N,y_N)|\leq C|(z',x_N,y_N)|^{-(N-1)}. 
\end{align}
However, we can estimate the operator families $\{K_3(\lambda)\}_\lambda$ by applying Lemma \ref{Lem:Rbddw} below 
only to the tangential direction; 
the lemma is proved by \cite[Theorem 3.2]{GirWei03PAM}, see also \cite[Theorem 3.3]{EnoShiba13FE}. 
This is performed in the proof of Lemma \ref{Lem:RbddVS}. 
\end{Rem}

%%%%%%%%%%%%%%%%%%%%%%%%%%%%%%%%%%%%%%%%%%%%%%%%%%%%%%%%%
% Lem Rbdd on wsp
%%%%%%%%%%%%%%%%%%%%%%%%%%%%%%%%%%%%%%%%%%%%%%%%%%%%%%%%%
\begin{Lem} 
\label{Lem:Rbddw}
Let $\Lambda$ be a set and $\{m_\lambda\}_{\lambda\in\Lambda}\subset C^\infty(\mathbb{R}^d\setminus\{0\})$ satisfy 
\begin{align}
|\partial_\xi^\alpha m_\lambda(\xi)|\leq C_\alpha |\xi|^{-|\alpha|} 
\end{align}
for any $\lambda\in\Lambda$, $\xi\in{\mathbb{R}^d\setminus\{0\}}$ and multi-index $\alpha\in{\mathbb{N}}_0^d$. 
Then, for $q\in(1,\infty)$, 
the operator family $\{T_\lambda\mid\lambda\in\Lambda\}$ given by $T_\lambda f=\mathscr{F}^{-1}_\xi m_\lambda{\mathscr{F}}_x f$ satisfies 
\begin{align}
\mathscr{R}_{\mathcal{L}(L_q(\mathbb{R}^d))}(\{T_\lambda\mid\lambda\in\Lambda\}) \leq C_{N,q} \max_{|\alpha|\leq N+2}C_\alpha. 
\end{align}
\end{Lem}

%%%%%%%%%%%%%%%%%%%%%%%%%%%%%%%%%%%%%%%%%%%%%%%%%%%%%%%%%
% Proof of Rbdd after Volevich with higher Singularity
%%%%%%%%%%%%%%%%%%%%%%%%%%%%%%%%%%%%%%%%%%%%%%%%%%%%%%%%%
\begin{proof}[Proof of Lemma \ref{Lem:RbddVS}]
Following Saito \cite[Lemmma 5.5]{Sai15MMAS}, 
we prove the estimate \eqref{eq:DefRbdd} in the definition of the $\mathscr{R}$-boundedness with $p=q$. 
Let $m\in{\mathbb{N}}$, $\{\lambda_j\}_{j=1}^m\subset{\Sigma_{\varepsilon,\gamma_0}}$ and $\{h_j\}_{j=1}^m\subset{L_q(\Omega)}$, 
and let $\{r_j\}_{j=1}^m$ be a sequence of independent, symmetric and $\{\pm1\}$-valued random variables on $(0,1)$. 
Define the operator $K_{30}(\lambda;x_N,y_N)$ for $\lambda\in{\Sigma_{\varepsilon,\gamma_0}}$ and $x_N,y_N\in(0,\delta)$ by  
\begin{align*}
K_{30}(\lambda;x_N,y_N)h_0(x') 
= \mathscr{F}_{\xi'}^{-1}\left[ \Phi_i(y_N) m_2(\lambda,\xi') k_{i_1}(d_{\ell_1}(x_N)) k_{i_2}(d_{\ell_2}(y_N)) \widehat{h_0}(\xi') \right](x') 
\end{align*}
so that
\begin{align}
[K_3(\lambda)h](x) = \int_0^\delta K_{30}(\lambda;x_N,y_N)[h(\cdot,y_N)] (x') \,dy_N. 
\end{align}
Note $\Phi_i(y_N)$ is constant with respect to $(\lambda,\xi')$ and bounded with respect to $y_N$, 
and so, by \eqref{esti:multiclasskis} and Lemma \ref{Lem:MultiplierClass}, 
\begin{align}
|(\tau{\partial}_\tau)^\ell{\partial_{\xi'}^{\alpha'}}	\left(\Phi_i(y_N) m_2(\lambda,\xi') k_{i_1}(d_{\ell_1}(x_N)) k_{i_2}(d_{\ell_2}(y_N))\right)|
\leq C_{N,q,\varepsilon,\gamma_0}|\xi'|^{-|\alpha'|}
\end{align}
for $(\lambda,\xi')\in{\Sigma_{\varepsilon,\gamma_0}}\times(\mathbb{R}^{N-1}\setminus\{0\})$. 
Therefore, by Lemma \ref{Lem:Rbddw} and the H\"{o}lder inequality, we have 
\begin{align} \notag 
&\int_0^1 \bigg\| \sum_{j=1}^m r_j(u)K_3(\lambda_j)h_j \bigg\|_{L_q(\Omega)}^q \,du \\ \notag
&= \int_0^1 \int_0^\delta \int_{\mathbb{R}^{N-1}}\left| \int_0^\delta 
\sum_{j=1}^m r_j(u)K_{30}(\lambda_j;x_N,y_N)h_j(\cdot,y_N) \,dy_N \right|^q \,dx' \,dx_N \,du \\ \notag 
&\leq\delta^{q-1} \int_0^1 \int_0^\delta \int_{\mathbb{R}^{N-1}} \int_0^\delta 
\bigg| \sum_{j=1}^m r_j(u)K_{30}(\lambda_j;x_N,y_N)h_j(\cdot,y_N) \bigg|^q \,dy_N \,dx' \,dx_N \,du \\ \notag 
&=\delta^{q-1} \int_0^1 \int_0^\delta \int_0^\delta 
\bigg\| \sum_{j=1}^m r_j(u)K_{30}(\lambda_j;x_N,y_N)h_j(\cdot,y_N) \bigg\|_{L_q(\mathbb{R}^{N-1})}^q \,dy_N \,dx_N \,du \\ \notag 
&\leq C_{N,q,\varepsilon,\gamma_0}\delta^{q-1} \int_0^1 \int_0^\delta \int_0^\delta 
\bigg\| \sum_{j=1}^m r_j(u)h_j(\cdot,y_N) \bigg\|_{L_q(\mathbb{R}^{N-1})}^q \,dy_N \,dx_N \,du \\ \notag 
&= C_{N,q,\varepsilon,\gamma_0}\delta^q \int_0^1 
\bigg\| \sum_{j=1}^m r_j(u)h_j \bigg\|_{L_q(\Omega)}^q \,du, 
\end{align}
which completes the proof. 
\end{proof}

%%%%%%%%%%%%%%%%%%%%%%%%%%%%%%%%%%%%%%%%%%%%%%%%%%%%%%%%%%%%%%%
% Lemma: Derivatived Rbdd after Volevich
%%%%%%%%%%%%%%%%%%%%%%%%%%%%%%%%%%%%%%%%%%%%%%%%%%%%%%%%%%%%%%%
By the lemmas above, the same argument as in \cite[Lemma 5.4]{Sai15MMAS} implies the following lemma. 

\begin{Lem} \label{Lem:DRbddV} 
Let $0<\varepsilon<\pi/2$, $\gamma_0\geq 0$ and $m_k \in \mathbb{M}_{-2,k,\varepsilon,\gamma_0}$, 
and let $K_j(\lambda)$ ($j=1,2,3$) be the operators given in Lemmas \ref{Lem:RbddV} and \ref{Lem:RbddVS}. 
For any $1<q<\infty$, $j=1,2,3$, $\ell =0,1$ and $1\leq m,n \leq N$, there hold 
\begin{align}
\mathscr{R}_{\mathcal{L}(L_q(\Omega))}(\{ (\tau{\partial}_\tau)^\ell \lambda K_j(\lambda) \mid \lambda\in{\Sigma_{\varepsilon,\gamma_0}} \})&\leq C_{N,q,\varepsilon,\gamma_0,\mu,\delta,m_k}, \\ 
\mathscr{R}_{\mathcal{L}(L_q(\Omega))}(\{ (\tau{\partial}_\tau)^\ell \gamma K_j(\lambda) \mid \lambda\in{\Sigma_{\varepsilon,\gamma_0}} \})&\leq C_{N,q,\varepsilon,\gamma_0,\mu,\delta,m_k}, \\ 
\mathscr{R}_{\mathcal{L}(L_q(\Omega))}(\{ (\tau{\partial}_\tau)^\ell \lambda^{1/2} \partial_m K_j(\lambda) \mid \lambda\in{\Sigma_{\varepsilon,\gamma_0}} \})&\leq C_{N,q,\varepsilon,\gamma_0,\mu,\delta,m_k}, \\
\mathscr{R}_{\mathcal{L}(L_q(\Omega))}(\{ (\tau{\partial}_\tau)^\ell \partial_m{\partial}_n K_j(\lambda) \mid \lambda\in{\Sigma_{\varepsilon,\gamma_0}} \})&\leq C_{N,q,\varepsilon,\gamma_0,\mu,\delta,m_k}, 
\end{align}
where $k=1$ when $j=1$ and $k=2$ when $j=2,3$. 
If $i_1=2$ (that is, $k_{i_1}(d_{\ell_1}(x_N))=e^{-Ad_{\ell_1}(x_N)}$) in \eqref{eq:DefK2} and \eqref{eq:DefK3}, 
and if $m_2 \in \mathbb{M}_{-2,2,\varepsilon,\gamma_0}$ is replaced by 
\begin{align} \label{eq:DefAIM} 
m_2 \in A^{-1}{\mathbb{M}}_{0,2,\varepsilon,\gamma_0} = \{ A^{-1}m(\lambda,\xi') \mid m(\lambda,\xi') \in \mathbb{M}_{0,2,\varepsilon,\gamma_0} \}, 
\end{align}
then, for $1<q<\infty$, $j=2,3$, $\ell =0,1$ and $1\leq m \leq N$, 
we have
\begin{align}
\mathscr{R}_{\mathcal{L}(L_q(\Omega))}(\{ (\tau{\partial}_\tau)^\ell \partial_m K_j(\lambda) \mid \lambda\in{\Sigma_{\varepsilon,\gamma_0}} \})&\leq C_{N,q,\varepsilon,\gamma_0,\mu,\delta,m_2}. 
\end{align}
\end{Lem}

%%%%%%%%%%%%%%%%%%%%%%%%%%%%%%%%%%%%%%%%%%%%%%%%%%%%%%%%%%%%%%%
% Lemma: Rbdd 
%%%%%%%%%%%%%%%%%%%%%%%%%%%%%%%%%%%%%%%%%%%%%%%%%%%%%%%%%%%%%%%
Now, we state the lemma which plays a crucial role in showing Theorem \ref{Thm:Rbdddh} 
from the estimate of the symbol. 

\begin{Lem} \label{Lem:Rbdd}  
Let $0<\varepsilon<\pi/2$ and $\gamma_0>0$. 
Recall that $\mathbb{M}_{s,k,\varepsilon,\gamma_0}$, $A^{-1}{\mathbb{M}}_{0,2,\varepsilon,\gamma_0}$, A, $d_\ell(x_N)$, $k_i(x_N)$ and $\Sigma_{\varepsilon,\gamma_0}$ 
are given by \eqref{eq:DefbbM}, \eqref{eq:DefAIM}, \eqref{eq:DefAB}, \eqref{eq:Defdl}, \eqref{eq:Defki} and \eqref{eq:resoset}, respectively, 
and write $\lambda=\gamma+i\tau$ for $\lambda\in{\Sigma_{\varepsilon,\gamma_0}}$. 
Assume that $m_k \in \mathbb{M}_{-2,k,\varepsilon,\gamma_0}$ ($k=1,2$). 
For all $j=1,2,3$, $\lambda\in{\Sigma_{\varepsilon,\gamma_0}}$, $i_1=1,2,3$ and $\ell_1,\ell_2=1,2$, 
there exist operators 
\begin{align}
\mathcal{K}_1(\lambda)=\mathcal{K}^1_{1,\ell_1,\ell_2}(\lambda,m_1), \quad \mathcal{K}_2(\lambda)=\mathcal{K}^{i_1}_{2,\ell_1,\ell_2}(\lambda,m_2), \quad \mathcal{K}_3(\lambda)=\mathcal{K}^{i_1}_{3,\ell_1,\ell_2}(\lambda,m_2) 
\end{align}
(precisely, they are given by \eqref{eq:defL1}, \eqref{eq:defL2} and \eqref{eq:defL3}, respectively)  
satisfying $\mathcal{K}_j(\lambda) \in \mathcal{L}(L_q(\Omega)^{1+N},W^2_q(\Omega))$ for $1<q<\infty$ 
such that the following two assertions hold:  
\begin{itemize}
\item[a)] Given $h \in C_0^\infty(\mathbb{R}^N)$, we have the formula 
\begin{align} \label{eq:Lemopform1}
[\mathcal{K}_1(\lambda) (\lambda^{1/2}h,\nabla h)](x) = \mathscr{F}_{\xi'}^{-1}\left[ m_1(\lambda,\xi') B k_1(d_{\ell_1}(x_N)) \widehat{h}(\xi',d_{\ell_2}(0)) \right](x'), \\ \label{eq:Lemopform}
\begin{aligned} 
[\mathcal{K}_2(\lambda) (h,\nabla h)](x) &= \mathscr{F}_{\xi'}^{-1}\left[ m_2(\lambda,\xi') A k_{i_1}(d_{\ell_1}(x_N)) \widehat{h}(\xi',d_{\ell_2}(0)) \right](x'), \\
[\mathcal{K}_3(\lambda) (h,\nabla h)](x) &= \mathscr{F}_{\xi'}^{-1}\left[ m_2(\lambda,\xi') k_{i_1}(d_{\ell_1}(x_N)) \widehat{h}(\xi',d_{\ell_2}(0)) \right](x'). 
\end{aligned}
\end{align}
\item[b)]
For any $1<q<\infty$, $j=1,2,3$, $\ell =0,1$ and $1\leq m,n \leq N$, there hold 
\begin{align} \label{eq:RbddLjU}
\mathscr{R}_{\mathcal{L}(L_q(\Omega))}(\{ (\tau{\partial}_\tau)^\ell \lambda \mathcal{K}_j(\lambda) \mid \lambda\in{\Sigma_{\varepsilon,\gamma_0}} \})&\leq C_{N,q,\varepsilon,\gamma_0,\mu,\delta,m_k}, \\ 
\mathscr{R}_{\mathcal{L}(L_q(\Omega))}(\{ (\tau{\partial}_\tau)^\ell \gamma \mathcal{K}_j(\lambda) \mid \lambda\in{\Sigma_{\varepsilon,\gamma_0}} \})&\leq C_{N,q,\varepsilon,\gamma_0,\mu,\delta,m_k}, \\ 
\mathscr{R}_{\mathcal{L}(L_q(\Omega))}(\{ (\tau{\partial}_\tau)^\ell \lambda^{1/2} \partial_m \mathcal{K}_j(\lambda) \mid \lambda\in{\Sigma_{\varepsilon,\gamma_0}} \})&\leq C_{N,q,\varepsilon,\gamma_0,\mu,\delta,m_k}, \\ 
\mathscr{R}_{\mathcal{L}(L_q(\Omega))}(\{ (\tau{\partial}_\tau)^\ell \partial_m{\partial}_n \mathcal{K}_j(\lambda) \mid \lambda\in{\Sigma_{\varepsilon,\gamma_0}} \})&\leq C_{N,q,\varepsilon,\gamma_0,\mu,\delta,m_k}, 
\end{align}
where $k=1$ when $j=1$ and $k=2$ when $j=2,3$. 
\end{itemize}

Furthermore, if $m_2 \in A^{-1}{\mathbb{M}}_{0,2,\varepsilon,\gamma_0}$ instead of $m_2 \in \mathbb{M}_{-2,2,\varepsilon,\gamma_0}$ and if $i_1=2$, 
then, the operators $\mathcal{K}_2(\lambda)=\mathcal{K}^2_{2,\ell_1,\ell_2}(\lambda,m_2)$ and $\mathcal{K}_3(\lambda)=\mathcal{K}^2_{3,\ell_1,\ell_2}(\lambda,m_2)$ satisfy the following assertions: 
\begin{itemize}
\item[a)] 
Given $h \in C_0^\infty(\mathbb{R}^N)$, we have the formula \eqref{eq:Lemopform}. 
\item[b)] 
For any $1<q<\infty$, $j=2,3$, $\ell =0,1$ and $1\leq m \leq N$, there holds 
\begin{align} \label{eq:RbddLjP}
\mathscr{R}_{\mathcal{L}(L_q(\Omega))}(\{ (\tau{\partial}_\tau)^\ell \partial_m \mathcal{K}_j(\lambda) \mid \lambda\in{\Sigma_{\varepsilon,\gamma_0}} \})&\leq C_{N,q,\varepsilon,\gamma_0,\delta,m_2}. 
\end{align}
\end{itemize}
\end{Lem}

%%%%%%%%%%%%%%%%%%%%%%%%%%%%%%%%%%%%%%%%%%%%%%%%%%%%%%%%%
% Remark for Lem:Rbdd
%%%%%%%%%%%%%%%%%%%%%%%%%%%%%%%%%%%%%%%%%%%%%%%%%%%%%%%%%
\begin{Rem}
It might not look natural to deduce the $\mathscr{R}$-boundedness for \eqref{eq:Lemopform} in place of 
\begin{align}
[\mathcal{K}_2(\lambda) (\lambda^{1/2}h,\nabla h)](x) 
&= \mathscr{F}_{\xi'}^{-1}\left[ m_2(\lambda,\xi') A k_{i_1}(d_{\ell_1}(x_N)) \widehat{h}(\xi',d_{\ell_2}(0)) \right](x'), \\
[\mathcal{K}_3(\lambda) (\lambda^{1/2}h,\nabla h)](x) 
&= \mathscr{F}_{\xi'}^{-1}\left[ m_2(\lambda,\xi') k_{i_1}(d_{\ell_1}(x_N)) \widehat{h}(\xi',d_{\ell_2}(0)) \right](x').  
\end{align}
However, the $\mathscr{R}$-boundedness for \eqref{eq:Lemopform} is needed 
in order to estimate the pressure term in showing Theorem \ref{Thm:Rbdd} from Theorem \ref{Thm:Rbdddh}, 
see Remark \ref{Rem:Rbdddh}. 
On the other hand, there is no need to prove the $\mathscr{R}$-boundedness for 
\begin{align}
[\mathcal{K}_1(\lambda) (h,\nabla h)](x) = \mathscr{F}_{\xi'}^{-1}\left[ m_1(\lambda,\xi') B k_1(d_{\ell_1}(x_N)) \widehat{h}(\xi',d_{\ell_2}(0)) \right](x') 
\end{align}
since we use the operator $\mathcal{K}_1(\lambda)$ only in Lemma \ref{Lem:RbddHarmla} below 
to estimate the solution $u_j$ of \eqref{eq:PDEuj}. 
\end{Rem}

%%%%%%%%%%%%%%% 

\begin{proof}[Proof of Lemma \ref{Lem:Rbdd}]
This lemma is proved by a trick due to Volevich \cite{Vol65MS} and Lemma \ref{Lem:DRbddV}. 
We write $\varphi_{d_1(0)}=\varphi_\delta$ and $\varphi_{d_2(0)}=\varphi_0$, see \eqref{eq:Defdl}. 
For $h \in C_0^\infty(\mathbb{R}^N)$, we rewrite the right-hand side of \eqref{eq:Lemopform1} as 
\begin{align} \label{eq:deriveL1}
&\mathscr{F}_{\xi'}^{-1}\left[ m_1(\lambda,\xi') B k_1(d_{\ell_1}(x_N)) \widehat{h}(\xi',d_{\ell_2}(0)) \right]{}\\ \notag 
&= -\int_0^\delta \partial_{y_N} \mathscr{F}_{\xi'}^{-1}\left[ \varphi_{d_{\ell_2}(0)}(y_N) m_1(\lambda,\xi') B k_1(d_{\ell_1}(x_N)) e^{-B{d_{\ell_2}(y_N)}}{\widehat{h}(\xi',y_N)} \right]{}\, dy_N \\ \notag
&= -\int_0^\delta \mathscr{F}_{\xi'}^{-1}\left[ \varphi_{d_{\ell_2}(0)}'(y_N)  
m_1(\lambda,\xi') B k_1(d_{\ell_1}(x_N)) e^{-B{d_{\ell_2}(y_N)}} \widehat{h}(\xi',y_N) \right]{}\, dy_N \\ \notag
&\quad{+(-1)^{\ell_2}}\int_0^\delta \mathscr{F}_{\xi'}^{-1}\left[ \varphi_{d_{\ell_2}(0)}(y_N) 
m_1(\lambda,\xi') B^2 k_1(d_{\ell_1}(x_N)) e^{-B{d_{\ell_2}(y_N)}} \widehat{h}(\xi',y_N) \right]{}\, dy_N \\ \notag
&\quad-\int_0^\delta \mathscr{F}_{\xi'}^{-1}\left[ \varphi_{d_{\ell_2}(0)}(y_N) 
m_1(\lambda,\xi') B k_1(d_{\ell_1}(x_N)) e^{-B{d_{\ell_2}(y_N)}} \widehat{\partial_Nh}(\xi',y_N) \right]{}\, dy_N \\ \notag
&= -\frac{1}{\gamma_0^{1/2}}\int_0^\delta \mathscr{F}_{\xi'}^{-1}\left[ \varphi_{d_{\ell_2}(0)}'(y_N) \frac{\gamma_0^{1/2}}{\lambda^{1/2}} 
m_1(\lambda,\xi') B k_1(d_{\ell_1}(x_N)) e^{-B{d_{\ell_2}(y_N)}} \widehat{\lambda^{1/2}h}(\xi',y_N) \right]{}\, dy_N \\ \notag
&\quad{+(-1)^{\ell_2}}\int_0^\delta \mathscr{F}_{\xi'}^{-1}\left[ \varphi_{d_{\ell_2}(0)}(y_N) 
\frac{\lambda^{1/2}}{\mu B} m_1(\lambda,\xi') B k_1(d_{\ell_1}(x_N)) e^{-B{d_{\ell_2}(y_N)}} \widehat{\lambda^{1/2}h}(\xi',y_N) \right]{}\, dy_N \\ \notag
&\quad{-(-1)^{\ell_2}}\sum_{j'=1}^{N-1} \int_0^\delta \mathscr{F}_{\xi'}^{-1}\left[ \varphi_{d_{\ell_2}(0)}(y_N) 
\frac{i\xi_{j'}}{B} m_1(\lambda,\xi') B k_1(d_{\ell_1}(x_N)) e^{-B{d_{\ell_2}(y_N)}} \widehat{\partial_{j'}h}(\xi',y_N) \right]{}\, dy_N \\ \notag 
&\quad-\int_0^\delta \mathscr{F}_{\xi'}^{-1}\left[ \varphi_{d_{\ell_2}(0)}(y_N) 
m_1(\lambda,\xi') B k_1(d_{\ell_1}(x_N)) e^{-B{d_{\ell_2}(y_N)}} \widehat{\partial_Nh}(\xi',y_N) \right]{}\, dy_N, 
\end{align}
where we have used $B=B^2/B=\lambda/(\mu B)+\sum_{j'=1}^{N-1}(-i\xi_{j'}/B)i\xi_{j'}$. 
We thus define $\mathcal{K}_1(\lambda)$ by 
\begin{align} \label{eq:defL1}
\mathcal{K}_1(\lambda)(\lambda^{1/2}h,\nabla h) = (\text{the right-hand side of \eqref{eq:deriveL1}}). 
\end{align}
Similarly, by using the relation $A=A^2/A=\sum_{j'=1}^{N-1} (-i\xi_{j'}/A)i\xi_{j'}$, 
we rewrite the right-hand sides of \eqref{eq:Lemopform} as  
\begin{align} \label{eq:deriveL2}
&\mathscr{F}_{\xi'}^{-1}\left[ m_2(\lambda,\xi') A k_{i_1}(d_{\ell_1}(x_N)) \widehat{h}(\xi',d_{\ell_2}(0)) \right]{}\\ \notag 
&= -\int_0^\delta \partial_{y_N} \mathscr{F}_{\xi'}^{-1}\left[ \varphi_{d_{\ell_2}(0)}(y_N) m_2(\lambda,\xi') A k_{i_1}(d_{\ell_1}(x_N)) e^{-A{d_{\ell_2}(y_N)}}{\widehat{h}(\xi',y_N)} \right]{}\, dy_N \\ \notag
&= -\int_0^\delta \mathscr{F}_{\xi'}^{-1}\left[ \varphi_{d_{\ell_2}(0)}'(y_N)  
m_2(\lambda,\xi') A k_{i_1}(d_{\ell_1}(x_N)) e^{-A{d_{\ell_2}(y_N)}} \widehat{h}(\xi',y_N) \right]{}\, dy_N \\ \notag 
&\quad{-(-1)^{\ell_2}}\sum_{j'=1}^{N-1} \int_0^\delta \mathscr{F}_{\xi'}^{-1}\left[ \varphi_{d_{\ell_2}(0)}(y_N) 
\frac{i\xi_{j'}}{A} m_2(\lambda,\xi') A k_{i_1}(d_{\ell_1}(x_N)) e^{-A{d_{\ell_2}(y_N)}} \widehat{\partial_{j'}h}(\xi',y_N) \right]{}\, dy_N \\ \notag
&\quad-\int_0^\delta \mathscr{F}_{\xi'}^{-1}\left[ \varphi_{d_{\ell_2}(0)}(y_N) 
m_2(\lambda,\xi') A k_{i_1}(d_{\ell_1}(x_N)) e^{-A{d_{\ell_2}(y_N)}} \widehat{\partial_Nh}(\xi',y_N) \right]{}\, dy_N, 
\\ \label{eq:deriveL3}
&\mathscr{F}_{\xi'}^{-1}\left[ m_2(\lambda,\xi') k_{i_1}(d_{\ell_1}(x_N)) \widehat{h}(\xi',d_{\ell_2}(0)) \right]{}\\ \notag 
&= -\int_0^\delta \partial_{y_N} \mathscr{F}_{\xi'}^{-1}\left[ \varphi_{d_{\ell_2}(0)}(y_N) m_2(\lambda,\xi') k_{i_1}(d_{\ell_1}(x_N)) e^{-A{d_{\ell_2}(y_N)}}{\widehat{h}(\xi',y_N)} \right]{}\, dy_N \\ \notag
&= -\int_0^\delta \mathscr{F}_{\xi'}^{-1}\left[ \varphi_{d_{\ell_2}(0)}'(y_N)  
m_2(\lambda,\xi') k_{i_1}(d_{\ell_1}(x_N)) e^{-A{d_{\ell_2}(y_N)}} \widehat{h}(\xi',y_N) \right]{}\, dy_N \\ \notag
&\quad{-(-1)^{\ell_2}}\sum_{j'=1}^{N-1} \int_0^\delta \mathscr{F}_{\xi'}^{-1}\left[ \varphi_{d_{\ell_2}(0)}(y_N) 
\frac{i\xi_{j'}}{A} m_2(\lambda,\xi') k_{i_1}(d_{\ell_1}(x_N)) e^{-A{d_{\ell_2}(y_N)}} \widehat{\partial_{j'}h}(\xi',y_N) \right]{}\, dy_N \\ \notag
&\quad-\int_0^\delta \mathscr{F}_{\xi'}^{-1}\left[ \varphi_{d_{\ell_2}(0)}(y_N) 
m_2(\lambda,\xi') k_{i_1}(d_{\ell_1}(x_N)) e^{-A{d_{\ell_2}(y_N)}} \widehat{\partial_Nh}(\xi',y_N) \right]{}\, dy_N.  
\end{align}
Then we define 
\begin{align} \label{eq:defL2}
\mathcal{K}_2(\lambda)(h,\nabla h) = (\text{the right-hand side of \eqref{eq:deriveL2}}), \\ \label{eq:defL3} 
\mathcal{K}_3(\lambda)(h,\nabla h) = (\text{the right-hand side of \eqref{eq:deriveL3}}). 
\end{align}

If $m_1$ belongs to $\mathbb{M}_{-2,1,\varepsilon,\gamma_0}$, 
\begin{align}
\frac{\gamma_0^{1/2}}{\lambda^{1/2}} m_1(\lambda,\xi'), \ 
\frac{\lambda^{1/2}}{\mu B} m_1(\lambda,\xi'), \ 
\frac{i\xi_{j'}}{B} m_1(\lambda,\xi')
\end{align}
also belong to $\mathbb{M}_{-2,1,\varepsilon,\gamma_0}$ by \eqref{esti:multiclassAB} and Lemma \ref{Lem:MultiplierClass}. 
Similarly, if $m_2$ is in $\mathbb{M}_{-2,2,\varepsilon,\gamma_0}$ or $A^{-1}{\mathbb{M}_{0,2,\varepsilon,\gamma_0}}$, 
\begin{align}
\frac{i\xi_{j'}}{A} m_2(\lambda,\xi')
\end{align}
is also in $\mathbb{M}_{-2,2,\varepsilon,\gamma_0}$ or $A^{-1}{\mathbb{M}_{0,2,\varepsilon,\gamma_0}}$, respectively. 
Thus, the assertion for $K_j(\lambda)$ ($j=1,2,3$) in Lemma \ref{Lem:DRbddV} together with Lemma \ref{Lem:Rbddprop} ii) 
implies the conclusion for $\mathcal{K}_j(\lambda)$. 
\end{proof}

%%%%%%%%%%%%%%%%%%%%%%%%%%%%%%%%%%%%%%%%%%%%%%%%%%%%%%%%%%
%%%%%%%%%%%%%%%%%%%%%%%%%%%%%%%%%%%%%%%%%%%%%%%%%%%%%%%%%%
%%														%%
%%	Sec.6												%%
%%														%%
%%%%%%%%%%%%%%%%%%%%%%%%%%%%%%%%%%%%%%%%%%%%%%%%%%%%%%%%%%
%%%%%%%%%%%%%%%%%%%%%%%%%%%%%%%%%%%%%%%%%%%%%%%%%%%%%%%%%%
\section{Proof of Theorem \ref{Thm:Rbdddh}} \label{Sec:Pf}

In this section, 
we first prove the $\mathscr{R}$-boundedness of $\mathcal{S}_N(\lambda)$ and $\mathcal{T}(\lambda)$ 
by deduction of estimates of symbols in the solution formula \eqref{eq:solform}. 
We then discuss the other solution operators $\mathcal{S}_j(\lambda)$ ($j=1,\cdots,N-1$) 
by analyzing the Laplace resolvent problem \eqref{eq:PDEuj}.

%%%%%%%%%%%%%%%%%%%%%%%%%%%%%%%%%%%%%%%%%%%%%%%%%%%%%%%%%%%%%%%
%%%%%%%%%%%%%%%%%%%%%%%%%%%%%%%%%%%%%%%%%%%%%%%%%%%%%%%%%%%%%%%
% Proof of Rbdd with only boundary data for u_N, p 
%%%%%%%%%%%%%%%%%%%%%%%%%%%%%%%%%%%%%%%%%%%%%%%%%%%%%%%%%%%%%%%
%%%%%%%%%%%%%%%%%%%%%%%%%%%%%%%%%%%%%%%%%%%%%%%%%%%%%%%%%%%%%%%

%%%%%%%%%%%%%%%%%%%%%%%%%%%%%%%%%%%%%%%%%%%%%%%%%%%%%%%%%%%%%%%
% Lemma: Estimate of detL
%%%%%%%%%%%%%%%%%%%%%%%%%%%%%%%%%%%%%%%%%%%%%%%%%%%%%%%%%%%%%%%
We begin with analysis of $\det{\fb{L}}$ given by \eqref{eq:detL}. 

\begin{Lem} \label{Lem:EstidetL} 
Let $0<\varepsilon<\pi/2$ and $\gamma_0>0$, and let $\alpha' \in \mathbb{N}_0^{N-1}$ be any multi-index. 
There exists a constant $C_{N,\varepsilon,\gamma_0,\alpha'}>0$ such that 
for any $(\lambda,\xi') \in \Sigma_{\varepsilon,\gamma_0}\times (\mathbb{R}^{N-1}\setminus\{0\})$ and $\ell=0,1$, 
the following estimate holds: 
\begin{align} \label{esti:InvdetL}
\left| (\tau{\partial}_\tau)^\ell{\partial_{\xi'}^{\alpha'}}\frac{1}{\det{\fb{L}}} \right| 
\leq C_{N,\varepsilon,\gamma_0,\alpha'} (|\lambda|^{1/2}+A)^{-6} \left(1+\frac{1}{A}\right)A^{-|\alpha'|}. 
\end{align}
Here, $\Sigma_{\varepsilon,\gamma_0}$ and $A$ are given by 
\eqref{eq:resoset} and \eqref{eq:DefAB}, respectively, and $\lambda=\gamma+i\tau$. 
\end{Lem}

%%%%%%%%%%%%%%%%%%%%%%%%%%%%%%%%%%%%%%%%%%%%%%%%%%%%%%%%%
% Remark gamma0
%%%%%%%%%%%%%%%%%%%%%%%%%%%%%%%%%%%%%%%%%%%%%%%%%%%%%%%%%
\begin{Rem} \label{Rem:gamma0}
We need the condition $\gamma_0>0$ to obtain \eqref{esti:InvdetL} for $A\leq1$, 
see \eqref{eq:Estin-}, \eqref{eq:Estin+} and \eqref{eq:EstiDl-} below. 
In fact, if $\lambda=0$, the singularity is too high at $\xi'=0$ such as 
\begin{align} \label{eq:Q}
|\det{\fb{L}}|^{-1} \sim A^{-10}, \quad \forall \xi'\in \mathbb{R}^{N-1}\setminus\{0\} \text{ with } |\xi'|\leq1
\end{align}
(though the proof is omitted), while
\begin{align*} 
|\det{\fb{L}}|^{-1} \sim (|\lambda|^{1/2}+A)^{-6}A^{-1}, \quad 
\forall (\lambda,\xi')\in{\Sigma_{\varepsilon,\gamma_0}}\times(\mathbb{R}^{N-1}\setminus\{0\}) \text{ with } |\xi'|\leq1
\end{align*}
from \eqref{esti:InvdetL}, \eqref{eq:EstidetL1p} and \eqref{eq:EstidetL1m}. 
Here, $M_1\sim M_2$ means that $c_{\varepsilon,\gamma_0}M_2 \leq M_1 \leq C_{\varepsilon,\gamma_0}M_2$ 
for all $\lambda$ and $\xi'$ with some constants $c_{\varepsilon,\gamma_0},C_{\varepsilon,\gamma_0}>0$ independent of $\lambda$ and $\xi'$. 
\end{Rem}

%%%%%%%%%%%%%%%%%%%%%%%%%%%%%%%%%%%%%%%%%%%%%%%%%%%%%%%%%
% Lem: |∂^α'_\xi f(\xi)^{-1}|\leq...
%%%%%%%%%%%%%%%%%%%%%%%%%%%%%%%%%%%%%%%%%%%%%%%%%%%%%%%%%
To prove Lemma \ref{Lem:EstidetL}, we first show the following lemma. 
\begin{Lem} \label{Lem:InvMultiClass}
Let $0<\varepsilon<\pi/2$ and $\gamma_0\geq0$. 
Let $m$ be a function defined on $\Sigma_{\varepsilon,\gamma_0}\times(\mathbb{R}^{N-1}\setminus\{0\})$ 
and $f$ a positive real-valued function defined there. 
Assume that, for any $\ell=0,1$ and multi-index $\beta'\in{\mathbb{N}}_0^{N-1}$, 
there exists $C_{\beta'}>0$ such that 
\begin{align} \label{eq:LemIMC0}
|m(\lambda,\xi')| &\geq f(\lambda,\xi'), \\ \label{eq:LemIMC1}
|(\tau{\partial}_\tau)^\ell \partial_{\xi'}^{\beta'} m(\lambda,\xi')| &\leq C_{\beta'}{f(\lambda,\xi')}{A}^{-|\beta'|} 
\end{align}
for $(\lambda,\xi') \in \Sigma_{\varepsilon,\gamma_0}\times(\mathbb{R}^{N-1}\setminus\{0\})$. 
Then, for any $\ell=0,1$ and multi-index $\alpha'\in{\mathbb{N}}_0^{N-1}$, 
we have the estimate 
\begin{align} \label{eq:LemIMCc}
|(\tau{\partial}_\tau)^\ell \partial_{\xi'}^{\alpha'}{m(\lambda,\xi')}^{-1}| &\leq C_{\alpha'} f(\lambda,\xi')^{-1}{A}^{-|\alpha'|} 
\end{align}
for $(\lambda,\xi') \in \Sigma_{\varepsilon,\gamma_0}\times(\mathbb{R}^{N-1}\setminus\{0\})$. 
This assertion still holds if $A$ is replaced by $(|\lambda|^{1/2}+A)$ in \eqref{eq:LemIMC1} and \eqref{eq:LemIMCc}.  
\end{Lem}

\begin{proof}
We give the proof of \eqref{eq:LemIMCc} for $\ell=0$, 
since we can show the case $\ell=1$ and the case where $A$ is replaced by $(|\lambda|^{1/2}+A)$ similarly. 
By the Fa\`a di Bruno's formula (cf. \cite[Lemma 2.3]{BahCheDan11GMW}), 
for any multi-index $\alpha'\in{\mathbb{N}}_0^{N-1}$ with $|\alpha'|\geq1$ 
and $(\lambda,\xi')\in{\Sigma_{\varepsilon,\gamma_0}}\times(\mathbb{R}^{N-1}\setminus\{0\})$, 
\begin{align*}
\left|\partial_{\xi'}^{\alpha'}\frac{1}{m(\lambda,\xi')}\right| 
&= \Bigg|\sum_{\ell=1}^{|\alpha'|} \frac{(-1)^\ell \ell!}{m(\lambda,\xi')^{\ell+1}} 
\sum_{\substack{\alpha_1'+\cdots+\alpha_\ell'=\alpha' \\ |\alpha_i'|\geq1}} \hspace{-3mm} c_{\alpha_1'\cdots\alpha_\ell'} 
(\partial_{\xi'}^{\alpha_1'}{m(\lambda,\xi')})\cdots(\partial_{\xi'}^{\alpha_\ell'}{m(\lambda,\xi')})\Bigg| \\ 
&\leq \sum_{\ell=1}^{|\alpha'|} \frac{C_{\alpha'}}{|m(\lambda,\xi')|} 
\sum_{\substack{\alpha_1'+\cdots+\alpha_\ell'=\alpha' \\ |\alpha_i'|\geq1}} \hspace{-3mm} 
\bigg(\frac{|\partial_{\xi'}^{\alpha_1'}{m(\lambda,\xi')}|}{|m(\lambda,\xi')|}\bigg) 
\cdots \bigg(\frac{|\partial_{\xi'}^{\alpha_\ell'}{m(\lambda,\xi')}|}{|m(\lambda,\xi')|}\bigg) \\ 
&\leq \sum_{\ell=1}^{|\alpha'|} \frac{C_{\alpha'}}{f(\lambda,\xi')} 
\sum_{\substack{\alpha_1'+\cdots+\alpha_\ell'=\alpha' \\ |\alpha_i'|\geq1}} \hspace{-3mm} 
\bigg(\frac{f(\lambda,\xi'){A}^{-|\alpha_1'|}}{f(\lambda,\xi')}\bigg) 
\cdots \bigg(\frac{f(\lambda,\xi'){A}^{-|\alpha_\ell'|}}{f(\lambda,\xi')}\bigg) \\
&\leq C_{\alpha'} f(\lambda,\xi')^{-1}{A}^{-|\alpha'|}. 
\end{align*}
Hence, the proof is complete since the case $\alpha'=0$ is obvious by \eqref{eq:LemIMC0}. 
\end{proof}

%%%%%%%%%%%%%%%%%%%%%%%%%%%%%%%%%%%%%%%%%%%%%%%%%%%%%%%%%
% Proof of esti of detL^{-1}
%%%%%%%%%%%%%%%%%%%%%%%%%%%%%%%%%%%%%%%%%%%%%%%%%%%%%%%%%

	%%%%%%%%%%%%%%%%%%%%%%%%%%%%%%%%%%%%%%%%%%%%%%%%%%%%%%%%%
	%%%%%%%%%%%%%%%%%%%%%%%%%%%%%%%%%%%%%%%%%%%%%%%%%%%%%%%%%
	%%	  このrenewcommnandには注意！！						%
	%%%%%%%%%%%%%%%%%%%%%%%%%%%%%%%%%%%%%%%%%%%%%%%%%%%%%%%%%
	%%%%%%%%%%%%%%%%%%%%%%%%%%%%%%%%%%%%%%%%%%%%%%%%%%%%%%%%%

%\renewcommand{\widehat{}(\xi',x_N)}[2][x_N]{\widehat{#2}}
%\renewcommand{\widehat{}_{\widehat{#2}_{#3}}(\xi',x_N)}
%\renewcommand{\widehat{}_d(\xi',x_N)}[2][x_N]{\widehat{#2}_d(\xi',#1)}%i\xi'\cdot{\widehat{#1}{'}(\xi',x_N)}

\begin{proof}[Proof of Lemma \ref{Lem:EstidetL}]
%%%%%%%%%%%%%%%%%%%%%%%%%%%%%%%%%%%% midasi
The proof is divided into three steps: 
\begin{itemize}
\item[(i)] 
For any $(\lambda,\xi') \in (\mathbb{C}\setminus(-\infty,0])\times(\mathbb{R}^{N-1}\setminus\{0\})$, 
$\det{\fb{L}}\ne0$. 
\item[(ii)] 
For any $0<\varepsilon<\pi/2$ and $\gamma_0>0$, there exists $c_{\varepsilon,\gamma_0}>0$ such that 
\begin{align} \label{eq:estidetLfb}
|\det{\fb{L}}| \geq c_{\varepsilon,\gamma_0}(|\lambda|^{1/2}+A)^6\min\{ 1,A \} \ 
(\forall(\lambda,\xi') \in \Sigma_{\varepsilon,\gamma_0}\times(\mathbb{R}^{N-1}\setminus\{0\})).
\end{align}
\item[(iii)]
Conclusion \eqref{esti:InvdetL}. 
\end{itemize}

%%%%%%%%%%%%%%%%%%%%%%%%%%%%%%%%%%%% Step1-1 Im lambda\ne0 
(i) Following the argument in \cite[Lemma 2.2]{Abe04MMAS}, 
we first prove $\det{\fb{L}}\ne0$ under the assumption $\text{Im\,}\lambda\ne0$. 
We argue by contradiction. 
Assume that $\det{\fb{L}}=0$ with some $(\lambda,\xi')\in(\mathbb{C}\setminus{\mathbb{R}})\times(\mathbb{R}^{N-1}\setminus\{0\})$. 
Then there exists 
\begin{align} \label{eq:EnerAss}
\fb{x}:=(\mu_{1N},\mu_{2N},\beta_{1N},\beta_{2N})^\mathsf{T}\ne0 
\end{align}
such that $\fb{L}\fb{x}=0$. We define 
\begin{gather}
\widehat{u}_d(x_N)=\widehat{u}_d(\xi',x_N)\text{ in \eqref{eq:ODEsolform}}, \quad 
\widehat{u}_N(x_N)=\widehat{u}_N(\xi',x_N)\text{ in \eqref{eq:ODEsolform}}, \\ 
\widehat{\theta}(x_N)=\widehat{\theta}(\xi',x_N)\text{ in \eqref{eq:ODEsolform}} 
\end{gather} 
with coefficients 
\begin{align}
\mu_{\ell d} = (-1)^\ell A\mu_{\ell N}, \quad \beta_{\ell d} = (-1)^\ell(\mu_{\ell N}+B\beta_{\ell N}), \quad 
\gamma_\ell=(-1)^{\ell-1}\frac{\mu(B+A)}{A}\mu_{\ell N}, 
\end{align}
which can be deduced by \eqref{eq:ODErelation} and \eqref{eq:RelateODEsf}. 
Note that $\widehat{u}_d, \widehat{u}_N, \widehat{\theta} \in C^\infty([0,\delta])$ by the definition \eqref{eq:ODEsolform}. 
Then they obey \eqref{eq:FSRdht} with zero data, that is, 
\begin{align} \label{eq:FSRdht0} % after reduction for Tangent direction with data 0
\left\{
\begin{array}{ll}
\mu(B^2-\partial_N^2) \widehat{u}_d(x_N) - A^2\widehat{\theta}(x_N) = 0 & 0\leq x_N\leq \delta, \\
\mu(B^2-\partial_N^2) \widehat{u}_N(x_N) + \partial_N{\widehat{\theta}(x_N)} = 0 & 0\leq x_N\leq \delta, \\ 
\widehat{u}_d(x_N) + \partial_N{\widehat{u}_N(x_N)} = 0 & 0\leq x_N\leq \delta, \\
\mu(\partial_N{\widehat{u}_d(x_N)} - A^2\widehat{u}_N(x_N))\nu_N(x_N) = 0 & x_N \in \{0,\delta\}, \\
(2\mu{\partial}_N{\widehat{u}_N(x_N)} - \widehat{\theta}(x_N))\nu_N(x_N) = 0 & x_N \in \{0,\delta\} 
\end{array}
\right.
\end{align} 
and satisfy 
\begin{align} \label{eq:FSRdhp2}
&(\partial_N^2-A^2)\widehat{\theta}(x_N)=0 \quad (0\leq x_N\leq \delta), \\ \label{eq:FSRdhuN2}
&(A^2-\partial_N^2)(B^2-\partial_N^2)\widehat{u}_N(x_N)=0 \quad (0\leq x_N\leq \delta) 
\end{align}
as in \eqref{eq:FSRdhp} and \eqref{eq:FSRdhu}. 
By inserting the third equation to the fourth one, we also get 
\begin{align} \label{eq:FSRdhuN3}
(\partial_N^2+A^2)\widehat{u}_N(x_N)=0 \quad (x_N \in \{0,\delta\}). 
\end{align}
Multiplying \eqref{eq:FSRdhuN2} by $\overline{\widehat{u}_N(x_N)}$, integrating the resultant formula over $(0,\delta)$, 
and integration by parts yield 
\begin{align} \label{eq:FSRibp}
\qquad \begin{aligned}
&\big(\partial_N^3\widehat{u}_N(\delta)\overline{\widehat{u}_N(\delta)}-\partial_N^3\widehat{u}_N(0)\overline{\widehat{u}_N(0)}\big) \\ & 
-\big(\partial_N^2\widehat{u}_N(\delta)\overline{\partial_N{\widehat{u}_N}(\delta)}-\partial_N^2\widehat{u}_N(0)\overline{\partial_N{\widehat{u}_N}(0)}\big) 
+\|\partial_N^2\widehat{u}_N\|_{L_2(0,\delta)}^2  \\ \quad & 
-(B^2+A^2)\big(\partial_N{\widehat{u}_N}(\delta)\overline{\widehat{u}_N(\delta)}-\partial_N{\widehat{u}_N}(0)\overline{\widehat{u}_N(0)}\big) 
+(B^2+A^2)\|\partial_N{\widehat{u}_N(\xi',\cdot)}\|_{L_2(0,\delta)}^2 \quad \\ & 
+A^2B^2\|\widehat{u}_N(\xi',\cdot)\|_{L_2(0,\delta)}^2 \\ & 
=0. 
\end{aligned} \qquad 
\end{align}
From the second equation of \eqref{eq:FSRdht0} multiplied by $\partial_N$, 
from \eqref{eq:FSRdhp2} and from the fifth equation in \eqref{eq:FSRdht0}, 
we get 
\begin{align} \notag 
\partial_N^3\widehat{u}_N=B^2\partial_N{\widehat{u}_N}+\mu^{-1}{\partial}_N^2\widehat{\theta}
=B^2\partial_N{\widehat{u}_N}+\mu^{-1}A^2\widehat{\theta}=(B^2+2A^2)\partial_N{\widehat{u}_N}\text{ on } \{0,\delta\}. 
\end{align}
By this and \eqref{eq:FSRdhuN3}, the equation \eqref{eq:FSRibp} implies 
\begin{align}
&2A^2\text{Re\,}{\big(\partial_N{\widehat{u}_N}(\delta)\overline{\widehat{u}_N(\delta)}-\partial_N{\widehat{u}_N}(0)\overline{\widehat{u}_N(0)}\big)} \\ \notag & 
+\|\partial_N^2\widehat{u}_N\|_{L_2(0,\delta)}^2
+(B^2+A^2)\|\partial_N{\widehat{u}_N}\|_{L_2(0,\delta)}^2+A^2B^2\|\widehat{u}_N\|_{L_2(0,\delta)}^2=0. 
\end{align}
If we take the imaginary part, we have $\widehat{u}_N=0$ by  $\text{Im\,}\lambda\ne0$, 
but this contradicts the assumption \eqref{eq:EnerAss}. 

%%%%%%%%%%%%%%%%%%%%%%%%%%%%%%%%%%%% Step1-2 not lamdda>0

Next, we show $\det{\fb{L}}\ne0$ even if $\lambda>0$. If we set 
\begin{align}
x=\sqrt{1+\lambda/(\mu A^2)}, 
\end{align}
we have $B=Ax$ and, so, by the definition \eqref{eq:detL} of $\det{\fb{L}}$, 
\begin{align} \label{eq:detLfA}
(B-A)^2\det{\fb{L}}= A^8\prod_{+,-} f_A^\pm(x)
\end{align}
with 
\begin{align*} 
f_A^\pm(x) &= (x^2+1)^2(1\pm{e^{-A\delta }})(1\mp{e^{-A\delta x}}) - 4x(1\mp{e^{-A\delta }})(1\pm{e^{-A\delta x}}). 
\end{align*}
Also, we rewrite $f_A^-(x)$ as  
\begin{align*}
f_A^-(x) 
&=\big((x^2+1)^2-4x\big)(1-e^{-A\delta (x+1)})-\big((x^2+1)^2+4x\big)(e^{-A\delta }-e^{-A\delta x}) \\
&=(x-1)(x^3+x^2+3x-1)A\delta(x+1)\int_0^1 e^{-A\delta(x+1)t} \, dt \\ & 
	-(x^4+2x^2+4x+1)A\delta(x-1)\int_0^1 e^{-A\delta(1+(x-1)t)} \, dt. 
\end{align*}
Then we have $f_A^\pm(x)=p_1^\pm p_2^\pm p_3^\pm-q_1^\pm q_2^\pm q_3^\pm$ 
with $p_1^\pm \geq q_1^\pm>0$ and $p_i^\pm > q_i^\pm>0$ for $i=2,3$, where  
\begin{align*}
p_1^+&=(x^2+1)^2,& 
p_2^+&=1+e^{-A\delta },& 
p_3^+&=1-e^{-A\delta x}, \\ 
q_1^+&=4x,& 
q_2^+&=1+e^{-A\delta x},& 
q_3^+&=1-e^{-A\delta }, \\ 
p_1^-&=(x-1)A\delta,& 
p_2^-&=(x^3+x^2+3x-1)(x+1),& 
p_3^-&=\int_0^1 e^{-A\delta(1+x)t} \, dt, \\ 
q_1^-&=(x-1)A\delta,& 
q_2^-&=x^4+2x^2+4x+1,& 
q_3^-&=\int_0^1 e^{-A\delta(1+(x-1)t)} \, dt.  
\end{align*}
In fact, to verify $p_3^->q_3^-$, we observe 
\begin{align*}
&p_3^--q_3^- 
= \int_0^1 (e^{-A\delta(1+x)t}-e^{-A\delta(1+(x-1)t)})\,dt 
= \int_0^1\left( \int_0^1 A\delta(1-2t)e^{-A\delta(s+(1+x-2s)t)} \,ds \right) \,dt \\ 
&= A\delta \int_0^1 e^{-A\delta s} \int_0^1 (1-2t)e^{-\alpha(s)t}\,dt\,ds
= A\delta \int_0^1 e^{-A\delta s} \int_0^{1/2} (1-2t)(e^{-\alpha(s)t}-e^{-\alpha(s)(1-t)})\,dt\,ds, 
\end{align*}
where we have set $\alpha(s)=A\delta(1+x-2s)>0$ for $s\in(0,1)$. 
Since the integrand of the right-hand side is positive, we get $p_3^--q_3^->0$. 
Other inequalities are verified easily. 
Hence, we get $f_A^\pm(x)>0$ for $x>1$ and $A>0$, 
which implies $\det{\fb{L}}\ne0$ for $(\lambda,\xi')\in(0,\infty)\times(\mathbb{R}^{N-1}\setminus\{0\})$ by \eqref{eq:detLfA}. 

(ii) We shall show the estimate \eqref{eq:estidetLfb} of $\det{\fb{L}}$ in this step. Set 
\begin{align*}
\ell_\pm(A,B)
=\frac{1}{(B-A)} \{ (B^2+A^2)^2(1\pm e^{-A\delta })(1\mp e^{-B\delta }) - 4A^3B(1\mp e^{-A\delta })(1\pm e^{-B\delta }) \} 
\end{align*}
so that 
\begin{align}
\det{\fb{L}}=\ell_+(A,B){\ell_-(A,B)}. 
\end{align} 
Then it is sufficient to prove
\begin{align} \label{eq:EstidetL0pm} 
|\ell_+(A,B)| \geq c(|\lambda|^{1/2}+A)^3, \quad
|\ell_-(A,B)| \geq c(|\lambda|^{1/2}+A)^3\min\{ 1,A \}. 
\end{align}
We first show \eqref{eq:EstidetL0pm} for $(\lambda,\xi')\in{\Sigma_{\varepsilon,\gamma_0}}\times(\mathbb{R}^{N-1}\setminus\{0\})$ 
such that $|\xi'|>M$ for some $M>0$. 
We rewrite $\ell_\pm(A,B)$ as 
\begin{align} \label{eq:RewritedetL} 
\begin{aligned}
&\ell_\pm(A,B) \\ 
&=\frac{1}{(B-A)} \Big\{ \big((B^2+A^2)^2-4A^3B\big)(1-e^{-(B+A)\delta}) %\\ &\hspace{28mm}
	\mp \big((B^2+A^2)^2+4A^3B\big)(e^{-B\delta }-e^{-A\delta }) \Big\} \\  
&=\frac{(B-A)(B^3+AB^2+3A^2B-A^3)}{(B-A)}(1-e^{-(B+A)\delta}) %\\ &\hspace{28mm}
	\mp \big((B^2+A^2)^2+4A^3B\big)\frac{e^{-B\delta }-e^{-A\delta }}{B-A} \\ 
&=D_3(1-e^{-(B+A)\delta}) \mp (B^4+2A^2B^2+4A^3B+A^4)\mathcal{M}(\delta), 
\end{aligned}
\end{align}
where $D_3$ and $\mathcal{M}(\delta)$ is given by \eqref{eq:defDis} and \eqref{eq:DefccM} with $x_N=\delta$. 
Thus, by \eqref{esti:equivB}, \eqref{esti:multiclasski} with $\ell, \alpha'=0$ and 
\begin{align} \label{eq:estipb}
|e^{-B\delta }| = e^{-\text{Re\,} B\delta} \leq e^{-c_{\varepsilon,\mu}(|\lambda|^{1/2}+A)\delta} \leq 1, 
\end{align}
we have 
\begin{align*}
|\ell_\pm(A,B)| 
&\geq |D_3| - |D_3 e^{-B\delta}e^{-A\delta}| - |B^{-1}(B^4+2A^2B^2+4A^3B+A^4)B{\mathcal{M}}(\delta)| \\
&\geq 2c(|\lambda|^{1/2}+A)^3 - C_1(|\lambda|^{1/2}+A)^3e^{-M\delta} - C_2(|\lambda|^{1/2}+A)^3e^{-c_{\varepsilon,\mu}M\delta} \\ 
&\geq c(|\lambda|^{1/2}+A)^3 
\end{align*}
for $(\lambda,\xi')\in{\Sigma_{\varepsilon,\gamma_0}}\times(\mathbb{R}^{N-1}\setminus\{0\})$ with $A>M$, 
by taking $M>0$ so large that $C_1e^{-M\delta}+C_2e^{-c_{\varepsilon,\mu}M\delta} \leq c$. 

%%%%%%%%%%%%%%%%%%%%%%%%%%% Case 2
Next, we consider the case where $(\lambda,\xi')\in{\Sigma_{\varepsilon,\gamma_0}}\times(\mathbb{R}^{N-1}\setminus\{0\})$ 
such that $|\xi'|/|\lambda|^{1/2}<\eta$ for some $\eta>0$. 
Define $n_\pm(D,A,B)$ by 
\begin{align} \label{eq:detLcase2}
\begin{aligned}
\ell_\pm(A,B) %\\ \notag
&= \frac{B^3}{1-D} \{ (1+2D^2+D^4)(1\pm e^{-A\delta })(1\mp e^{-B\delta }) - 4D^3(1\mp e^{-A\delta })(1\pm e^{-B\delta }) \} \\
&=: \frac{B^3}{1-D} n_\pm(D,A,B) 
\end{aligned}
\end{align}
where 
\begin{align}
D=\frac A B. 
\end{align}
Since $|1+e^{-B\delta}|$ is continuous and nonzero for $(\lambda,\xi')\in{\Sigma_{\varepsilon,\gamma_0}}\times(\mathbb{R}^{N-1}\setminus\{0\})$, 
it has the minimum value $c_0$.  
Then, since
\begin{align}
&1-e^{-A\delta }\left\{ 
\begin{aligned}
&\geq 1-e^{-\delta} &&\text{for } A\geq1, \\
&= A\delta\int_0^1 e^{-A\delta t} \,dt \geq A\delta e^{-\delta} &&\text{for } 0<A\leq1, 
\end{aligned} 
\right. \\ \label{esti:Dbymin1A}
&|D| \leq \frac{A}{c_{\varepsilon,\mu}(|\lambda|^{1/2}+A)} \leq 
\left\{ \begin{aligned}
&\eta/c_{\varepsilon,\mu} \leq 1/2, \\ 
&A/(c_{\varepsilon,\mu}\gamma_0^{1/2}) 
\end{aligned} 
\right. \quad  
\end{align}
(by taking $\eta$ so small if necessary), 
we get $1-e^{-A\delta} \geq 2c\min\{1,A\}$ for some $c>0$ 
and $|D|\leq C_{\gamma_0}\min\{1,A\}$ for some $C_{\gamma_0}>0$. 
Then we obtain 
\begin{align} \label{eq:Estin-}
\begin{aligned}
&|n_-(D,A,B)| \\ 
&\geq |(1-e^{-A\delta })(1+e^{-B\delta })| - |(2D^2+D^4)(1-e^{-A\delta })(1+e^{-B\delta })|-4|D^3(1+e^{-A\delta })(1-e^{-B\delta })| \\ 
&\geq 2cc_0\min\{1,A\}-C_1C_{\gamma_0}(\eta/c_{\varepsilon,\mu})\min\{1,A\}
-C_2C_{\gamma_0}(\eta/c_{\varepsilon,\mu})^2\min\{1,A\} \\ 
&\geq cc_0\min\{1,A\} 
\end{aligned}
\end{align} 
if we take $\eta>0$ so small that 
$C_1C_{\gamma_0}(\eta/c_{\varepsilon,\mu})+C_2C_{\gamma_0}(\eta/c_{\varepsilon,\mu})^2\leq cc_0$. 
Moreover, by \eqref{esti:Dbymin1A} and 
\begin{align}
|e^{-B\delta }| \leq e^{-c(|\lambda|^{1/2}+A)\delta}\leq e^{-c\gamma_0^{1/2}\delta}, 
\end{align}
we have
\begin{align} \label{eq:Estin+}
\begin{aligned}
&|n_+(D,A,B)| \\ 
&\geq |(1+e^{-A\delta })(1-e^{-B\delta })|-|(2D^2+D^4)(1+e^{-A\delta })(1-e^{-B\delta })|-4|D^3(1-e^{-A\delta })(1+e^{-B\delta })| \\ 
&\geq (1-e^{-c\gamma_0^{1/2}\delta})-12(\eta/c_{\varepsilon,\mu})^2-8(\eta/c_{\varepsilon,\mu})^3 \\ 
&\geq (1-e^{-c\gamma_0^{1/2}\delta})/2 
\end{aligned}
\end{align}
if we take $\eta>0$ so small that 
$12(\eta/c_{\varepsilon,\mu})^2+8(\eta/c_{\varepsilon,\mu})^3\leq (1-e^{-c\gamma_0^{1/2}\delta})/2$. 
Thus, by \eqref{esti:equivB} and \eqref{eq:detLcase2}--\eqref{eq:Estin+}, 
we obtain \eqref{eq:EstidetL0pm} in this case.   

%%%%%%%%%%%%%%%%%%%%%%%%%%%%%% Case 3
Finally, the estimate \eqref{eq:EstidetL0pm} also holds on the remainder region 
\begin{align}
D_r = \{ (\lambda,\xi') \in \overline{\Sigma_{\varepsilon,\gamma_0}}\times(\mathbb{R}^{N-1}\setminus\{0\}) 
\mid |\xi'|\leq M, \ |\xi'|/|\lambda|^{1/2}\geq \eta \} 
\end{align}
by step (i) since $\det{\fb{L}}$ is a continuous function of $(\lambda,\xi')$ on $D_r$ and $D_r$ is compact.

%%%%%%%%%%%%%%%%%%%%%%%%%%%%%%%%%%%% Step3
(iii) By Lemma \ref{Lem:InvMultiClass} and step (ii), we will obtain \eqref{esti:InvdetL} 
if we prove \eqref{eq:LemIMC1} with 
$m(\lambda,\xi')=\det{\fb{L}}$ and $f(\lambda,\xi')=c_{\varepsilon,\gamma_0}(|\lambda|^{1/2}+A)^6\min\{ 1,A \}$. 
Thus, by the Leibniz rule, it suffices to show  
\begin{align} \label{eq:EstidetL1p} 
|(\tau{\partial}_\tau)^\ell{\partial_{\xi'}^{\beta'}} \ell_+(A,B)| &\leq C_{\varepsilon,\beta'}(|\lambda|^{1/2}+A)^3 A^{-|\beta'|}, \\ \label{eq:EstidetL1m}
|(\tau{\partial}_\tau)^\ell{\partial_{\xi'}^{\beta'}} \ell_-(A,B)| &\leq C_{\varepsilon,\gamma_0,\beta'}(|\lambda|^{1/2}+A)^3 \min\{ 1,A \} A^{-|\beta'|} 
\end{align}
for any $(\lambda,\xi')\in{\Sigma_{\varepsilon,\gamma_0}}\times(\mathbb{R}^{N-1}\setminus\{0\})$, $\ell = 0,1$ and multi-index $\beta'\in{\mathbb{N}}_0^{N-1}$. 
The estimate \eqref{eq:EstidetL1m} for $A\geq1$ and \eqref{eq:EstidetL1p} 
are obtained by \eqref{eq:RewritedetL} and Lemma \ref{Lem:EstiSymbol}. 
We shall show \eqref{eq:EstidetL1m} for $A\leq1$. 
By \eqref{eq:RewritedetL}, 
\begin{align*}
\ell_-(A,B) 
&= (B^3+AB^2+3A^2B-A^3)(1-e^{-(B+A)\delta}) \\
&+ (B^3(B-A)+(AB^3+2A^2B^2+4A^3B+A^4))\mathcal{M}(\delta) \\
&= B^3((1-e^{-(B+A)\delta})+(e^{-B\delta }-e^{-A\delta })) \\ 
&+ A\{ (B^2+3AB-A^2)(1-e^{-(B+A)\delta})+(B^3+2AB^2+4A^2B+A^3)\mathcal{M}(\delta) \}. 
\end{align*}
Then, if we rewrite the first term of the right-hand side as 
\begin{align}
&B^3((1-e^{-(B+A)\delta})+(e^{-B\delta }-e^{-A\delta })) \\ 
&=B^3(1-e^{-A\delta })(1+e^{-B\delta }) = B^3A\delta \int_0^1 e^{-A\delta t} \,dt (1+e^{-B\delta }), 
\end{align}
by Lemma \ref{Lem:EstiSymbol}, we get
\begin{align} \label{eq:EstiDl-}
\begin{aligned}
&|(\tau{\partial}_\tau)^\ell{\partial_{\xi'}^{\beta'}} \ell_-(A,B)| \\ 
&\leq \int_0^1 \big|(\tau{\partial}_\tau)^\ell{\partial_{\xi'}^{\beta'}}\big(B^3A\delta{e^{-A\delta t}}(1+e^{-B\delta })\big)\big|\,dt \\ 
&+ |(\tau{\partial}_\tau)^\ell{\partial_{\xi'}^{\beta'}}\big(A\{ 
(B^2+3AB-A^2)(1-e^{-(B+A)\delta})+(B^3+2AB^2+4A^2B+A^3)\mathcal{M}(\delta) \}\big)| \\ 
&\leq C_{\beta',\varepsilon} \{(|\lambda|^{1/2}+A)^3A\delta + A(|\lambda|^{1/2}+A)^2\} A^{-|\beta'|} \\ 
&\leq C_{\beta',\varepsilon} (\delta+\gamma_0^{-1/2}) (|\lambda|^{1/2}+A)^3 A^{1-|\beta'|}. 
\end{aligned}
\end{align}
This proves \eqref{eq:EstidetL1m} for $A\leq1$ and, thus, the proof is complete. 
\end{proof}

%%%%%%%%%%%%%%%%%%%%%%%%%%%%%%%%%%%%%%%%%%%%%%%%%%%%%%%%%%%%%%%
% Rewrite sol formula
%%%%%%%%%%%%%%%%%%%%%%%%%%%%%%%%%%%%%%%%%%%%%%%%%%%%%%%%%%%%%%%

	%%%%%%%%%%%%%%%%%%%%%%%%%%%%%%%%%%%%%%%%%%%%%%%%%%%%%%%%%
	%%%%%%%%%%%%%%%%%%%%%%%%%%%%%%%%%%%%%%%%%%%%%%%%%%%%%%%%%
	%% Renewcommand										% 
	%%%%%%%%%%%%%%%%%%%%%%%%%%%%%%%%%%%%%%%%%%%%%%%%%%%%%%%%%
	%%%%%%%%%%%%%%%%%%%%%%%%%%%%%%%%%%%%%%%%%%%%%%%%%%%%%%%%%

%\renewcommand{\widehat{}(\xi',x_N)}[2][x_N]{\widehat{#2}(\xi',#1)}
%\renewcommand{\widehat{}_{\widehat{#2}_{#3}(\xi',#1)}(\xi',x_N)}
%\renewcommand{\widehat{}_d(\xi',x_N)}[2][x_N]{\widehat{#2}_d(\xi',#1)}%i\xi'\cdot{\widehat{#1}{'}(\xi',x_N)}

From now on, in order to prove the assertions for $\mathcal{S}_N(\lambda)$ and $\mathcal{T}(\lambda)$ in Theorem \ref{Thm:Rbdddh}, 
we rewrite the solution formula \eqref{eq:solform} of $u_N$ and $\theta$. 
We will construct $\mathcal{S}_j(\lambda)$ ($j=1,\cdots,N-1$) together with the $\mathscr{R}$-boundedness at the end of this section. 
As one can see from \eqref{esti:InvdetL}, the estimate of $\det{\fb{L}}$ is inhomogeneous in the sense that 
\begin{align}
|\det{\fb{L}}|^{-1}=
\left\{ 
\begin{array}{ll}
O(|\xi'|^{-1}), &\text{as } \xi'\to0, \\ 
O(1), &\text{as } |\xi'|\to\infty 
\end{array}
\right. 
\end{align}  
for fixed $\lambda\in{\Sigma_{\varepsilon,\gamma_0}}$. 
In order to overcome this difficulty, we divide each term of the solution formula into two parts: 
the part with the same singularity as that for the case of the Neumann-Dirichlet boundary condition 
and the one with higher singularity. 
Let $\zeta_0, \zeta_1\in C^\infty(\mathbb{R}^{N-1})$ be cut-off functions such that
\begin{align} \label{eq:Defzetai}
0\leq\zeta_0(\xi')\leq1, \quad \zeta_0(\xi') = \left\{
\begin{array}{ll}
1, & |\xi'| \geq 2,\\
0, & |\xi'| \leq 1,
\end{array}
\right. 
\quad \zeta_1(\xi')=A(1-\zeta_0(\xi')) 
\end{align}
so that $1=\zeta_0(\xi')+\zeta_1(\xi')/A$. 
By using this, we rewrite \eqref{eq:solform}: 
\begin{align} \label{eq:rwSN} %DEF of S_N and P
u_N
&= \sum_{j=1}^{N-1}\sum_{\ell=1}^2 \mathscr{F}_{\xi'}^{-1}\left[ 
	\left\{\zeta_0{\frac{i\xi_j}{A}\frac{L_{1,2\ell-1}}{\mu B\det{\fb{L}}}} A+\zeta_1{\frac{i\xi_j}{A}\frac{L_{1,2\ell-1}}{\mu B\det{\fb{L}}}}\right\} 
	B{\mathcal{M}}(d_\ell(x_N)) \widehat{h}_j(\xi',\delta) \right] \\ \notag
&+ \sum_{j=1}^{N-1}\sum_{\ell=1}^2 \mathscr{F}_{\xi'}^{-1}\left[ 
	\left\{\zeta_0{\frac{i\xi_j}{A}\frac{L_{1,2\ell   }}{\mu   \det{\fb{L}}}} A+\zeta_1{\frac{i\xi_j}{A}\frac{L_{1,2\ell   }}{\mu   \det{\fb{L}}}}\right\} 
	e^{-Bd_\ell(x_N)}   \widehat{h}_j(\xi',\delta) \right] \\ \notag
&- \sum_{j=1}^{N-1}\sum_{\ell=1}^2 \mathscr{F}_{\xi'}^{-1}\left[ 
	\left\{\zeta_0{\frac{i\xi_j}{A}\frac{L_{2,2\ell-1}}{\mu B\det{\fb{L}}}} A+\zeta_1{\frac{i\xi_j}{A}\frac{L_{2,2\ell-1}}{\mu B\det{\fb{L}}}}\right\} 
	B{\mathcal{M}}(d_\ell(x_N)) \widehat{h}_j(\xi',0) \right] \\ \notag
&- \sum_{j=1}^{N-1}\sum_{\ell=1}^2 \mathscr{F}_{\xi'}^{-1}\left[ 
	\left\{\zeta_0{\frac{i\xi_j}{A}\frac{L_{2,2\ell   }}{\mu   \det{\fb{L}}}} A+\zeta_1{\frac{i\xi_j}{A}\frac{L_{2,2\ell   }}{\mu   \det{\fb{L}}}}\right\} 
	e^{-Bd_\ell(x_N)}   \widehat{h}_j(\xi',0) \right] \\ \notag
&+ \sum_{\ell=1}^2 \mathscr{F}_{\xi'}^{-1}\left[ 
	\left\{\zeta_0{\frac{L_{3,2\ell-1}}{\mu B\det{\fb{L}}}} A+\zeta_1{\frac{L_{3,2\ell-1}}{\mu B\det{\fb{L}}}}\right\} 
	B{\mathcal{M}}(d_\ell(x_N)) \widehat{h}_N(\xi',\delta) \right] \\ \notag
&+ \sum_{\ell=1}^2 \mathscr{F}_{\xi'}^{-1}\left[ 
	\left\{\zeta_0{\frac{L_{3,2\ell   }}{\mu   \det{\fb{L}}}} A+\zeta_1{\frac{L_{3,2\ell   }}{\mu   \det{\fb{L}}}}\right\} 
	e^{-Bd_\ell(x_N)}   \widehat{h}_N(\xi',\delta) \right] \\ \notag
&- \sum_{\ell=1}^2 \mathscr{F}_{\xi'}^{-1}\left[ 
	\left\{\zeta_0{\frac{L_{4,2\ell-1}}{\mu B\det{\fb{L}}}} A+\zeta_1{\frac{L_{4,2\ell-1}}{\mu B\det{\fb{L}}}}\right\} 
	B{\mathcal{M}}(d_\ell(x_N)) \widehat{h}_N(\xi',0) \right] \\ \notag
&- \sum_{\ell=1}^2 \mathscr{F}_{\xi'}^{-1}\left[ 
	\left\{\zeta_0{\frac{L_{4,2\ell   }}{\mu   \det{\fb{L}}}} A+\zeta_1{\frac{L_{4,2\ell   }}{\mu   \det{\fb{L}}}}\right\} 
	e^{-Bd_\ell(x_N)}   \widehat{h}_N(\xi',0) \right] \\ \notag
&\begin{aligned}
= \sum_{j=1}^{N-1}\sum_{\ell=1}^2 \bigg\{ &\mathcal{K}^3_{2,\ell,1}\left(\lambda,\zeta_0{\frac{i\xi_j}{A}\frac{L_{1,2\ell-1}}{\mu B\det{\fb{L}}}}\right)(M_{\lambda^{-\frac12}}\lambda^{\frac12}{h_j},\nabla h_j) \\ \notag 
&+\mathcal{K}^3_{3,\ell,1}\left(\lambda,\zeta_1{\frac{i\xi_j}{A}\frac{L_{1,2\ell-1}}{\mu B\det{\fb{L}}}}\right)(M_{\lambda^{-\frac12}}\lambda^{\frac12}{h_j},\nabla h_j)\bigg\}
\end{aligned} \\ \notag
&\begin{aligned}{+} \sum_{j=1}^{N-1}\sum_{\ell=1}^2 \bigg\{&\mathcal{K}^1_{2,\ell,1}\left(\lambda,\zeta_0{\frac{i\xi_j}{A}\frac{L_{1,2\ell   }}{\mu   \det{\fb{L}}}}\right)(M_{\lambda^{-\frac12}}\lambda^{\frac12}{h_j},\nabla h_j) \\ \notag &+\mathcal{K}^1_{3,\ell,1}\left(\lambda,\zeta_1{\frac{i\xi_j}{A}\frac{L_{1,2\ell   }}{\mu   \det{\fb{L}}}}\right)(M_{\lambda^{-\frac12}}\lambda^{\frac12}{h_j},\nabla h_j)\bigg\}\end{aligned} \\ \notag
&\begin{aligned}{-} \sum_{j=1}^{N-1}\sum_{\ell=1}^2 \bigg\{&\mathcal{K}^3_{2,\ell,2}\left(\lambda,\zeta_0{\frac{i\xi_j}{A}\frac{L_{2,2\ell-1}}{\mu B\det{\fb{L}}}}\right)(M_{\lambda^{-\frac12}}\lambda^{\frac12}{h_j},\nabla h_j) \\ \notag &+\mathcal{K}^3_{3,\ell,2}\left(\lambda,\zeta_1{\frac{i\xi_j}{A}\frac{L_{2,2\ell-1}}{\mu B\det{\fb{L}}}}\right)(M_{\lambda^{-\frac12}}\lambda^{\frac12}{h_j},\nabla h_j)\bigg\}\end{aligned} \\ \notag
&\begin{aligned}
{-} \sum_{j=1}^{N-1}\sum_{\ell=1}^2 \bigg\{&\mathcal{K}^1_{2,\ell,2}\left(\lambda,\zeta_0{\frac{i\xi_j}{A}\frac{L_{2,2\ell   }}{\mu   \det{\fb{L}}}}\right)(M_{\lambda^{-\frac12}}\lambda^{\frac12}{h_j},\nabla h_j) \\ \notag 
&+\mathcal{K}^1_{3,\ell,2}\left(\lambda,\zeta_1{\frac{i\xi_j}{A}\frac{L_{2,2\ell   }}{\mu   \det{\fb{L}}}}\right)(M_{\lambda^{-\frac12}}\lambda^{\frac12}{h_j},\nabla h_j)\bigg\}
\end{aligned} \\ \notag
&+ \sum_{\ell=1}^2 \left\{\mathcal{K}^3_{2,\ell,1}\left(\lambda,\zeta_0{\frac{L_{3,2\ell-1}}{\mu B\det{\fb{L}}}}\right)(h_N,\nabla h_N)+\mathcal{K}^3_{3,\ell,1}\left(\lambda,\zeta_1{\frac{L_{3,2\ell-1}}{\mu B\det{\fb{L}}}}\right)(h_N,\nabla h_N)\right\} \\ \notag
&+ \sum_{\ell=1}^2 \left\{\mathcal{K}^1_{2,\ell,1}\left(\lambda,\zeta_0{\frac{L_{3,2\ell   }}{\mu   \det{\fb{L}}}}\right)(h_N,\nabla h_N)+\mathcal{K}^1_{3,\ell,1}\left(\lambda,\zeta_1{\frac{L_{3,2\ell   }}{\mu   \det{\fb{L}}}}\right)(h_N,\nabla h_N)\right\} \\ \notag
&- \sum_{\ell=1}^2 \left\{\mathcal{K}^3_{2,\ell,2}\left(\lambda,\zeta_0{\frac{L_{4,2\ell-1}}{\mu B\det{\fb{L}}}}\right)(h_N,\nabla h_N)+\mathcal{K}^3_{3,\ell,2}\left(\lambda,\zeta_1{\frac{L_{4,2\ell-1}}{\mu B\det{\fb{L}}}}\right)(h_N,\nabla h_N)\right\} \\ \notag
&- \sum_{\ell=1}^2 \left\{\mathcal{K}^1_{2,\ell,2}\left(\lambda,\zeta_0{\frac{L_{4,2\ell   }}{\mu   \det{\fb{L}}}}\right)(h_N,\nabla h_N)+\mathcal{K}^1_{3,\ell,2}\left(\lambda,\zeta_1{\frac{L_{4,2\ell   }}{\mu   \det{\fb{L}}}}\right)(h_N,\nabla h_N)\right\}, \\ \label{eq:rwP}
\theta
&= \sum_{j=1}^{N-1}\sum_{\ell=1}^2 \mathscr{F}_{\xi'}^{-1}\left[ 
	\zeta_0{(-1)^{\ell-1}\mu\frac{(B+A)}{A}\frac{i\xi_j}{A}\frac{L_{1,2\ell-1}}{\det{\fb{L}}}} A 
	e^{-Ad_\ell(x_N)} \widehat{h}_j(\xi',\delta) \right] \\ \notag
&+ \sum_{j=1}^{N-1}\sum_{\ell=1}^2 \mathscr{F}_{\xi'}^{-1}\left[ 
	\zeta_1{(-1)^{\ell-1}\mu\frac{(B+A)}{A}\frac{i\xi_j}{A}\frac{L_{1,2\ell-1}}{\det{\fb{L}}}} 
	e^{-Ad_\ell(x_N)} \widehat{h}_j(\xi',\delta) \right] \\ \notag
&- \sum_{j=1}^{N-1}\sum_{\ell=1}^2 \mathscr{F}_{\xi'}^{-1}\left[ 
	\zeta_0{(-1)^{\ell-1}\mu\frac{(B+A)}{A}\frac{i\xi_j}{A}\frac{L_{2,2\ell-1}}{\det{\fb{L}}}} A
	e^{-Ad_\ell(x_N)} \widehat{h}_j(\xi',0) \right] \\ \notag
&- \sum_{j=1}^{N-1}\sum_{\ell=1}^2 \mathscr{F}_{\xi'}^{-1}\left[ 
	\zeta_1{(-1)^{\ell-1}\mu\frac{(B+A)}{A}\frac{i\xi_j}{A}\frac{L_{2,2\ell-1}}{\det{\fb{L}}}} 
	e^{-Ad_\ell(x_N)} \widehat{h}_j(\xi',0) \right] \\ \notag
&+ \sum_{\ell=1}^2 \mathscr{F}_{\xi'}^{-1}\left[ 
	\zeta_0{(-1)^{\ell-1}\mu\frac{(B+A)}{A}\frac{L_{3,2\ell-1}}{\det{\fb{L}}}} A
	e^{-Ad_\ell(x_N)} \widehat{h}_N(\xi',\delta) \right] \\ \notag
&+ \sum_{\ell=1}^2 \mathscr{F}_{\xi'}^{-1}\left[ 
	\zeta_1{(-1)^{\ell-1}\mu\frac{(B+A)}{A}\frac{L_{3,2\ell-1}}{\det{\fb{L}}}} 
	e^{-Ad_\ell(x_N)} \widehat{h}_N(\xi',\delta) \right] \\ \notag
&- \sum_{\ell=1}^2 \mathscr{F}_{\xi'}^{-1}\left[ 
	\zeta_0{(-1)^{\ell-1}\mu\frac{(B+A)}{A}\frac{L_{4,2\ell-1}}{\det{\fb{L}}}} A
	e^{-Ad_\ell(x_N)} \widehat{h}_N(\xi',0) \right] \\ \notag 
&- \sum_{\ell=1}^2 \mathscr{F}_{\xi'}^{-1}\left[ 
	\zeta_1{(-1)^{\ell-1}\mu\frac{(B+A)}{A}\frac{L_{4,2\ell-1}}{\det{\fb{L}}}} 
	e^{-Ad_\ell(x_N)} \widehat{h}_N(\xi',0) \right] \\ \notag 
&{=} \sum_{j=1}^{N-1}\sum_{\ell=1}^2 (-1)^{\ell-1}\mu 
	\mathcal{K}^2_{2,\ell,1}\left(\lambda,\zeta_0{\frac{(B+A)}{A}\frac{i\xi_j}{A}\frac{L_{1,2\ell-1}}{\det{\fb{L}}}}\right)
	(M_{\lambda^{-\frac12}}\lambda^{\frac12}{h_j},\nabla h_j) \\ \notag 
&{+} \sum_{j=1}^{N-1}\sum_{\ell=1}^2 (-1)^{\ell-1}\mu
	\mathcal{K}^2_{3,\ell,1}\left(\lambda,\zeta_1{\frac{(B+A)}{A}\frac{i\xi_j}{A}\frac{L_{1,2\ell-1}}{\det{\fb{L}}}}\right)
	(M_{\lambda^{-\frac12}}\lambda^{\frac12}{h_j},\nabla h_j) \\ \notag
&-\sum_{j=1}^{N-1}\sum_{\ell=1}^2 (-1)^{\ell-1}\mu
\mathcal{K}^2_{2,\ell,2}\left(\lambda,\zeta_0{\frac{(B+A)}{A}\frac{i\xi_j}{A}\frac{L_{2,2\ell-1}}{\det{\fb{L}}}}\right)
(M_{\lambda^{-\frac12}}\lambda^{\frac12}{h_j},\nabla h_j) \\ \notag 
&-\sum_{j=1}^{N-1}\sum_{\ell=1}^2 (-1)^{\ell-1}\mu
\mathcal{K}^2_{3,\ell,2}\left(\lambda,\zeta_1{\frac{(B+A)}{A}\frac{i\xi_j}{A}\frac{L_{2,2\ell-1}}{\det{\fb{L}}}}\right)
(M_{\lambda^{-\frac12}}\lambda^{\frac12}{h_j},\nabla h_j) \\ \notag
&+ \sum_{\ell=1}^2 (-1)^{\ell-1}\mu
\mathcal{K}^2_{2,\ell,1}\left(\lambda,\zeta_0{\frac{(B+A)}{A}\frac{L_{3,2\ell-1}}{\det{\fb{L}}}}\right)(h_N,\nabla h_N) \\ \notag
&+ \sum_{\ell=1}^2 (-1)^{\ell-1}\mu
\mathcal{K}^2_{3,\ell,1}\left(\lambda,\zeta_1{\frac{(B+A)}{A}\frac{L_{3,2\ell-1}}{\det{\fb{L}}}}\right)(h_N,\nabla h_N) \\  \notag
&- \sum_{\ell=1}^2 (-1)^{\ell-1}\mu
\mathcal{K}^2_{2,\ell,2}\left(\lambda,\zeta_0{\frac{(B+A)}{A}\frac{L_{4,2\ell-1}}{\det{\fb{L}}}}\right)(h_N,\nabla h_N) \\ \notag
&- \sum_{\ell=1}^2 (-1)^{\ell-1}\mu
\mathcal{K}^2_{3,\ell,2}\left(\lambda,\zeta_1{\frac{(B+A)}{A}\frac{L_{4,2\ell-1}}{\det{\fb{L}}}}\right)(h_N,\nabla h_N), 
\end{align}
where $\mathcal{K}_{j,\ell_1,\ell_2}^{i_1}$ and $M_{\lambda^{-1/2}}$ are defined in Lemma \ref{Lem:Rbdd} and \eqref{eq:defMltn}, respectively. 
Then we define the operators $\mathcal{S}_N(\lambda)$ and $\mathcal{T}(\lambda)$ by 
\begin{align} \label{eq:defSNP}
\begin{aligned}
\mathcal{S}_N(\lambda)(\lambda^{1/2}{\fb{h}}',h_N,\nabla{\fb{h}})=(\text{the right-hand side of \eqref{eq:rwSN}}), \\ 
\mathcal{T}(\lambda)(\lambda^{1/2}{\fb{h}}',h_N,\nabla{\fb{h}})=(\text{the right-hand side of \eqref{eq:rwP}}). 
\end{aligned}
\end{align}

The following lemma concerns the estimates of the symbols of the solution formula.  
%%%%%%%%%%%%%%%%%%%%%%%%%%%%%%%%%%%%%%%%%%%%%%%%%%%%%%%%%
% Lemma Estimate of symbols
%%%%%%%%%%%%%%%%%%%%%%%%%%%%%%%%%%%%%%%%%%%%%%%%%%%%%%%%%
\begin{Lem} \label{Lem:EstiSym}
Let $0<\varepsilon<\pi/2$ and $\gamma_0>0$, 
and also let $\mathbb{M}_{s,k,\varepsilon,\gamma_0}$ and $A^{-1}{\mathbb{M}}_{0,2,\varepsilon,\gamma_0}$ 
be given by \eqref{eq:DefbbM} and \eqref{eq:DefAIM}, respectively. 
\begin{itemize}
\item[a)] We have $\zeta_0(\xi') \in \mathbb{M}_{0,2,\varepsilon,0}$. (And so, $\zeta_1(\xi')/A=1-\zeta_0(\xi')\in{\mathbb{M}_{0,2,\varepsilon,0}}$.) 
\item[b)] For $j=0,1$, $\zeta_j(\xi')/\det{\fb{L}}\in \mathbb{M}_{-6,2,\varepsilon,\gamma_0}$. 
\item[c)] For $k=1,2,3,4$ and $\ell=1,2$, 
\begin{align}
L_{k,2\ell-1} \in \mathbb{M}_{5,2,\varepsilon,0}, \quad L_{k,2\ell} \in \mathbb{M}_{4,2,\varepsilon,0}. 
\end{align}
\item[d)] For $j=0,1$, $j'=1,\cdots,N-1$, $k=1,2$, $\tilde{k}=3,4$ and $\ell=1,2$, there hold 
\begin{align}
\zeta_j \frac{i\xi_{j'}}{A}\frac{L_{k,2\ell-1}}{\mu B\det{\fb{L}}}, \ \zeta_j \frac{i\xi_{j'}}{A}\frac{L_{k,2\ell}}{\mu\det{\fb{L}}}, \ 
\zeta_j \frac{L_{\tilde{k},2\ell-1}}{\mu B\det{\fb{L}}}, \ \zeta_j \frac{L_{\tilde{k},2\ell}}{\mu\det{\fb{L}}} 
\in \mathbb{M}_{-2,2,\varepsilon,\gamma_0}, \\ 
\zeta_j \frac{(B+A)}{A}\frac{i\xi_{j'}}{A}\frac{L_{k,2\ell-1}}{\det{\fb{L}}}, \ 
\zeta_j \frac{(B+A)}{A}\frac{L_{\tilde{k},2\ell-1}}{\det{\fb{L}}} 
\in A^{-1}{\mathbb{M}_{0,2,\varepsilon,\gamma_0}}. 
\end{align}
\end{itemize}
\end{Lem}

%%%%%%%%%%%%%%%%%%%%%%%%%%%%%%%%%%%%% proof
\begin{proof}
a) For any multi-index $\alpha' \in \mathbb{N}_0^{N-1}$, 
if $|\alpha'|\geq1$, $\partial_{\xi'}^{\alpha'}\zeta_0$ is $C^\infty$ function 
whose support is in the annulus $\{1\leq|\xi'|\leq2\}$. 
As for $\alpha'=0$, $\zeta_0(\xi') \in [0,1]$ for $\xi'\in{\mathbb{R}}^{N-1}$. 
These and $(\tau{\partial}_\tau)\zeta_0(\xi')=0$ imply
\begin{align}
|(\tau{\partial}_\tau)^\ell \partial_{\xi'}^{\alpha'} \zeta_0(\xi')| \leq CA^{-|\alpha'|} 
\quad (\forall \ell=0,1, \ \alpha'\in{\mathbb{N}}_0^{N-1}). 
\end{align}
b) Note that 
\begin{align}
1+\frac{1}{A}\leq \left\{
\begin{aligned}
2,	&\quad \text{on } \text{supp}\,\zeta_0\subset \{|\xi'|\geq1\}, \\
3/A,	&\quad \text{on } \text{supp}\,\zeta_1\subset \{|\xi'|\leq2\}. 
\end{aligned}
\right.
\end{align}
Then, by the Leibniz rule, Lemma \ref{Lem:MultiplierClass}, \eqref{esti:multiclassAB}, a), and Lemma \ref{Lem:EstidetL}, we have
\begin{align*}
\left|(\tau{\partial}_\tau)^\ell \partial_{\xi'}^{\alpha'} \frac{\zeta_j(\xi')}{\det{\fb{L}}}\right| 
&\leq C \sum_{\beta'\leq\alpha'} \left|\partial_{\xi'}^{\beta'}\left(A^j\frac{\zeta_j(\xi')}{A^j}\right)\right| 
	\left| (\tau{\partial}_\tau)^\ell \partial_{\xi'}^{\alpha'-\beta'} \frac{1}{\det{\fb{L}}}\right| \\
&\leq C \fb{1}_{\text{supp}\,\zeta_j} \sum_{\beta'\leq\alpha'} A^{j-|\beta'|} (|\lambda|^{1/2}+A)^{-6}\left(1+\frac{1}{A}\right)|A|^{-(|\alpha'|-|\beta'|)} \\ 
&\leq C(|\lambda|^{1/2}+A)^{-6}A^{-|\alpha'|}, 
\end{align*}
where $\fb{1}_{\text{supp}\,\zeta_j}$ is the charactristic function. 
c) is implied by \eqref{eq:cofacL}, Lemma \ref{Lem:MultiplierClass} and Lemma \ref{Lem:EstiSymbol}, 
and d) is obtained by Lemma \ref{Lem:MultiplierClass}, b), c) and \eqref{esti:multiclassAB}. 
\end{proof}

%%%%%%%%%%%%%%%%%%%%%%%%%%%%%%%%%%%%%%%%%%%%%%%%%%%%%%%%%
% Proof of Main Thm: U_N, P part
%%%%%%%%%%%%%%%%%%%%%%%%%%%%%%%%%%%%%%%%%%%%%%%%%%%%%%%%%
Now we prove the assertions for $\mathcal{S}_N(\lambda)$ and $\mathcal{T}(\lambda)$. 
\begin{proof}[Proof of the assertions for $\mathcal{S}_N(\lambda)$ and $\mathcal{T}(\lambda)$ in Theorem \ref{Thm:Rbdddh}]
By Lemma \ref{Lem:Rbdd}, Lemma \ref{Lem:EstiSym} d), \eqref{eq:RbdMl} and Lemma \ref{Lem:Rbddprop}, 
$(\mathcal{S}(\lambda),\mathcal{T}(\lambda))$ satisfies the $\mathscr{R}$-boundedness properties \eqref{eq:Rbdddh}, 
which completes the proof. 
\end{proof}

%%%%%%%%%%%%%%%%%%%%%%%%%%%%%%%%%%%%%%%%%%%%%%%%%%%%%%%%%%%%%%%
%%%%%%%%%%%%%%%%%%%%%%%%%%%%%%%%%%%%%%%%%%%%%%%%%%%%%%%%%%%%%%%
% Proof of Rbdd with only boundary data for u_j 
%%%%%%%%%%%%%%%%%%%%%%%%%%%%%%%%%%%%%%%%%%%%%%%%%%%%%%%%%%%%%%%
%%%%%%%%%%%%%%%%%%%%%%%%%%%%%%%%%%%%%%%%%%%%%%%%%%%%%%%%%%%%%%%

It remains to construct the solution operator $\mathcal{S}_j(\lambda)$ ($1\leq j\leq N-1$) 
and to prove the $\mathscr{R}$-boundedness for it. 
We first reduce the equation \eqref{eq:PDEuj} to the case where the data are only on boundary. 
We consider the equation 
\begin{align} \label{eq:Harmw}
\lambda u_{1j} - \mu \Delta u_{1j} = \tilde{f} \text{ in } \mathbb{R}^N. 
\end{align}
Here,  
\begin{align} \label{eq:Defstf} % Def of scalar tilde f 
\tilde{f}=-E_0\partial_j \theta, 
\end{align}
where $E_0$ is an extension operator defined by \eqref{eq:DefE0}.
Then, if we set $u_j=u_{1j}+u_{2j}$, we have 
\begin{align} \label{eq:Harmla} 
\left\{ 
\begin{array}{ll} 
\lambda u_{2j} - \mu \Delta u_{2j} = 0 & \text{ in } \Omega, \\ 
\partial_N{u}_{2j} = \tilde{h} & \text{on } \partial\Omega 
\end{array} 
\right. 
\end{align} 
with 
\begin{align} \label{eq:Defsth}
\tilde{h}=\mu^{-1}{\nu_N}{h}_j-\partial_j{u}_N-\partial_N{u}_{1j}. 
\end{align}
We obtain the $\mathscr{R}$-boundedness of the solution operator families of \eqref{eq:Harmw} 
by the following lemma, see for instance \cite[Lemma 2.6]{Sai15MMAS}. 

%%%%%%%%%%%%%%%%%%%%%%%%%%%%%%%%%%%%%%%%%%%%%%%%%%%%%%%%%
% Lem: RbddHarmw
%%%%%%%%%%%%%%%%%%%%%%%%%%%%%%%%%%%%%%%%%%%%%%%%%%%%%%%%%
\begin{Lem} 
\label{Lem:RbddHarmw} 
For all $\lambda \in \mathbb{C}\setminus(-\infty,0]$, 
there exists an operator $H_1(\lambda)$ satisfying $H_1(\lambda)\in{\mathcal{L}}(L_q(\mathbb{R}^N),W^2_q(\mathbb{R}^N))$ for $1<q<\infty$ 
such that the following assertions hold: 
\begin{itemize}
\item[a)] 
For any $1<q<\infty$, $\lambda \in \mathbb{C}\setminus(-\infty,0]$ and $\tilde{f} \in L_q(\mathbb{R}^N)$, 
$u_{1j}=H_1(\lambda)\tilde{f} \in W^2_q(\mathbb{R}^N)$ is a unique solution of \eqref{eq:Harmw}. 
\item[b)]  
For any $1<q<\infty$, $0<\varepsilon<\pi/2$, $\ell=0,1$ and $1\leq m,n\leq N$, there hold 
\begin{align}
\mathscr{R}_{\mathcal{L}(L_q(\Omega))}(\{ (\tau{\partial}_\tau)^\ell \lambda H_1(\lambda) \mid \lambda\in{\Sigma_{\varepsilon,0}} \})&\leq C_{N,q,\varepsilon,\mu}, \\ 
\mathscr{R}_{\mathcal{L}(L_q(\Omega))}(\{ (\tau{\partial}_\tau)^\ell \gamma H_1(\lambda) \mid \lambda\in{\Sigma_{\varepsilon,0}} \})&\leq C_{N,q,\varepsilon,\mu}, \\ 
\mathscr{R}_{\mathcal{L}(L_q(\Omega))}(\{ (\tau{\partial}_\tau)^\ell \lambda^{1/2} \partial_m H_1(\lambda) \mid \lambda\in{\Sigma_{\varepsilon,0}} \})&\leq C_{N,q,\varepsilon,\mu}, \\ 
\mathscr{R}_{\mathcal{L}(L_q(\Omega))}(\{ (\tau{\partial}_\tau)^\ell \partial_m{\partial}_n H_1(\lambda) \mid \lambda\in{\Sigma_{\varepsilon,0}} \})&\leq C_{N,q,\varepsilon,\mu}, 
\end{align}
where $\Sigma_{\varepsilon,0}$ is given by \eqref{eq:resoset} with $\gamma_0=0$ and $\lambda=\gamma+i\tau$. 
\end{itemize}
\end{Lem}

%%%%%%%%%%%%%%%%%%%%%%%%%%%%%%%%%%%%%%%%%%%%%%%%%%%%%%%%%
% Lem: RbddHarmla
%%%%%%%%%%%%%%%%%%%%%%%%%%%%%%%%%%%%%%%%%%%%%%%%%%%%%%%%%	
As for \eqref{eq:Harmla}, we show the following lemma 
to get the $\mathscr{R}$-boundedness of the solution operator families. 
\begin{Lem} \label{Lem:RbddHarmla}
For all $\lambda \in \mathbb{C}\setminus(-\infty,0]$, there exists an operator $H_2(\lambda)$ 
(precisely, it is given by \eqref{eq:uj2solform}) 
satisfying $H_2(\lambda)\in{\mathcal{L}}(L_q(\Omega)^{N+1}, W^2_q(\Omega))$ for $1<q<\infty$ 
such that the following assertions hold: 
\begin{itemize}
\item[a)] 
For any $1<q<\infty$, $\lambda \in \mathbb{C}\setminus(-\infty,0]$ and $\tilde{h} \in W^1_q(\Omega)$, 
$u_{2j}=H_2(\lambda)(\lambda^{1/2}\tilde{h}, \nabla \tilde{h}) \in W^2_q(\Omega)$ 
is a unique solution of \eqref{eq:Harmla}. 
\item[b)]  
For any $1<q<\infty$, $0<\varepsilon<\pi/2$, $\gamma_0>0$, $\ell=0,1$ and $1\leq m,n\leq N$, there hold 
\begin{align}
\mathscr{R}_{\mathcal{L}(L_q(\Omega))}(\{ (\tau{\partial}_\tau)^\ell \lambda H_2(\lambda) \mid \lambda\in{\Sigma_{\varepsilon,\gamma_0}} \})&\leq C_{N,q,\varepsilon,\gamma_0,\mu,\delta}, \\ 
\mathscr{R}_{\mathcal{L}(L_q(\Omega))}(\{ (\tau{\partial}_\tau)^\ell \gamma H_2(\lambda) \mid \lambda\in{\Sigma_{\varepsilon,\gamma_0}} \})&\leq C_{N,q,\varepsilon,\gamma_0,\mu,\delta}, \\ 
\mathscr{R}_{\mathcal{L}(L_q(\Omega))}(\{ (\tau{\partial}_\tau)^\ell \lambda^{1/2} \partial_m H_2(\lambda) \mid \lambda\in{\Sigma_{\varepsilon,\gamma_0}} \})&\leq C_{N,q,\varepsilon,\gamma_0,\mu,\delta}, \\ 
\mathscr{R}_{\mathcal{L}(L_q(\Omega))}(\{ (\tau{\partial}_\tau)^\ell \partial_m{\partial}_n H_2(\lambda) \mid \lambda\in{\Sigma_{\varepsilon,\gamma_0}} \})&\leq C_{N,q,\varepsilon,\gamma_0,\mu,\delta}, 
\end{align}
where $\Sigma_{\varepsilon,\gamma_0}$ is given by \eqref{eq:resoset} and $\lambda=\gamma+i\tau$. 
\end{itemize}
\end{Lem}

\begin{proof}
Applying the partial Fourier transform with respect to $x'$ to \eqref{eq:Harmla} implies
\begin{align}
\left\{
\begin{array}{ll}
(B^2-\partial_N^2)\widehat{u}_{j2}(\xi',x_N)=0 & 0<x_N<\delta, \\
\partial_N \widehat{u}_{j2}(\xi',x_N)=\mathscr{F}_{x'}(\xi',x_N) & x_N\in\{0,\delta\}. 
\end{array}
\right.
\end{align}
Thus, we set 
\begin{align}
\widehat{u}_{j2}(\xi',x_N)=\sum_{\ell=1}^2 \beta_\ell e^{-Bd_\ell(x_N)}
\end{align}
with some coefficients $\beta_1$ and $\beta_2$ which depend on $(\lambda,\xi')$ and obey
\begin{align} 
B
\begin{bmatrix}
1&-e^{-B\delta }\\ e^{-B\delta }&-1
\end{bmatrix}
\begin{bmatrix}
\beta_1 \\ \beta_2
\end{bmatrix}
=
\begin{bmatrix}
\mathscr{F}_{x'}(\xi',\delta) \\ \mathscr{F}_{x'}(\xi',0) 
\end{bmatrix}. 
\end{align}
By solving this, we obtain the solution formula for \eqref{eq:Harmla}: 
\begin{align} \label{eq:uj2sf}
\begin{aligned}
u_{2j}(x) 
&= \mathscr{F}_{\xi'}^{-1}\bigg[ \frac{1}{B^2(1-e^{-2 B\delta{}})}\big\{ 
(e^{-Bd_1(x_N)}+e^{-B\delta } e^{-Bd_2(x_N)})B{\mathscr{F}_{x'}}(\xi',\delta) \\ 
&\hspace{15mm} - (e^{-B\delta } e^{-Bd_1(x_N)}+e^{-Bd_2(x_N)})B{\mathscr{F}_{x'}}(\xi',0) \big\}\bigg](x') \\ 
&= \mathcal{K}^1_{1,1,1}\left(\lambda,\frac{1}{B^2(1-e^{-2 B\delta{}})}\right)(\lambda^{\frac12}\tilde h,\nabla \tilde h)
+\mathcal{K}^1_{1,2,1}\left(\lambda,\frac{e^{-B\delta }}{B^2(1-e^{-2 B\delta{}})}\right)(\lambda^{\frac12}\tilde h,\nabla \tilde h) \\ 
&-\mathcal{K}^1_{1,1,2}\left(\lambda,\frac{e^{-B\delta }}{B^2(1-e^{-2 B\delta{}})}\right)(\lambda^{\frac12}\tilde h,\nabla \tilde h)
-\mathcal{K}^1_{1,2,2}\left(\lambda,\frac{1}{B^2(1-e^{-2 B\delta{}})}\right)(\lambda^{\frac12}\tilde h,\nabla \tilde h). 
\end{aligned}
\end{align}
We thus define the operator $H_2(\lambda)$ by 
\begin{align} \label{eq:uj2solform}
H_2(\lambda)(\lambda^{\frac12}\tilde h,\nabla \tilde h)=(\text{the right-hand side of \eqref{eq:uj2sf}}). 
\end{align}
Since \eqref{esti:equivB} and \eqref{esti:multiclasski} 
imply the assumptions \eqref{eq:LemIMC0} and \eqref{eq:LemIMC1} of Lemma \ref{Lem:InvMultiClass} 
with $m(\lambda,\xi')=1-e^{-2 B\delta{}}$ and $f\equiv 1-e^{-2c\gamma_0^{1/2}\delta}$ (constant function): 
\begin{align}
|1-e^{-2 B\delta{}}|\geq 1-e^{-2c\gamma_0^{1/2}\delta}, \quad 1-e^{-2 B\delta{}} \in \mathbb{M}_{0,1,\varepsilon,0}, 
\end{align}
we get $(1-e^{-2 B\delta{}})^{-1} \in \mathbb{M}_{0,1,\varepsilon,\gamma_0}$. 
And so, we have 
\begin{align}
\frac{1}{B^2(1-e^{-2 B\delta{}})}, \ \frac{e^{-B\delta }}{B^2(1-e^{-2 B\delta{}})} \in \mathbb{M}_{-2,1,\varepsilon,\gamma_0} 
\end{align}
from Lemma \ref{Lem:MultiplierClass} and Lemma \ref{Lem:EstiSymbol}. 
Then Lemma \ref{Lem:Rbdd} and Lemma \ref{Lem:Rbddprop} ii) imply the desired conclusion. 
\end{proof}

%%%%%%%%%%%%%%%%%%%%%%%%%%%%%%%%%%%%%%%%%%%%%%%%%%%%%%%%%
% Proof of Main Thm: U_j part
%%%%%%%%%%%%%%%%%%%%%%%%%%%%%%%%%%%%%%%%%%%%%%%%%%%%%%%%%
Let us close the paper with completion of proof of Theorem \ref{Thm:Rbdddh}. 

\begin{proof}[Proof of the remaining assertions of Theorem \ref{Thm:Rbdddh}]
In view of the arguments above, we define $\mathcal{S}_j(\lambda)$ by 
\begin{align} \label{eq:defSj}
\begin{aligned}
&\mathcal{S}_j(\lambda)(\lambda^{1/2}\fb{h}', h_N, \nabla \fb{h}) \\
&=H_1(\lambda)Q_0(\lambda)(\lambda^{1/2}\fb{h}', h_N, \nabla \fb{h}) 
+H_2(\lambda)\big(Q_1(\lambda), Q_2(\lambda)\big)(\lambda^{1/2}\fb{h}', h_N, \nabla \fb{h}), 
\end{aligned}
\end{align}
where we have set 
\begin{align*}
Q_0(\lambda)(\lambda^{1/2}\fb{h}', h_N, \nabla \fb{h})
&=\tilde{f} \\ 
&=-E_0\partial_j{\mathcal{T}}(\lambda)(\lambda^{1/2}\fb{h}', h_N, \nabla \fb{h}), \\ 
Q_1(\lambda)(\lambda^{1/2}\fb{h}', h_N, \nabla \fb{h})
&=\lambda^{1/2}\tilde{h} \\ 
&=\mu^{-1}{\tilde{\nu}_N}\lambda^\frac{1}{2}{h}_j
-\lambda^\frac{1}{2}{\partial}_j{\mathcal{S}}_N(\lambda)(\lambda^{1/2}\fb{h}', h_N, \nabla \fb{h}) \\
&-\lambda^\frac{1}{2}{\partial}_NH_1(\lambda)Q_0(\lambda)(\lambda^{1/2}\fb{h}', h_N, \nabla \fb{h}), \\ 
Q_2(\lambda)(\lambda^{1/2}\fb{h}', h_N, \nabla \fb{h})
&=\nabla\tilde{h} \\ 
&=\mu^{-1}(\nabla{h}_j)\tilde{\nu}_N
+\mu^{-1}M_{\lambda^{-1/2}}\lambda^{1/2}{h}_j\nabla{\tilde{\nu}_N} \\ 
&-\partial_j\nabla{\mathcal{S}}_N(\lambda)(\lambda^{1/2}\fb{h}', h_N, \nabla \fb{h}) \\ 
&-\partial_N\nabla H_1(\lambda)Q_0(\lambda)(\lambda^{1/2}\fb{h}', h_N, \nabla \fb{h}). 
\end{align*}
Here, $\tilde{f}$ and $\tilde{h}$ are given by \eqref{eq:Defstf} and \eqref{eq:Defsth}, respectively, 
$E_0$ is an extension operator defined by \eqref{eq:DefE0}, 
$\tilde{\nu}_N$ is the $N$-th component of $\tilde{\nu}$, 
and $M_{\lambda^{-1/2}}$ and $\tilde{\nu}$ are defined by \eqref{eq:defMltn}. 
We have 
\begin{align} \label{eq:RbdQi}
\mathscr{R}_{\mathcal{L}(L_q(\Omega))}(\{ Q_i(\lambda) \mid \lambda\in{\Sigma_{\varepsilon,\gamma_0}} \})
\leq C_{N,q,\varepsilon,\gamma_0,\mu,\delta}
\end{align}
for $i=0,1,2$ from Lemma \ref{Lem:Rbddprop}, Lemma \ref{Lem:RbddHarmw}, \eqref{eq:RbdMl} 
and the fact that $\mathcal{S}_N(\lambda)$ and $\mathcal{T}(\lambda)$ satisfy \eqref{eq:Rbdddh}.  
Then $\mathcal{S}_j(\lambda)$ satisfies \eqref{eq:Rbdddh} 
by Lemma \ref{Lem:Rbddprop}, Lemma \ref{Lem:RbddHarmw}, Lemma \ref{Lem:RbddHarmla} and \eqref{eq:RbdQi}. 
Since $(\mathcal{S}(\lambda),\mathcal{T}(\lambda))$ satisfies a) in Theorem \ref{Thm:Rbdddh}, 
by summing up the aforementioned arguments, Theorem \ref{Thm:Rbdddh} follows. 
\end{proof}

%%%%%%%%%%%%%%%%%%%%%%%%%%%%%%%%%%%%%%%%%%%%%%%%%%%%%%%%%
%%%%%%%%%%%%%%%%%%%%%%%%%%%%%%%%%%%%%%%%%%%%%%%%%%%%%%%%%
% Acknowledgement
%%%%%%%%%%%%%%%%%%%%%%%%%%%%%%%%%%%%%%%%%%%%%%%%%%%%%%%%%
%%%%%%%%%%%%%%%%%%%%%%%%%%%%%%%%%%%%%%%%%%%%%%%%%%%%%%%%%
\fb{Acknowledgement.}
I am very grateful to Professor Yoshihiro Shibata for bringing the problem studied in this paper to my attention
and for many useful suggestions. 
I also thank Professor Hirokazu Saito for many stimulating discussions.


\begin{thebibliography}{99}
\bibitem{Abe04MMAS}
Abe T. On a resolvent estimate of the Stokes equation with Neumann-Dirichlet-type boundary condition on an infinite layer. 
Mathematical Methods in the Applied Sciences, 2004; {\bf 27}:1007--1048.
\bibitem{Abel06MN}
Abels H. Generalized Stokes resolvent equations in an infinite layer with mixed boundary conditions. 
Mathematische Nachrichten, 2006; {\bf 279}(4):351--367.
\bibitem{Abel05MFM1}
Abels H. Reduced and generalized Stokes resolvent equations in asymptotically flat layers, part I: Unique solvability. 
Journal of Mathematical Fluid Mechanics, 2005; {\bf 7}:201--222.
\bibitem{Abel05MFM2}
Abels H. Reduced and generalized Stokes resolvent equations in asymptotically flat layers, part II: $H_\infty$-calculus. 
Journal of Mathematical Fluid Mechanics, 2005; {\bf 7}:223--260.
\bibitem{Abel05ADE}
Abels H. The initial-value problem for the Navier-Stokes equations with a free surface in $L^q$-Sobolev spaces. 
Advances in Differential Equations, 2005; {\bf 10}(1):45--64.
\bibitem{Sai15MMAS}
Saito H. On the $\mathcal{R}$-boundedness of solution operator families of the generalized Stokes resolvent problem in an infinite layer. 
Mathematical Methods in the Applied Sciences, 2015; {\bf 38}:1888--1925.
\bibitem{AbeShiba03MSJ}
Abe T, Shibata Y. On a resolvent estimate of the Stokes equation on an infinite layer. 
Journal of the Mathematical Society of {J}apan, 2003; {\bf 55}(2):469--497.
\bibitem{AbeShiba03MFM}
Abe T, Shibata Y. On a resolvent estimate of the Stokes equation on an infinite layer part 2, $\lambda=0$ case. 
Journal of Mathematical Fluid Mechanics, 2003; {\bf 5}:245--274.
\bibitem{AbelWie05DIE}
Abels H, Wiegner M. Resolvent estimates for the Stokes operator on an infinite layer. 
Differential and Integral Equations, 2005; {\bf 18}(10):1081--1110.
\bibitem{AbeYama10MFM}
Abe T, Yamazaki M. On a stationary problem of the Stokes equation in an infinite layer in Sobolev and Besov spaces. 
Journal of Mathematical Fluid Mechanics, 2010; {\bf 12}:61--100.
\bibitem{BelBol17MN}
von Below L, Bolkart M. The Stokes equations in layer domains on spaces of bounded and integrable functions. 
Mathematische Nachrichten, 2017; {\bf 290}(10):1553--1587.
\bibitem{Abel02EE}
Abels H. Boundedness of imaginary powers of the Stokes operator in an infinite layer. 
Journal of Evolution Equations, 2002; {\bf 2}:439--457.
\bibitem{Shiba13MFM}
Shibata Y. Generalized resolvent estimates of the Stokes equations with first order boundary condition in a general domain. 
Journal of Mathematical Fluid Mechanics, 2013; {\bf 15}:1--40.
\bibitem{Shiba14DIE}
Shibata Y. On the $\mathcal{R}$-boundedness of solution operators for the Stokes equations with free boundary condition. 
Differential and Integral Equations, 2014; {\bf 27}(3--4):313--368.
\bibitem{SimZie98QM}
Simader CG, Ziegler T. The weak Dirichlet and Neumann problems in $L^q$ for the Laplacian and the Helmholtz decomposition in infinite cylinders and layers. 
Recent developments in partial differential equations, quaderni di matematica, 1998; {\bf 2}:39--161.
\bibitem{Wei01MA}
Weis L. Operator-valued multiplier theorems and maximal $L_p$-regularity. 
Mathematische Annalen, 2001; {\bf 319}:735--758.
\bibitem{ShibaShimi12MSJ}
Shibata Y, Shimizu S. On the maximal $L_p$-$L_q$ regularity of the Stokes problem with first order boundary condition. 
Journal of the Mathematical Society of Japan, 2012; {\bf 64}(2):561--626.
\bibitem{Bea81CPAM}
Beale JT. The initial value problem for the Navier-Stokes equations with a free surface.
Communications on Pure and Applied Mathematics, 1981; {\bf 34}:359--392.
\bibitem{All85AFST}
Allain G. Un probl\`{e}me de Navier-Stokes avec surface libre et tension superficielle. 
Annales de la Facult{\`{e}} des Sciences de Toulouse. S{\`{e}}rie, 1985; {\bf 7}:29--56.
\bibitem{All87AMO}
Allain G. Small-time existence for the Navier-Stokes equations with a free surface. 
Applied Mathematics {\&} Optimization, 1987; {\bf 16}:37--50.
\bibitem{Tani96ARMA}
Tani A. Small-time existence for the three-dimensional Navier-Stokes equations for an incompressible fluid with a free surface. 
Archive for Rational Mechanics and Analysis, 1996; {\bf 133}:299--331.
\bibitem{Tera85Hi}
Teramoto Y. The initial value problem for viscous incompressible flow down an inclined plane.
Hiroshima Mathematical Journal, 1985; {\bf 15}:619--643.
\bibitem{Tera92Ky}
Teramoto Y. On the Navier-Stokes flow down an inclined plane. 
Journal of Mathematics of Kyoto University, 1992; {\bf 32}(3):593--619.
\bibitem{Bea84ARMA}
Beale JT. Large-time regularity of viscous surface waves. 
Archive for Rational Mechanics and Analysis, 1984; {\bf 84}(4):307--352.
\bibitem{TaniTana95ARMA}
Tani A, Tanaka N. Large-time existence of surface waves in incompressible viscous fluids with or without surface tension. 
Archive for Rational Mechanics and Analysis, 1995; {\bf 130}:303--314.
\bibitem{BeaNishi85NMS}
Beale JT, Nishida T. Large-time behavior of viscous surface waves. 
Recent Topics in Nonlinear PDE, II, {N}orth-{H}olland Mathematics Studies, Vol. 128, 1985, 1--14. 
\bibitem{HataKawa09NA}
Hataya Y, Kawashima S. Decaying solution of the Navier-Stokes flow of infinite volume without surface tension. 
Nonlinear Analysis, 2009; {\bf 71}:e2535–e2539.
\bibitem{Hata09Ky}
Hataya Y. Decaying solution of a Navier-Stokes flow without surface tension. 
Journal of Mathematics of Kyoto University, 2009; {\bf 49}(4):691--717.
\bibitem{Sai18DE}
Saito H. Global solvability of the Navier-Stokes equations with a free surface in the maximal $L_p$-$L_q$ regularity class. 
Journal of Differential Equations, 2018; {\bf 264}:1475--1520.
\bibitem{Shiba15DE}
Shibata Y. On some free boundary problem of the Navier-Stokes equations in the maximal $L_p$-$L_q$ regularity class. 
Journal of Differential Equations, 2015; {\bf 258}:4127--4155.
\bibitem{ShibaShimi07DIE}
Shibata Y, Shimizu S. On a free boundary problem for the Navier-Stokes equations. 
Differential and Integral Equations, 2007; {\bf 20}(3):241--276.
\bibitem{Shiba16MFDPF}
Shibata Y. On the $\mathscr{R}$-bounded solution operator and the maximal $L_p$-$L_q$ regularity of the Stokes equations with free boundary condition. 
Mathematical Fluid Dynamics, Present and Future, Springer Proceedings in Mathematics {\&} Statistics
Vol. 183, Tokyo, 2016, 203--285.
\bibitem{Shiba16DCDS}
Shibata Y. Local well-posedness of free surface problems for the Navier-Stokes equations in a general domain. 
Discrete and Continuous Dynamical Systems - Series {S}, 2016; {\bf 9}(1):315--342.
\bibitem{HytNeeVerWei16EMG}
Hyt\"{o}nen T, van Neerven J, Veraar M, Weis L. Analysis in Banach spaces, volume I: 
Martingales and Littlewood-Paley Theory of Ergebnisse der Mathematik und ihrer Grenzgebiete. 3. Folge
/ A Series of Modern Surveys in Mathematics. Springer, Cham, 2016.
\bibitem{GruSol91MS}
Grubb G, Solonnikov VA. Boundary value problems for the nonstationary Navier-Stokes equations treated by pseudo-differential methods. 
Mathematica Scandinavica, 1991; {\bf 69}:217--290.
\bibitem{ShibaShimi08RAM}
Shibata Y, Shimizu S. On the $L_p$-$L_q$ maximal regularity of the Neumann problem for the Stokes equations in a bounded domain. 
Journal F{\"{u}}r Die Reine Und Angewandte Mathematik, 2008; {\bf 615}:157--209.
\bibitem{DenHiePru03MAMS}
Denk R, Hieber M, Pr\"{u}ss J. $\mathcal{R}$-boundedness, Fourier multipliers and problems of elliptic and parabolic type, 
Memoirs of the American Mathematical Society, 2003; {\bf 166}(788):viii+114 pp.
\bibitem{KunWei04Sp}
Kunstmann PC, Weis L. Maximal $L_p$-regularity for parabolic equations, Fourier multiplier theorems and $H^\infty$-functional calculus, 
Lecture Notes in Mathematics, 2004; {\bf 1855}:65--311. 
\bibitem{ShibaShimi03DIE}
Shibata Y, Shimizu S. On a resolvent estimate for the Stokes system with Neumann boundary condition. 
Differential and Integral Equations, 2003; {\bf 16}(4):385--426.
\bibitem{GirWei03PAM}
Girardi M, Weis L. Criteria for R-boundedness of operator families. 
Lecture Notes in Pure and Applied Mathematics, 2003; {\bf 234}:203--221.
\bibitem{EnoShiba13FE}
Enomoto Y, Shibata Y. On the $\mathscr{R}$-sectoriality and the initial boundary value problem for the viscous compressible fluid flow. 
Funkcialaj Ekvacioj, 2013; {\bf 56}:441--505.
\bibitem{Vol65MS}
Volevich LR. Solubility of boundary value problems for general elliptic systems. 
Matematicheski{\u{i}} Sbornik. (Novaya Seriya), 1965; {\bf 68}(100), No.3:373--416.
\bibitem{BahCheDan11GMW}
Bahouri H, Chemin J.-Y,, Danchin R. Fourier Analysis and Nonlinear Partial Differential Equations, 
volume 343 of Grundlehren der mathematischen Wissenschaften. Springer, January 2011.
\end{thebibliography}
\end{document}